\documentclass{amsart}
\usepackage{amssymb,amsmath,stmaryrd,mathrsfs,amsthm}
\usepackage[all]{xy}
\usepackage[neveradjust]{paralist}
\usepackage{hyperref}
\usepackage{mathtools}
\usepackage{wasysym}
\usepackage{braket}
\let\setof\Set

\makeatletter

\def\mdef#1#2{\expandafter\expandafter\expandafter\gdef\expandafter\expandafter\noexpand#1\expandafter{\expandafter\ensuremath\expandafter{#2}}}

\newcommand{\sU}{\ensuremath{\mathscr{U}}}
\newcommand{\sV}{\ensuremath{\mathscr{V}}}

\newcommand{\cE}{\ensuremath{\mathcal{E}}}
\newcommand{\cM}{\ensuremath{\mathcal{M}}}
\newcommand{\cQ}{\ensuremath{\mathcal{Q}}}
\newcommand{\cT}{\ensuremath{\mathcal{T}}}

\newcommand{\bbV}{\ensuremath{\mathbb{V}}}

\newcommand{\colim}{\ensuremath{\operatorname{colim}}}

\DeclareSymbolFont{bbold}{U}{bbold}{m}{n}
\DeclareSymbolFontAlphabet{\mathbbb}{bbold}
\mdef\bbSet{\mathbbb{Set}}

\newcommand{\bA}{\ensuremath{\mathbf{A}}}
\newcommand{\bB}{\ensuremath{\mathbf{B}}}
\newcommand{\bC}{\ensuremath{\mathbf{C}}}
\newcommand{\bM}{\ensuremath{\mathbf{M}}}
\newcommand{\bS}{\ensuremath{\mathbf{S}}}
\newcommand{\bV}{\ensuremath{\mathbf{V}}}

\newcommand{\fa}{\ensuremath{\mathfrak{a}}}

\newcommand{\Sbar}{\ensuremath{\overline{S}}}
\newcommand{\Ubar}{\ensuremath{\overline{U}}}
\newcommand{\Xbar}{\ensuremath{\overline{X}}}
\newcommand{\Ybar}{\ensuremath{\overline{Y}}}
\newcommand{\rbar}{\ensuremath{\overline{r}}}
\newcommand{\xbar}{\ensuremath{\overline{x}}}
\newcommand{\ybar}{\ensuremath{\overline{y}}}

\newcommand{\inv}{^{-1}}
\newcommand{\meet}{\ensuremath{\wedge}}
\newcommand{\join}{\ensuremath{\vee}}

\newcommand{\op}{\ensuremath{^{\mathit{op}}}}

\SelectTips{cm}{}
\newdir{ >}{{}*!/-20pt/@{>}}

\newcommand{\pullbackcorner}[1][dr]{\save*!/#1-1.2pc/#1:(-1,1)@^{|-}\restore}
\newcommand{\iso}{\cong}

\mdef\eqv{\simeq}

\newcommand{\too}[1][]{\ensuremath{\overset{#1}{\longrightarrow}}}

\newcommand{\imp}{\ensuremath{\Rightarrow}}
\mdef\impp{\,\Longrightarrow\,}
\newcommand{\toto}{\ensuremath{\rightrightarrows}}
\newcommand{\into}{\ensuremath{\hookrightarrow}}
\newcommand{\mono}{\rightarrowtail}
\newcommand{\epi}[1][]{\ensuremath{\overset{#1}{\twoheadrightarrow}}}
\newcommand{\maps}{\colon}

\let\xto\xrightarrow

\def\rightarrowtailfill@{\arrowfill@{\Yright\joinrel\relbar}\relbar\rightarrow}
\newcommand\xrightarrowtail[2][]{\ext@arrow 0055{\rightarrowtailfill@}{#1}{#2}}
\let\xmono\xrightarrowtail

\def\twoheadrightarrowfill@{\arrowfill@{\relbar\joinrel\relbar}\relbar\twoheadrightarrow}
\newcommand\xtwoheadrightarrow[2][]{\ext@arrow 0055{\twoheadrightarrowfill@}{#1}{#2}}
\let\xepi\xtwoheadrightarrow

\def\twoheadleftarrowfill@{\arrowfill@\twoheadleftarrow\relbar{\relbar\joinrel\relbar}}
\newcommand\xtwoheadleftarrow[2][]{\ext@arrow 0055{\twoheadleftarrowfill@}{#1}{#2}}

\newtheorem{thm}{Theorem}[section]
  
\newtheorem{cor}{Corollary}
  \let\c@cor\c@thm
  \numberwithin{cor}{section}
  
\newtheorem{prop}{Proposition}
  \let\c@prop\c@thm
  \numberwithin{prop}{section}
  
\newtheorem{lem}{Lemma}
  \let\c@lem\c@thm
  \numberwithin{lem}{section}
  
\theoremstyle{definition}
\newtheorem{defn}{Definition}
  \let\c@defn\c@thm
  \numberwithin{defn}{section}
  
\newtheorem{defns}{Definitions}
  \let\c@defns\c@thm
  \numberwithin{defns}{section}

  \let\c@notn\c@thm
  \numberwithin{notn}{section}

  \let\c@notns\c@thm
  \numberwithin{notns}{section}
  
\newtheorem{cnv}{Convention}
  \let\c@cnv\c@thm
  \numberwithin{cnv}{section}
  
\theoremstyle{remark}
\newtheorem{rmk}{Remark}
  \let\c@rmk\c@thm
  \numberwithin{rmk}{section}

  \let\c@rmks\c@thm
  \numberwithin{rmks}{section}
  
\newtheorem{eg}{Example}
  \let\c@eg\c@thm
  \numberwithin{eg}{section}

  \let\c@egs\c@thm
  \numberwithin{egs}{section}

\def\thmqedhere{\expandafter\csname\csname @currenvir\endcsname @qed\endcsname}

\let\c@equation\c@thm
\numberwithin{equation}{section}

\newlength\oldleftmargini       
\newlength\oldleftmarginii
\newlength\oldleftmarginiii
\newlength\oldleftmarginiv
\newlength\oldleftmarginv
\newlength\oldleftmarginvi
\newcount\maxenum
\newif\ifkillspacing
\def\@adjust@enum@labelwidth{%
  \advance\@listdepth by 1\relax
  \ifkillspacing                
    \csname c@\@enumctr\endcsname\maxenum
    \settowidth{\@tempdima}{%
      \csname label\@enumctr\endcsname\hspace{\labelsep}}%
    \csname leftmargin\romannumeral\@listdepth\endcsname
      \@tempdima
  \else                         
    \csname fixspacing\romannumeral\@listdepth\endcsname
  \fi
  \advance\@listdepth by -1\relax}
\def\fixspacingi{\ifnum\oldleftmargini=0\setlength\oldleftmargini\leftmargini\else\setlength\leftmargini\oldleftmargini\fi}
\def\fixspacingii{\ifnum\oldleftmarginii=0\setlength\oldleftmarginii\leftmarginii\else\setlength\leftmarginii\oldleftmarginii\fi}
\def\fixspacingiii{\ifnum\oldleftmarginiii=0\setlength\oldleftmarginiii\leftmarginiii\else\setlength\leftmarginiii\oldleftmarginiii\fi}
\def\fixspacingiv{\ifnum\oldleftmarginiv=0\setlength\oldleftmarginiv\leftmarginiv\else\setlength\leftmarginiv\oldleftmarginiv\fi}
\def\fixspacingv{\ifnum\oldleftmarginv=0\setlength\oldleftmarginv\leftmarginv\else\setlength\leftmarginv\oldleftmarginv\fi}
\def\fixspacingvi{\ifnum\oldleftmarginvi=0\setlength\oldleftmarginvi\leftmarginvi\else\setlength\leftmarginvi\oldleftmarginvi\fi}

\def\pl@label#1#2{%
  \edef\pl@the{\noexpand#1{\@enumctr}}%
  \pl@lab\expandafter{\the\pl@lab\csname yourthe\@enumctr\endcsname}%
  \advance\@tempcnta1
  \pl@loop}
\def\@enumlabel@#1[#2]{%
  \@plmylabeltrue
  \@tempcnta0
  \pl@lab{}%
  \let\pl@the\pl@qmark
  \expandafter\pl@loop\@gobble#2\@@@
  \ifnum\@tempcnta=1\else
    \PackageWarning{paralist}{Incorrect label; no or multiple
      counters.\MessageBreak The label is: \@gobble#2}%
  \fi
  \expandafter\edef\csname label\@enumctr\endcsname{\the\pl@lab}%
  \expandafter\edef\csname the\@enumctr\endcsname{\the\pl@lab}%
  \expandafter\let\csname yourthe\@enumctr\endcsname\pl@the
  #1}

\def\alwaysmath#1{\expandafter\expandafter\expandafter\global\expandafter\expandafter\expandafter\let\expandafter\expandafter\csname your@#1\endcsname\csname #1\endcsname
  \expandafter\def\csname #1\endcsname{\ensuremath{\csname your@#1\endcsname}}}
\alwaysmath{alpha}
\alwaysmath{beta}
\alwaysmath{gamma}
\alwaysmath{Gamma}
\alwaysmath{delta}
\alwaysmath{Delta}
\alwaysmath{epsilon}

\alwaysmath{zeta}
\alwaysmath{eta}
\alwaysmath{vartheta}
\alwaysmath{theta}
\alwaysmath{Theta}
\alwaysmath{iota}
\alwaysmath{kappa}
\alwaysmath{lambda}
\alwaysmath{Lambda}
\alwaysmath{mu}
\alwaysmath{nu}
\alwaysmath{xi}
\alwaysmath{pi}
\alwaysmath{rho}
\alwaysmath{sigma}
\alwaysmath{Sigma}
\alwaysmath{tau}
\alwaysmath{upsilon}
\alwaysmath{Upsilon}
\alwaysmath{phi}
\alwaysmath{Pi}
\alwaysmath{Phi}
\newcommand{\ph}{\ensuremath{\varphi}}
\alwaysmath{chi}
\alwaysmath{psi}
\alwaysmath{Psi}
\alwaysmath{omega}
\alwaysmath{Omega}
\alwaysmath{ell}
\alwaysmath{infty}
\alwaysmath{odot}
\alwaysmath{top}
\alwaysmath{bot}
\alwaysmath{neg}
\mdef\na{\nabla}

\let\al\alpha
\let\be\beta
\let\gm\gamma
\let\Gm\Gamma

\let\De\Delta

\let\th\theta

\let\vth\vartheta

\mdef\ff{\Vdash}
\mdef\fff{\Vvdash}
\mdef\ss{\vDash}
\mdef\pp{\mathrel{\,\vdash\,}}
\mdef\vthtil{\widetilde{\vth}}
\mdef\Set{\mathbf{Set}}
\mdef\ffs{\ff_\bS}
\mdef\ffc{\ff_\bC}
\def\nno{\textsc{nno}}
\def\apg{\textsc{apg}}
\def\apgs{\textsc{apg}s}
\let\Iff\Leftrightarrow
\def\qq#1{$\ulcorner${\sf#1}$\urcorner$}

\def\tqq#1{\text{\qq{#1}}}
\alwaysmath{prec}
\alwaysmath{preceq}
\alwaysmath{star}
\mdef\St{\mathcal{S}\mathit{t}}
\mdef\Cat{\mathcal{C}\mathit{at}}
\mdef\iin{\mathrel{\epsilon}}
\mdef\ordiin{{\mathord{\epsilon}}}

\let\yourexists\exists
\def\exists#1.{(\yourexists#1)}
\let\yourforall\forall
\def\forall#1.{(\yourforall#1)}
\let\im\yourexists
\let\coim\yourforall
\alwaysmath{im}
\alwaysmath{coim}

\newcommand{\mm}[1]{\ensuremath{\llbracket#1\rrbracket}}

\DeclareMathOperator\Sub{Sub}

\mdef\bSh{\mathbf{Sh}}

\mdef\phhat{\widehat{\ph}}
\mdef\psihat{\widehat{\psi}}
\mdef\phtil{\widetilde{\ph}}
\mdef\psitil{\widetilde{\psi}}

\def\toiso{\xto{\smash{\raisebox{-.5mm}{$\scriptstyle\sim$}}}}

\mdef\ordin{\mathord{\in}}

\newenvironment{blist}{\begin{list}{\labelitemi}{\leftmargin=1.3em\labelwidth=1em}}{\end{list}}
\def\setcurlabel#1{\def\@currentlabel{#1}}

\newenvironment{labellist}[2]{\begin{list}{#1}{#2}%
    \let\youritem\item%
    \def\item##1{\youritem[##1]\setcurlabel{##1}}}{\end{list}}

\def\csl{\!\sslash\!}
\mdef\ddo{\Delta_0}

\makeatother

\title[Stack semantics]{Stack semantics and the comparison of material
  and structural set theories}
\thanks{The author was supported by a National Science Foundation
  postdoctoral fellowship during the writing of this paper.}
\author{Michael A.\ Shulman}
\email{shulman@math.uchicago.edu}

\begin{document}
\maketitle

\begin{abstract}
  We extend the usual internal logic of a (pre)topos to a more general
  interpretation, called the \emph{stack semantics}, which allows for
  ``unbounded'' quantifiers ranging over the class of objects of the
  topos.  Using well-founded relations inside the stack semantics, we
  can then recover a membership-based (or ``material'') set theory
  from an arbitrary topos, including even set-theoretic axiom schemas
  such as collection and separation which involve unbounded
  quantifiers.  This construction reproduces the models of
  Fourman-Hayashi and of algebraic set theory, when the latter apply.
  It turns out that the axioms of collection and replacement are
  always valid in the stack semantics of any topos, while the axiom of
  separation expressed in the stack semantics gives a new
  topos-theoretic axiom schema with the full strength of ZF.  We call
  a topos satisfying this schema \emph{autological}.
\end{abstract}

\tableofcontents

\section{Introduction}
\label{sec:introduction}

It is well-known that elementary topos theory is roughly equivalent to
a weak form of set theory.  The most common statement is that the
theory of well-pointed topoi with natural numbers objects and
satisfying the axiom of choice (a.k.a.\ ETCS; see~\cite{lawvere:etcs})
is equiconsistent with ``bounded Zermelo'' set theory (BZC).  This is
originally due to Cole~\cite{cole:cat-sets} and
Mitchell~\cite{mitchell:topoi-sets}; see also
Osius~\cite{osius:cat-setth}.  Moreover, the proof is very direct: the
category of sets in any model of BZC is a model of ETCS, while from
any model of ETCS one can construct a model of BZC out of well-founded
relations.  (The same ideas apply to stronger and weaker theories, but
for clarity, in the introduction we will speak mostly about topoi.)

In fact, Lawvere and others have observed that ETCS can itself serve
as a foundation for much mathematics, replacing traditional
membership-based set theoretic foundations such as BZC or ZFC.
Although the term ``set theory'' often refers specifically to
membership-based theories, it is clear that ETCS is also a \emph{set
  theory}, in the inclusive sense of ``a theory about sets.''
Pursuant to this philosophy, we will therefore use the term
\emph{material set theory} for theories such as BZC and ZFC which are
based on a global membership predicate, and \emph{structural set
  theory} for theories such as ETCS which take functions as more
fundamental.  (Some explanation of this terminology can be found in
the appendix.)

Whatever names we use for them, the relationship between
these two types of theory is like other bridges between fields of
mathematics, in that it augments the toolkit of each side by the ideas
and techniques of the other.  Of central importance to such a
programme is that under this equivalence, certain constructions in
material set theory, such as forcing, can be identified with certain
constructions in topos theory, such as categories of sheaves.
However, there are surprising difficulties in making this precise,
regarding the issue of \emph{well-pointedness}.\footnote{Recall that a
  topos is \emph{well-pointed} if $1$ is a generator and there is no
  map $1\to 0$.  As we will see later, this notion requires some
  revision in an intuitionistic context.}  It is only a
\emph{well-pointed} topos which can be regarded as a structural set
theory and to which the Cole-Mitchell-Osius theory can be applied
directly, since only in a well-pointed topos is an object, like a set,
``determined by its elements'' $1\to X$.  But the categorical
constructions in question almost never preserve well-pointedness, so
some trick is necessary to make the correspondence work.

Of course, the trick is well-known in essence: any topos has an
``internal logic'' (often called the \emph{Mitchell-Benabou language})
which appears much more set-theoretic.  In the internal logic, we can
speak about ``elements'' of objects, and the semantics automatically
interprets statements about these elements into more appropriate
categorical terms.  For instance, a morphism $f\colon X\to Y$ in an
arbitrary topos is an epimorphism if and only if the statement ``for
all $y\in Y$, there exists an $x\in X$ with $f(x)=y$'' is true in the
internal logic.

This trick is quite powerful, and in fact it is fundamental to the
whole field of categorical logic.  However, the internal logic of a
topos is not a set theory of any sort (material or structural), but
rather a ``higher-order type theory.''  In particular, it does not
include \emph{unbounded quantifiers}, i.e.\ quantifiers ranging over
all sets.  This is problematic for the programme of bridging between
material and structural set theories, since many of the most
interesting axioms of set theory (especially separation, collection,
and replacement) involve such unbounded quantifiers.

There do exist secondary tricks which have been used to get around
this difficulty.  The first was proposed by
Fourman~\cite{fourman:sheaf-setthy}: if \bS\ is a \emph{complete}
topos (such as a Grothendieck topos), then we can mimic the
construction of the von Neumann hierarchy inside \bS.  We define
inductively an object $V_\alpha$ for each (external) ordinal $\alpha$,
setting $V_0 = 0$, $V_{\alpha+1}= P V_\alpha$, and taking colimits at
limit stages (hence the need for \bS\ to be complete).  Now the
internal logic can speak about ``elements'' of the objects $V_\alpha$,
which inherit a ``membership'' structure from their construction and
can be shown to model material set theories.
Hayashi~\cite{hayashi:setthy-topos} independently proposed a more
general approach in which the $V_\al$ can be replaced by any class of
objects satisfying suitable closure conditions.

While this construction works quite well, we find it somewhat
unsatisfying for several reasons.  Firstly, it appears to depend on
the non-elementary hypo\-thesis of completeness, which is not
preserved by some constructions on toposes (such as realizability and
filterquotients).  Secondly, it appears to depend on power objects,
which would make it inapplicable to compare pretoposes with
``predicative'' material set theories.  Finally, the treatment of
unbounded quantifiers is asymmetrical---they appear in the material
set theory as if by magic, but there is no natural,
category-theoretic, treatment of unbounded quantifiers on the topos
side.

A more recent solution, begun by Joyal and Moerdijk in~\cite{jm:ast},
is to consider not just a topos, but a \emph{category of classes} in
which it sits as the ``small objects.''  This is a category-theoretic
version of material class-set theories such as Bernays-G\"odel or
Morse-Kelley.  A category of classes contains a single object $V$,
``the class of all sets,'' which admits a membership relation which
satisfies the axioms of material set theory in the internal logic of
the category of classes.  This approach, called \emph{algebraic set
  theory}, avoids some of the problems with the Fourman-Hayashi
approach.  It is purely elementary at the level of the category of
classes, and applies as well in the absence of powersets.  And the
internal logic of the category of classes provides a compelling
category-theoretic treatment of ``unbounded'' quantifiers---they are
simply bounded quantifiers over the class of all sets.


However, algebraic set theory also has its own disadvantages, foremost
among which is the fact that instead of studying a topos by itself,
one must now always somehow construct, and always carry along, a
category of classes containing it.  One would ideally wish for a
\emph{canonical} embedding of any topos into a category of classes,
under which all desirable topos-theoretic structure corresponds
equivalently to suitable structure on the category of classes.  It is
shown in~\cite{af:ast-classes,afw:ast-ideals} that every small topos
can be embedded in \emph{some} category of classes, via a construction
which is in some ways canonical.  However, the resulting category of
classes will not, in general, satisfy axioms such as separation, even
for toposes which might be expected to validate them (for instance,
those which validate their material counter\-parts under the
Fourman-Hayashi interpretation).  A different construction is given
by~\cite{abss:foss-et,abss:rfosttcoc}, which does in some cases
produce categories of classes satisfying separation, but it is very
non-category-theoretic in flavor (in the sense that it distinguishes
between isomorphic objects) and certainly less canonical.

A second problem with algebraic set theory is that one has to worry
about whether any given categorical construction can be extended from
elementary toposes (where the results are usually known and easy) to
any categories of classes that may contain them.  In general, this
turns out to be true, but the proofs often require a great deal of
work (see, for
instance,~\cite{aglw:shvs-ast,vdbm:pred-i-exact,vdbm:pred-ii-realiz,vdbm:pred-iii-shvs}).



I do not mean to denigrate the great advances in understanding
produced by algebraic set theory, and the Fourman-Hayashi
interpretation has also been very fruitful.  The relationship between
the two is also known: the ``ideal semantics''
of~\cite{abss:foss-et,abss:rfosttcoc} generalizes the forcing
sem\-antics of Hayashi.  However, it would undoubtedly also be useful
to have \emph{a unified, canonical, elementary, and category-theoretic
  treatment of unbounded quantifiers over an arbitrary (pre)topos}.
The interpretations of Fourman-Hayashi and algebraic set theory would
then, ideally, turn out to be special cases or generalizations of this
unified framework.

The thesis of this paper is that such a thing is indeed possible; it
is what we call the \emph{stack semantics}.  We can motivate this
semantics in two ways.  The first is to observe that higher-order type
theory (that is, the internal logic of a topos) looks very much like a
\emph{fragment} of structural set theory (i.e.\ the theory of
well-pointed topoi).  Types correspond to objects of the category, and
each term $x$ of type $A$ corresponds to a morphism $1\to A$.  In this
way any assertion in type theory can be translated into an assertion
in the first-order theory of well-pointed categories.  However, the
latter theory also includes quantifiers over objects of the category,
so it is natural to seek a generalization of the internal logic which
models the first-order theory of well-pointed categories, rather than
higher-order type theory.\footnote{Alternately, one can use a type
  theory allowing propositions which quantify over the universe
  $\mathbf{Type}$.  We will not do this, but it has certain
  advantages.\label{fn:2}} The stack semantics is exactly such a
generalization.

(When the logic is intuitionistic, as it is in the stack semantics,
the notion of ``well-pointed'' has to be modified a bit; see
\S\ref{sec:struct-set}.  For clarity, we usually refer to this
modified property \emph{constructive well-pointedness}; thus what the
stack semantics really models is the first-order theory of
constructively well-pointed topoi.)

From this point of view, the stack semantics can be defined directly
by generalizing the ``Kripke-Joyal'' description of the internal logic
via a forcing relation.  This description is thus closely related to
Hayashi's forcing interpretation of material set theory, and also to
the forcing semantics defined in~\cite{abss:rfosttcoc,abss:foss-et}.

A second motivation for the stack semantics comes from trying to make
the logic of a category of classes more canonical and
category-theoretic.  Specifically, if \bS\ is a small topos, an
obvious ``canonical'' candidate for a category of classes containing
it is the category $\bSh(\bS)$ of sheaves for the coherent topology on
\bS.  (The category of ideals defined
in~\cite{af:ast-classes,afw:ast-ideals} is a certain well-behaved
subcategory of $\bSh(\bS)$.)  If we then want to interpret
``unbounded'' quantifiers over objects of \bS\ as ``bounded''
quantifiers in the internal logic of $\bSh(\bS)$, we need a suitable
object to serve as ``the class of sets.''  Or, more
category-theoretically, we want an internal category in $\bSh(\bS)$ to
serve as ``the category of sets,'' and ideally we would like this
internal category to also be canonically determined by \bS.

Now \bS\ of course has a canonical representation as a \emph{stack}
over itself, namely the ``self-indexing'' or ``fundamental fibration''
$X\mapsto \bS/X$.  This is not an internal category in $\bSh(\bS)$,
since it is not a strict functor, only a pseudofunctor.  But we can
choose some strictification of it, which is therefore an internal
category in $\bSh(\bS)$, and let our ``unbounded'' quantifiers run
over this.  This is precisely the approach taken
in~\cite[Ch.~V]{awodey:thesis}, where it was additionally observed
that in the internal logic of $\bSh(\bS)$, the self-indexing is (in
our terminology) constructively well-pointed.  However, now we take the additional step of
isolating the particular fragment of the internal logic of $\bSh(\bS)$
whose only ``unbounded'' quantifiers range over the self-indexing.
Modulo one small caveat, this fragment is canonically determined by
the topos \bS, and when expressed in the sheaf semantics of
$\bSh(\bS)$ via a forcing relation over the site \bS, it agrees with
the stack semantics defined using the first approach.

The caveat is that there is one other thing this logic depends on: the
chosen strictification of the self-indexing.  However, this choice
only affects the truth values of ``non-category-theoretic'' formulas,
such as those which assert equalities between objects.  In the stack semantics
proper, we exclude such formulas by restricting to a dependently typed
logic in which there simply is no equality predicate between objects.
(This is also precisely the necessary restriction so that the truth
of such formulas be invariant under equivalence of categories.)
This restricted logic can then be naturally identified with a fragment
of the internal logic of the \emph{2-category} $\St(\bS)$ of
\emph{stacks} on \bS.  (The internal logic of a 2-category is not
well-known, but it can be defined by close analogy with that of a
1-category; see~\cite{shulman:2logic}.)

The existence of these multiple approaches to the stack semantics
contributes to the feeling that it is a natural part of topos theory,
though heretofore seemingly unappreciated.  In fact, however, it is
implicitly present throughout the literature.  Topos theory abounds
with definitions, arguments, and theorems that are said to be
``internal'' statements, but in many cases this cannot be
\emph{directly} understood as an assertion in the usual internal logic
due to the presence of unbounded quantifiers.  As a very simple example,
the external topos-theoretic ``axiom of choice'' (AC) asserts that
every epimorphism splits.  This involves an unbounded quantification
over all objects of the topos, so it is not directly internalizable.
Instead, the ``internal axiom of choice'' (IAC) asserts (in one
formulation) that for any object $A$, the functor $\Pi_A\colon
\bS/A\to \bS$ preserves epimorphisms.  To a beginner, it may not be
obvious why this is the ``internalization'' of AC.  But it turns out
that IAC is equivalent to the validity in the stack semantics of the
statement ``every epimorphism splits.''

The proper way of internalization of such statements does soon become
second nature to experts, and it has
even started to appear more openly in the literature, such as in the
``informal arguments'' given in~\cite{mp:ttcst}.  However, to my
knowledge it has never before been given a \emph{formal, precise}
expression.  This is what the stack semantics provides: a formal
definition of the language whose statements we can expect to
internalize, along with a procedure for performing that
internalization and a theorem that the resulting semantics is sound.
Thus, it has the potential to simplify and clarify many arguments in
topos theory, by reducing them to internalizations of easier proofs in
structural set theory.

(This effect is even more pronounced when the stack semantics is
generalized to the full internal logic of $\St(\bS)$.
In~\cite{shulman:etlc} we will show that almost any
argument in the ``naive'' theory of large categories, such as Freyd's
adjoint functor theorem, or Giraud's theorem characterizing
Grothendieck topoi, can be interpreted in the generalized stack
semantics to immediately yield the corresponding result for indexed
categories over any base topos.)


However, in the present paper we are only concerned with the basic
theory of the stack semantics, and its application to comparison of
material and structural set theories.  In particular, by generalizing
the higher-order type theory of the internal logic to a semantics for
structural set theory, the stack semantics enables us to reconstruct
an interpretation of material set theory from any topos, by simply
performing a version of the Cole-Mitchell-Osius construction inside
the stack semantics.  As suggested previously, this unifies, rather
than replaces, the existing approaches of Fourman-Hayashi and of
algebraic set theory.  On the one hand, if the topos is complete, then
the Fourman-Hayashi model can be identified with that derived from the
stack semantics (in fact, one version of Hayashi's approach
essentially \emph{is} the stack semantics model).  On the other hand,
if a topos \bS\ is embedded inside a category of classes \bC\ (and
\bS\ ``generates'' \bC\ in a suitable sense), then the stack semantics
of \bS\ agrees with a fragment of the internal logic of \bC\ (as
summarized above in the case $\bC = \bSh(\bS)$), and in good cases the
resulting models of material set theory can also be identified.


In a sense, these equivalences are really only an
expected ``consistency check.''  We believe that where the stack
semantics really shows its worth is in suggesting topos-theoretic
counter\-parts of additional set-theoretic axioms such as separation,
collection, and replacement, which until now have been largely absent
from the literature.  With a little thought, it is not hard to write
down axioms for \emph{structural set theory} which correspond under
the Cole-Mitchell-Osius equivalence to the material-set-theoretic
axioms of separation, collection, and replacement.  By then
interpreting these axioms in the stack semantics, we can obtain true
topos-theoretic versions of such axioms which are applicable to an
arbitrary topos.


It turns out that under interpretation in the stack semantics, the
axioms of collection and of separation behave quite differently.  The
axiom of collection (and hence that of replacement) is \emph{always}
valid in the stack semantics of any topos.  (The fact that forcing
semantics always validate collection has been observed elsewhere,
e.g.~\cite{abss:foss-et}.)  It follows that the models of material set
theory constructed from the stack semantics always validate the
material version of collection, providing a neat proof of the
well-known fact that intuitionistically, collection without separation
adds no proof-theoretic strength.  In fact, collection emerges as
precisely the dividing line separating the ``internal'' stack
semantics from the ``external'' logic: a constructively well-pointed topos satisfies
collection (in the ``external'' sense, not in its stack semantics) if
and only if validity in its stack semantics is equivalent to
``external'' truth for \emph{all} sentences.

On the other hand, the stack-semantics version of separation seems to
be a new topos-theoretic property.  We call a topos with this property
\emph{autological}.  Autology\footnote{I am indebted to Peter Freyd
  for this noun form, which is definitely easier on the tongue than
  ``autological-ness''.} is also closely connected to the stack
semantics: a topos is autological if and only if its stack semantics
is \emph{representable}, in the same sense that the usual Kripke-Joyal
semantics is always representable.  (This motivates the name: a topos
is autological if the logic of its stack semantics can be completely
``internalized'' by representing subobjects.)  Aut\-ology is also
quite common and well-behaved; for instance, we will show in
\S\ref{sec:fourman} that all complete topoi are autological, and in
in~\cite{shulman:uqsa} we will study its preservation by other
topos-theoretic constructions (including realizability).

Moreover, since autology is simply the axiom of separation in the
stack semantics, the Cole-Mitchell-Osius model of material set theory
constructed from an auto\-logical topos must validate both collection
and separation.  Thus, autology provides a natural and \emph{fully
  topos-theoretic} strengthening of the notion of elementary topos,
which is equiconsistent with full ZF set theory, and which
``explains'' why complete and realizability topoi are known to model
(intuitionistic) ZF.  These two cases have been addressed by algebraic
set theory in~\cite{abss:foss-et,abss:rfosttcoc}, but as remarked
above, the construction there is non-category-theoretic and
non-elementary.

In fact, the relationship of autology to the axiom of separation for
categories of classes is unidirectional: if a category of classes
satisfies separation, then its topos of small objects is autological,
but not conversely.  The issue is that since autology is a first-order
axiom schema, it can only speak about proper classes which are
\emph{definable} in some sense over the topos, but the objects of a
category of classes need not be definable in any such way.

One might, however, ask whether every autological topos could be
embedded in \emph{some} category of classes satisfying separation,
analogously to to how every model of ZFC can be embedded in some model
of Bernays-G\"odel class-set theory (namely, the model consisting of
its definable proper classes).  We will show in~\cite{shulman:2cofc}
that the answer is almost yes: any autological topos can be embedded
in a \emph{2-category} of ``definable stacks'' which satisfies
separation, and which satisfies 2-categorical versions of the axioms
of algebraic set theory.  However, it seems unlikely that this can be
improved to a 1-category of classes, since this would require some
sort of \emph{definable} rigidification of the self-indexing.  This
suggests that perhaps we should study \emph{2-categories of large
  categories} instead of categories of classes.  Such a modification
is also category-theoretically natural, since in practice we generally
only care about elements of a proper class insofar as they form some
large category.  See~\cite{shulman:2cofc,shulman:etlc} for further
exploration of this idea.

The plan of the paper is as follows.  We begin by defining carefully
the material and structural set theories we plan to compare.  Since
material set theories are quite familiar, we list their axioms briefly
in \S\ref{sec:sets}, while in
\S\S\ref{sec:struct-set}--\ref{sec:strong-ax} we spend rather more
time on the axioms of structural set theory.  In both cases we include
``predicative'' theories as well as ``impredicative'' ones.  Then in
\S\ref{sec:constr-mat} we describe a version of the
Cole-Mitchell-Osius comparison that is valid even predicatively, since
the construction does not appear to be in the literature in this
generality.

The core of the paper is \S\ref{sec:internal-logic}, in which we
introduce the stack semantics and prove its basic properties,
including the interpretations of the axioms of separation and
collection mentioned above.  It is then
almost trivial to derive the desired interpretations of material set
theory; we say a few words about this in \S\ref{sec:mater-set-theor}.
Then in \S\S\ref{sec:fourman} and \ref{sec:cofc} we show that this
interpretation agrees with both the Fourman-Hayashi model and with the
models of algebraic set theory.  Appendix \ref{sec:mat-vs-struc} is
devoted to some philosophical remarks comparing material and
structural set theory, particularly as foundations for mathematics.

Our basic reference for topos-theoretic notions
is~\cite{ptj:elephant}.  One important thing to note is that we will
use intuitionistic logic in \emph{all} contexts.  This is in contrast
to the common practice in topos theory of reasoning about the topos
itself using classical logic, though the internal logic of the topos
is intuitionistic.  Our approach is necessary because we want all of
our arguments to remain valid in the stack semantics, which, like the
internal logic of a topos, is generally intuitionistic.  We treat
classical logic as an axiom schema (namely, $\ph\vee\neg\ph$ for all
formulas $\ph$) which may or may not be satisfied by any given model.

All of our theories, material and structural, are of course properly
expressed in a formal first-order language.  However, to make the
paper easier to read, we will usually write out logical formulas in
English, trusting the reader to translate them into formal symbols as
necessary.  To make it clear when our mathematical English is to be
interpreted as code for a logical formula, we will often use sans
serif font and corner quotes.  Thus, for example, in a structural set
theory, \qq{every surjection has a section} represents the axiom of
choice.

If \bM\ models some theory, by a \textbf{formula in \bM} we mean the
result of substituting elements of \bM\ for some of the free variables
of a formula \ph\ in a well-typed way.  The substituted elements of
\bM\ are called \textbf{parameters}.  If parameters are substituted
for \emph{all} the free variables of \ph\ (including the case when
\ph\ had no free variables), we say \ph\ is a \textbf{sentence in
  \bM}.  There is an obvious notion of when a sentence \ph\ in \bM\ is
\textbf{true}, which we denote $\bM\ss\ph$ (read \textbf{\bM\
  satisfies \ph}).  Recall that our metatheory, in which we formulate
this notion of truth, is intuitionistic.

Many of the axioms we state will be axiom schemas depending on some
formula \ph.  In each such case, \ph\ may have arbitrary additional
unmentioned parameters.  Equivalently, such axiom schemas can be
regarded as implicitly universally quantified over all the unmentioned
variables of \ph.

I would like to thank Steve Awodey, Toby Bartels, Peter Lumsdaine,
Colin McLarty, Thomas Streicher, Paul Taylor, and the patrons of the
$n$-Category Caf\'e for many helpful conversations and useful comments
on early drafts of this paper.


\section{Material set theories}
\label{sec:sets}

Our \emph{material set theories} will be theories in one-sorted first
order logic, with a single binary relation symbol \ordin, whose
variables are called \emph{sets}.  (We will briefly mention the
possibility of ``atoms'' which are not sets at the end of this
section.)
In such a theory, a \emph{\ddo-quantifier} is one of the form $\exists
x\in a.$ or $\forall x\in a.$, and a \emph{\ddo-formula} (also called
\emph{restricted} or \emph{bounded}) is one with only
\ddo-quantifiers.  The prefix ``\ddo-'' on any axiom schema indicates
that the formulas appearing in it are restricted to be \ddo; we
sometimes use the adjective ``full'' to indicate the lack of any such
restriction.


\begin{defns}\label{defn:rels}\ 
  \begin{blist}
  \item We write $\setof{ x | \ph(x)}$ for a set $a$ such that $x\in a
    \Iff \ph(x)$, if such a set exists.
  \item A \textbf{relation} from $a$ to $b$ is a set $r$ of ordered
    pairs $(x,y)$ such that $x\in a$ and $y\in b$, where $(x,y) =
    \{\{x\},\{x,y\}\}$.
  \item A relation $r$ from $a$ to $b$ is \textbf{entire} (or a
    \emph{many-valued function}) if for all $x\in a$ there exists
    $y\in b$ with $(x,y)\in r$.  If each such $y$ is unique, $r$ is a
    \textbf{function}.
  \item A relation $r$ is \textbf{bi-entire} if both $r$ and $r^o =
    \setof{(y,x)|(x,y)\in r}$ are entire.
  \end{blist}
\end{defns}

The following set of axioms (in intuitionistic logic) we call the
\textbf{core axioms}; they form a subsystem of basically all material
set theories.
\begin{blist}
\item \emph{Extensionality}: \qq{for any $x$ and $y$, $x=y$ if and
    only if for all $z$, $z\in x$ iff $z\in y$}.
\item \emph{Empty set:} \qq{the set $\emptyset = \{\}$ exists}.
\item \emph{Pairing:} \qq{for any $x$ and $y$, the set $\{x,y\}$
    exists}.
\item \emph{Union:} \qq{for any $x$, the set $\setof{z | \exists y\in
      x. (z\in y)}$ exists}.
\item \emph{\ddo-Separation:} For any \ddo-formula $\ph(x)$, \qq{for
    any $a$, $\setof{x\in a | \ph(x)}$ exists}.
\item \emph{Limited \ddo-replacement}: For any \ddo-formula
  $\ph(x,y)$, \qq{for all $a$ and $b$, if for every $x\in a$ there
    exists a unique $y$ with $\ph(x,y)$, and moreover $z\subseteq b$
    for all $z\in y$, then the set $\setof{y | \exists x\in
      a.\ph(x,y)}$ exists}.
\end{blist}

The first five of these are standard, while the last, though very
weak, is unusual.  We include it because we want our core axioms to
imply a well-behaved category of sets, for which purpose the other
five are not quite enough.  It is satisfied in basically all material
set theories, since it is implied both by ordinary replacement (even
\ddo-replacement), and by the existence of power sets (using
\ddo-separation).



We now list some additional axioms of material set theory, pointing
out the most important implications between them.  Of course, there
are many possible set-theoretic axioms; we will consider only the most
common and important ones.

One major axis along which set theories vary is ``impredicativity,''
expressed by the following sequence of axioms (of increasing
strength):
\begin{blist}
\item \emph{Exponentiation:} \qq{for any $a$ and $b$, the set $b^a$ of
    all functions from $a$ to $b$ exists}.
\item \emph{Fullness:} \qq{for any $a,b$ there exists a set $m$ of
    entire relations from $a$ to $b$ such that for any entire
    relation $f$ from $a$ to $b$, there is a $g\in m$ with
    $g\subseteq f$}.
\item \emph{Power set:} \qq{for any set $x$, the set $P x = \setof{y |
      y \subseteq x}$ exists}.
\end{blist}
Each of these is strictly stronger than the preceding one, but in the
presence of either of the following:
\begin{blist}
\item \emph{\ddo-classical logic:} for any \ddo-formula \ph, we have
  $\ph\join\neg\ph$.
\item \emph{Full classical logic:} for any formula \ph, we have
  $\ph\join\neg\ph$.
\end{blist}
the hierarchy of predicativity collapses, since power sets can then be
obtained (using limited \ddo-replacement) from sets of functions into a
2-element set.  Hence classical logic is also impredicative.  Also
considered impredicative is:
\begin{blist}
\item \emph{(Full) Separation:} For any formula $\ph(x)$, \qq{for any
    $a$, $\setof{x\in a | \ph(x)}$ exists}.
\end{blist}
Full separation and \ddo-classical logic together imply full classical
logic, since we can form $a = \setof{\emptyset|\ph}\subseteq
\setof{\emptyset}$, and then $(\emptyset\in a)\join (\emptyset\not\in
a)$ is an instance of \ddo-classical logic that implies
$\ph\join\neg\ph$.


We also have the collection and replacement axioms.
\begin{blist}
\item \emph{Collection:} for any formula $\ph(x,y)$, \qq{for any $a$,
    if for all $x\in a$ there is a $y$ with $\ph(x,y)$, then there is
    a $b$ such that for all $x\in a$, there is a $y\in b$ with
    $\ph(x,y)$, and for all $y\in b$, there is a $x\in a$ with
    $\ph(x,y)$}.
\item \emph{Replacement:} for any formula $\ph(x,y)$, \qq{for any $a$,
    if for every $x\in a$ there exists a unique $y$ with $\ph(x,y)$,
    then the set $\setof{y | \exists x\in a.\ph(x,y)}$ exists}.
\end{blist}
With an evident extension of \ref{defn:rels} to formulas, we can state
collection as ``if $\ph$ is entire on $a$, then there exists $b$ such
that $\ph$ is bi-entire from $a$ to $b$.''  Collection also evidently
implies replacement.  Classical ZF set theory is usually defined using
replacement rather than collection, but in the presence of the other
axioms of ZF, collection can be proven.  (Replacement and full
classical logic together also imply full separation.)  This is no
longer true in intuitionistic logic, though, or in classical logic
without the axiom of foundation, and in practice the extra strength of
collection is sometimes important.  There are also other versions of
collection, all of which are equivalent in the presence of full
separation, but the one we have chosen seems most appropriate in its
absence.



We state the axiom of infinity as the existence of the set $\omega$ of
finite von Neumann ordinals.  Together with the core axioms, this
implies the Peano axioms for $\omega$, including induction for any
property to which we can apply separation.  For a stronger version of
induction, we need a separate axiom.
\begin{blist}
\item \emph{Infinity:} \qq{there exists a set $\omega$ such that
    $\emptyset\in\omega$, if $x\in \omega$ then $x\cup \{x\}\in
    \omega$, and if $z$ is any other set such that $\emptyset \in z$
    and if $x\in z$ then $x\cup \{x\}\in z$, then $\omega\subseteq
    z$}.
\item \emph{Induction:} for any formula \ph, \qq{if $\ph(\emptyset)$
    and $\ph(x) \imp \ph(x\cup \{x\})$ for any $x\in \omega$, then
    $\ph(x)$ for all $x\in \omega$}.
\end{blist}
There is also:


\begin{blist}
\item \emph{Choice:} \qq{for any set $a$, if for all $x\in a$ there
    exists a $y\in x$, then there exists a function $f$ from $a$ to
    $\bigcup a$ such that $f(x)\in x$ for all $x\in a$}.
\end{blist}
By Diaconescu's argument~\cite{diaconescu:ac}, choice and power set
together imply \ddo-classical logic.  There are of course many
variants of choice, such as countable, dependent, or ``multiple''
choice, or the ``presentation axiom'' (a.k.a. ``EPSets'' or
``COSHEP,'' the Category Of Sets Has Enough Projectives), but we will
confine our attention to the ordinary version.



Finally, we have the family of ``foundation-type'' axioms.  Recall
that a set $x$ is called \emph{transitive} if $z\in y\in x$ implies
$z\in x$.
A set $X$ with a binary relation $\prec$ is called \emph{well-founded}
if whenever $Y\subseteq X$ is \emph{inductive}, in the sense that
$\forall y\prec x. (y\in Y)$ implies $x\in Y$, then in fact $Y=X$.  It
is \emph{extensional} if $\forall z\in X.(z\prec x\Iff z\prec y)$
implies $x=y$.
\begin{blist}
\item \emph{Set-induction:} for any formula $\ph(x)$, \qq{if for any
    set $y$, $\ph(x)$ for all $x\in y$ implies $\ph(y)$, then $\ph(x)$
    for any set $x$}.
\item \emph{Foundation}, a.k.a.\ \emph{\ddo-Set-induction:} like
  set-induction, but only for \ddo-formulas.
\item \emph{Transitive closures:} \qq{every set is a subset of a
    smallest transitive set}.
\item \emph{Mostowski's principle:} \qq{every well-founded extensional
    relation is isomorphic to a transitive set equipped with the
    relation $\in$}.
\end{blist}
The usual ``classical'' formulation of foundation is
\begin{blist}
\item \emph{Regularity:} \qq{if $x$ is nonempty, there is some $y\in
    x$ such that $x\cap y = \emptyset$}.
\end{blist}
In the presence of transitive closures, regularity is equivalent to
the conjunction of \ddo-set-induction and \ddo-classical logic.  Thus,
in an intuitionistic theory, \ddo-set-induction is preferable.  Of
course, \ddo-set-induction follows from full set-induction, while the
converse is true if we have full separation.  Set-induction also
implies ordinary induction, since $<$ and $\in$ are the same for von
Neumann ordinals.

Transitive closures can be constructed either from set-induction and
replacement (see~\cite{aczel:cst}), or from infinity, ordinary (full)
induction, and replacement.  Mostowski's original proof shows that his
principle follows from replacement and full separation.  Transitive
closures and Mostowski's principle are not usually considered
``foundation'' axioms, but we group them thusly for several reasons:
\begin{blist}
\item The conjunction of transitive closures, Mostowski's principle,
  and foundation is equivalent to the single statement ``to give a
  \emph{set} is the same as to give a well-founded extensional
  relation with a specified element.''  One could argue that this sums
  up the objects of study of well-founded material set theory.
\item As we will see in \S\ref{sec:constr-mat}, it follows that the
  material set theories that can be constructed from structural ones
  are precisely those which satisfy these three axioms.
\item Finally, the various axioms of \emph{anti-foundation} are also
  naturally stated as ``to give a set is the same as to give a
  \emph{(blank)} relation with a specified element,'' where
  \emph{(blank)} denotes some stronger type of extensionality (there
  are several choices; see~\cite{aczel:afa}).  Thus, these axioms
  naturally include analogues not just of foundation, but also of
  transitive closures and Mostowski's principle.
\end{blist}

Some particular material set theories with names include the
following.  Note that the prefix \emph{constructive} is generally
reserved for ``predicative'' theories (including, potentially,
exponentiation and fullness, but never power sets or full separation)
while \emph{intuitionistic} refers to ``impredicative'' (but still
non-classical) theories which may have power sets or full separation.
None of these theories include our axiom of limited \ddo-replacement
explicitly, but all include either power sets or a stronger version of
replacement, both of which imply it.
\begin{blist}
\item The standard \textbf{Zermelo-Fraenkel set theory with Choice
    (ZFC)} includes \emph{all} the axioms mentioned above, although it
  suffices to assume the core axioms together with full classical
  logic, power sets, replacement, infinity, foundation, and choice.
\item By \textbf{Zermelo set theory (Z)} one usually means the core
  axioms together with full classical logic, power sets, full
  separation, and infinity.
\item The theory \textbf{ZBQC} of~\cite{maclane:formandfunction},
  called \textbf{RZC} in~\cite[VI.10]{mm:shv-gl}, consists of the core
  axioms together with full classical logic, power sets, infinity,
  choice, and regularity.
\item The \textbf{Mostowski set theory (MOST)}
  of~\cite{mathias:str-maclane} consists of \textbf{ZBQC} together
  with transitive closures and Mostowski's principle.
\item By \textbf{Kripke-Platek set theory (KP)} is usually meant the
  core axioms together with full classical logic, a form of
  \ddo-collection (sufficient to imply our limited \ddo-replacement),
  and set-induction for some class of formulas.
\item Aczel's \textbf{Constructive Zermelo-Fraenkel (CZF)} consists of
  the core axioms together with fullness, collection, infinity, and
  set-induction; see~\cite{aczel:cst}.  His \textbf{CZF$_0$} consists
  of the core axioms, replacement, infinity, and induction.
\item The usual meaning of \textbf{Intuitionistic Zermelo-Fraenkel
    (IZF)} is the core axioms together with power sets, full
  separation, collection, infinity, and set-induction.  This is the
  strongest theory that can be built from the above axioms without any
  form of classical logic or choice.
\item The \textbf{Basic Constructive Set Theory (BCST)}
  of~\cite{afw:ast-ideals} consists essentially of the core axioms
  together with replacement.  Their \textbf{Constructive Set Theory
    (CST)} adds exponentials to this, and their \textbf{Basic
    Intuitionistic Set Theory (BIST)} adds power sets.  (These
  theories also allow atoms; see below.)
\item The \textbf{Rudimentary Set Theory (RST)}
  of~\cite{vdbm:pred-i-exact} consists of the core axioms together
  with collection and set-induction.
\end{blist}


So far we have considered only theories of pure sets, i.e.\ in which
everything is a set.  If we want to also allow ``atoms'' or
``urelements'' which are not sets, we can modify any of the above
theories by adding a predicate \qq{$x$ is a set}, whose negation is
read as \qq{$x$ is an atom}.  We then add the following axioms:
\begin{blist}
\item \emph{Decidability of sethood:} \qq{every $x$ is either a set or
    an atom}.
\item \emph{Only sets have elements:} \qq{if $x\in y$, then $y$ is a set}.
\item \emph{Set of atoms:} \qq{the set $\setof{x | x \text{ is an
        atom}}$ exists}.
\end{blist}
and modify the extensionality axiom so that it applies only to sets:
\begin{blist}
\item \emph{Extensionality for sets:} \qq{for any sets $x$ and $y$,
    $x=y$ if and only if for all $z$, $z\in x$ iff $z\in y$}.
\end{blist}

\section{Structural set theory: well-pointed Heyting pretoposes}
\label{sec:struct-set}

We now move on to structural set theories.  Here we must be a little
more careful with the ambient logic.  The theory of categories can, of
course, be formulated in ordinary first-order logic, but this is
unsatisfactory because it allows us to discuss equality of objects.
We need all our formulas to express ``category-theoretic'' facts, and
thus we must avoid ever asserting that two objects are equal (rather
than isomorphic).  However, we cannot simply remove the equality
predicate on objects from the ordinary first-order theory, because
there are situations in which we \emph{do} need to know that two
objects are the same.  For example, to compose two arrows $f$ and $g$, we
need to know that the source of $g$ is the same as the target of
$f$.

The solution to this problem, which has seemingly been rediscovered
many times (see, for
instance,~\cite{blanc:eqv-log,freyd:invar-eqv,makkai:folds,makkai:comparing})
is to use a language of \emph{dependent types}.  We give here a
somewhat informal description in an attempt to avoid the complicated
precise syntax of dependent type theory, but the reader familiar with
DTT should easily be able to expand our description into a formal one.

\begin{defn}\label{defn:lang-cat}
  The \textbf{language of categories} (with terminal object) consists
  of the following.
  \killspacingtrue
  \begin{enumerate}
  \item There is a type of \emph{objects}.
  \item There is a specified object-term $1$.
  \item For object-terms $X$ and $Y$, there is a type $X\to Y$
    of \emph{arrows from $X$ to $Y$}.
  \item For each object-term $X$, there is an arrow-term $1_X\colon
    X\to X$.
  \item For any object-terms $X$, $Y$, and $Z$ and each pair of
    arrow-terms $f\maps X\to Y$ and $g\maps Y\to Z$, there is an
    arrow-term $(g\circ f)\maps X\to Z$.
  \item For any arrow-terms $f,g\colon X\to Y$, there is an atomic
    formula $(f=g)$.
  \end{enumerate}
  \killspacingfalse
  Non-atomic formulas are built up from atomic ones in the usual way,
  using connectives $\top$, $\bot$, $\meet$, $\join$, $\imp$, $\neg$,
  and quantifiers $\im$ and $\coim$.  There are of course two types of
  quantifiers, for object-variables and for arrow-variables.  Any
  category is a structure for this language, and we have the usual
  satisfaction relation $\bA\ss\ph$ for a category \bA\ and a sentence
  \ph\ in \bA.
\end{defn}

Of course, this theory uses very little of the full machinery of DTT,
and in particular it may be said to live in the ``first-order''
fragment of DTT.  In this way it is closely related to the logic FOLDS
studied in~\cite{makkai:folds}, although we are allowing some term
constructors in addition to relations.

Our inclusion of a specified term $1$, intended to denote a terminal
object, is just for convenience; it will make it easier to single out
the \ddo-formulas.  The other main thing to note is that, as in FOLDS,
the only atomic formulas are equalities between parallel arrows; the
language does not allow us to even discuss whether two objects are
equal.  This implies that ``truth is isomorphism- and
equivalence-invariant,'' in a sense we can make precise as follows.

Let \bA\ be a category, let \ph\ be a sentence in \bA, and suppose we
are given, for each object-parameter $A$ in \ph, an isomorphism
$A\toiso A'$ in \bA.  Let $\ph'$ denote the sentence in \bA\ obtained
from \ph\ by replacing each object-parameter $A$ by $A'$ and each
arrow-parameter $A\too[j] B$ by the composite
\[A' \too[\iso] A \too[j] B\too[\iso] B'.
\]
We call $\ph'$ an \textbf{isomorph} of \ph.  Likewise, if $F\colon
\bA\to\bB$ is any functor and \ph\ is a sentence in \bA, we write
$F(\ph)$ for the sentence in \bB\ obtained by applying $F$ to all the
parameters of \ph.  The following can then easily be proven by
induction on the construction of formulas.

\begin{lem}[Isomorphism-invariance of truth]\label{thm:isoinvar-truth}
  If \ph\ is a sentence in \bA\ and $\ph'$ an isomorph of \ph, then
  $\bA\ss\ph$ if and only if $\bA\ss\ph'$.\qed
\end{lem}

\begin{lem}[Equivalence invariance of truth]
  If $F$ is fully faithful and essentially surjective, then
  $\bA\ss\ph$ if and only if $\bB \ss F(\ph)$.\qed
\end{lem}

In fact, the dependently typed theory of categories characterizes
exactly those properties of categories which are invariant under
equivalence in this sense;
see~\cite{blanc:eqv-log,freyd:invar-eqv,makkai:folds}.

We now begin a listing of the axioms of the structural set theories we
will consider.  First of all, we must have a category, so we include:
\begin{blist}
\item \emph{Identity:} $1_Y\circ f = f = f\circ 1_X$.
\item \emph{Associativity:} $h\circ (g\circ f)=(h\circ g)\circ f$.
\end{blist}
Many of the additional axioms of structural set theories assert
well-known categorical properties.  Recall that a \textbf{Heyting
  category} is a category \bS\ satisfying the following axioms.
\begin{blist}
\item \emph{Finite limits:} \qq{\bS\ has pullbacks, and $1$ is a terminal object}.
\item \emph{Regularity:} \qq{every morphism factors as a regular epi
    followed by a monic, and regular epis are stable under pullback}.
\item \emph{Coherency:} \qq{finite unions of subobjects exist and are
    stable under pullback}.
\item \emph{Dual images:} \qq{for any $f\maps X\to Y$, the pullback
    functor $f^*\maps \Sub(Y)\to\Sub(X)$ has a right adjoint
    $\coim_f$}.
\end{blist}
Here $\Sub(X)$ denotes the poset of subobjects of $X$, i.e.\ of
monomorphisms into $X$ modulo isomorphism.  If \bS\ also satisfies
\begin{blist}
\item \emph{Positivity/extensivity:} \qq{\bS\ has disjoint and
    pullback-stable binary coproducts}.
\end{blist}
it is called \textbf{positive} (or \textbf{extensive}).  If it is
positive and also satisfies:
\begin{blist}
\item \emph{Exactness:} \qq{every equivalence relation is a kernel}.
\end{blist}
it is called a \textbf{Heyting pretopos}.  A Heyting pretopos which
also satisfies:
\begin{blist}
\item \emph{Local cartesian closure:} \qq{for every arrow $f\maps X\to Y$,
    the pullback functor $f^*\maps \bS/Y\to\bS/X$ has a right adjoint
    $\Pi_f$}
\end{blist}
(which implies dual images) is called a \textbf{\Pi-pretopos}.
Finally, if it satisfies:
\begin{blist}
\item \emph{Power objects:} \qq{every object $X$ has a power object $P
    X$}
\end{blist}
(which, together with finite limits, implies all the other axioms of a
\Pi-pretopos) it is called an \textbf{(elementary) topos}.  See,
e.g.,~\cite[\S A.1]{ptj:elephant} for more details.  In particular, we
remind the reader that in a pretopos, every monic and every epic is
regular, and thus it is \emph{balanced} (every monic epic is an
isomorphism).  Also, the self-indexing of any regular category is a
prestack for its regular topology (generated by regular epimorphisms),
and it is a stack if the category is exact (such as a pretopos).

\begin{rmk}
  Axioms which assert the existence of adjoints (such as dual images
  and local cartesian closure) seemingly go beyond first-order logic,
  but they can easily be formulated in an elementary way, as simply
  asserting the existence of objects with a suitable universal
  property.
\end{rmk}


Three other important axioms are the following:
\begin{blist}
\item \emph{Booleanness:} \qq{every subobject has a complement}.
\item \emph{Full classical logic:} for any formula \ph, we have
  $\ph\join\neg\ph$.
\item \emph{Infinity:} \qq{there exists a natural number object
    (\nno)}.
\item \emph{Choice:} \qq{every regular epimorphism splits}.
\end{blist}
(If \bS\ is not cartesian closed, then the definition of an \nno\
should be taken with arbitrary parameters;
see~\cite[A2.5.3]{ptj:elephant}.)

Each of the above axioms (beyond those of a Heyting pretopos) is, more
or less obviously, a counterpart of some axiom of material set theory.
The first easy observation is that each material axiom implies its
corresponding structural axiom for the category of sets.

\begin{thm}\label{thm:iz-topos}
  If \bV\ satisfies the core axioms of material set theory, then the
  sets and functions in \bV\ form a Heyting pretopos $\bbSet(\bV)$.
  Moreover:
  \killspacingtrue
  \begin{enumerate}
  \item If \bV\ satisfies exponentiation, then $\bbSet(\bV)$ is
    locally cartesian closed.\label{item:iz-topos-lcc}
  \item If \bV\ satisfies the power set axiom, then $\bbSet(\bV)$ is a
    topos.
  \item If \bV\ satisfies \ddo-classical logic, then $\bbSet(\bV)$ is
    Boolean.
  \item If \bV\ satisfies full classical logic, so does $\bbSet(\bV)$.
  \item If \bV\ satisfies infinity and exponentials, then
    $\bbSet(\bV)$ has a \nno.\label{item:iz-topos-nno}
  \item If \bV\ satisfies the axiom of choice, then so does $\bbSet(\bV)$.
  \end{enumerate}
  \killspacingfalse
\end{thm}
\begin{proof}
  We first show that $\bbSet(\bV)$ is a category.  Using pairing and
  limited \ddo-replacement, we can form, for any set $a$ and any $x\in
  a$, the set $p_{a,x}$ of all pairs $\{x,y\}$ for $y\in a$.  Again
  using limited \ddo-replacement, we can form the set of all the
  $p_{a,x}$ for $x\in a$, and take its union, thereby obtaining the
  set of all pairs $\{x,y\}$ for $x,y\in a$.  Applying this
  construction twice to $a\cup b = \bigcup\{a,b\}$, we can then use
  \ddo-separation to cut out the set $a\times b$ of ordered pairs
  $(x,y)$ where $x\in a$ and $y\in b$.  With this in hand we can
  define function composition and identities using \ddo-separation.

  Finite limits are straightforward: we already have products,
  $\{\emptyset\}$ is a terminal object, and \ddo-separation supplies
  equalizers.  The construction of pullback-stable images, unions,
  and dual images is also easy from \ddo-separation.  A stable and
  disjoint coproduct $a+b$ can as usual be obtained as a subset of
  $\{0,1\}\times (a\cup b)$.  And if $r\subseteq a\times a$ is an
  equivalence relation, \ddo-separation supplies the equivalence class
  of any $x\in a$, and limited \ddo-replacement then supplies the set
  of all such equivalence classes.

  Exponentiation clearly implies cartesian closedness.  For local
  cartesian closedness, suppose given $f\maps a\to b$ and $g\maps x\to
  a$.  Then the fiber of $\Pi_f(g)\to b$ over $j\in b$ should be the
  set of all functions $s\colon f\inv(j) \to x$ such that $g\circ s =
  1_{f\inv(j)}$; this can be cut out of $x^{f\inv(j)}$ by
  \ddo-separation.  Finally, the entire set $\Pi_f(g)$ can be
  constructed from these and limited \ddo-replacement, since each
  element of each fiber is a subset of $a\times x$.

  The relationship between power sets and power objects is analogous.
  Likewise, \ddo-classical logic is equivalent to saying that every
  subset has a complement, which certainly implies Booleanness of
  $\bbSet(\bV)$.  The implication for full classical logic is also
  evident.

  Now suppose that \bV\ satisfies infinity and exponentials, and let
  $\omega$ be as in the axiom of infinity; we define $0\in \omega$ and
  $s\colon \omega\to\omega$ in the obvious way.  Suppose given
  $A\xto{g} B \xleftarrow{t} B$; we must construct a unique function
  $f\colon A\times \omega \to B$ such that $f(1_A\times 0) = g \pi_1$
  and $f(1_A\times s) = t f$.  Let $R = \setof{ (a,b)\in
    \omega\times\omega | a\in b}$ with projection $\pi_2\colon R\to
  \omega$, and let $X$ be the exponential $(B\times \omega \to
  \omega)^{(A\times R \to \omega)}$ in $\Set/\omega$; thus an element
  of $X$ is equivalent to an $n\in \omega$ together with a function
  $f\colon A \times \setof{ m | m \in n } \to B$.  Using
  \ddo-separation, we have the subset $Y\subseteq X$ of those $f\in X$
  such that $f(a,0) = g(a)$ and $f(a,s(m)) = t(f(a,m))$ whenever
  $s(m)\in n$.  We then prove by induction that for all $n\in\omega$
  there exists an $f\in Y$ with $n\in \mathrm{dom}(f)$ and that for
  any two such $f,f'$ we have $f(n)=f'(n)$.  The union of $Y$ is then
  the desired function.

  Finally, if \bV\ satisfies choice, then from any surjection $p\colon
  e \to b$ we can construct, using limited \ddo-replacement, the set
  $\setof{ p^{-1}(x) | x\in b}$, and applying the material axiom of
  choice to this gives a section of $p$.
\end{proof}

However, not every Heyting pretopos deserves to be called a model of
structural set theory.  The distinguishing characteristic of a
\emph{set}, as opposed to an object of some more general category, is
that a set is determined by its elements and nothing more.  This is
expressed by the following property.

\begin{defn}\label{def:wpt}
  A Heyting category \bS\ is \textbf{constructively well-pointed} if
  it satisfies the following.
  \killspacingtrue
  \begin{enumerate}[(a)]
  \item If $m\maps A\mono X$ is monic and every $1\too[x] X$ factors
    through $m$, then $m$ is an isomorphism \emph{(1 is a strong
      generator)}.\label{item:wpt-strgen}
  \item Every regular epimorphism $X\epi 1$ splits \emph{(1 is
      projective)}.\label{item:wpt-proj}
  \item If $1$ is expressed as the union of two subobjects $1 = A\cup
    B$, then either $A$ or $B$ must be isomorphic to $1$ \emph{(1 is
      indecomposable)}.\label{item:wpt-indec}
  \item There does not exist a map $1\to 0$
    \emph{(1 is nonempty)}.\label{item:wpt-nondeg}
  \end{enumerate}
  \killspacingfalse
\end{defn}


\begin{rmk}\label{rmk:nonelem-wpt}
  Of course, since regular epis are stable under pullback, $1$ is
  projective if and only if for any regular epi $Y\xepi{p} X$, every
  global element $1\to X$ lifts to $Y$.  Likewise, since unions are
  stable under pullback, $1$ is indecomposable if and only if whenever
  $X = A\cup B$, every global element $1\to X$ factors through either
  $A$ or $B$.  In particular, if \bS\ is locally small in some
  external set theory, then \bS\ is constructively well-pointed if and only if
  $\bS(1,-)\maps \bS\to\Set$ is a conservative coherent functor.
\end{rmk}

We immediately record the following.

\begin{prop}\label{thm:set-wpt}
  If \bV\ satisfies the core axioms of material set theory, then
  $\bbSet(\bV)$ is constructively well-pointed.
\end{prop}
\begin{proof}
  Functions $1\to X$ in $\bbSet(\bV)$ are in canonical correspondence
  with elements of $X$, and a function is monic just when it is
  injective.  Thus, if every map from $1$ factors through a monic, it
  must be bijective, and hence an isomorphism, so $1$ is a strong
  generator.  Next, the epics in $\bbSet(\bV)$ are the surjections,
  and if $X\xepi{p} 1$ is a surjection, then $X$ must be inhabited,
  hence $p$ splits.  And if $1 = A\cup B$, then the unique element of
  $1$ must be in either $A$ or $B$ by definition of unions, hence
  either $A$ or $B$ must be inhabited.  Finally, $\emptyset$ has no
  elements, so $1$ is nonempty.
\end{proof}

Recall that classically, a topos is said to be \emph{well-pointed} if
$1$ is a generator (i.e. for $f,g\maps X\toto Y$, having $fx=gx$ for
all $1\too[x] X$ implies $f=g$) and there is no map $1\to 0$ (1 is
nonempty).  Thus, any constructively well-pointed topos is
well-pointed in the classical sense.  Conversely, the following is
well-known (see, for instance,~\cite[VI.1 and VI.10]{mm:shv-gl}).

\begin{lem}\label{thm:wpt-bool}
  Let \bS\ be a Heyting category satisfying full classical logic, and
  in which $1$ is a nonempty strong generator.  Then \bS\ is Boolean
  and constructively well-pointed.
\end{lem}
\begin{proof}
  For any object $X$, either $X$ is initial or it isn't.  If it isn't,
  then there must be a global element $1\to X$, since otherwise the
  monic $0\mono X$ would be an isomorphism (since $1$ is a strong
  generator).  Therefore, if $X$ is not initial, then $X\to 1$ is
  split epic.  Now suppose that $X$ is any object such that $X\to 1$
  is regular epic; since $0\to 1$ is not regular epic, it follows that
  $X$ is not initial, and so $X\to 1$ must in fact be split epic; thus
  $1$ is projective.  Also, if $1 = A\cup B$, then $A$ and $B$ cannot
  both be initial, so one of them has a global element; thus $1$ is
  indecomposable.

  Now let $A\mono X$ be a subobject and assume that $A\cup \neg A$ is
  not all of $X$.  Then since $1$ is a strong generator, there is a
  $1\too[x] X$ not factoring through $A\cup \neg A$.  Let $V =
  x^*(A\cup \neg A)$.  Then $V$ is a subobject of $1$.  If $V$ is not
  initial, then it has a global element and hence is all of $1$, so
  $x$ factors through $A\cup \neg A$.  But if $V$ is initial, then
  $x^*(A)$ must also be initial, which implies that $x$ factors
  through $\neg A$.  This is a contradiction, so $A\cup \neg A = X$;
  hence \bS\ is Boolean.
\end{proof}

\begin{cor}
  Let \bS\ be a topos satisfying full classical logic, which is
  well-pointed in the classical sense.  Then \bS\ is Boolean and
  constructively well-pointed.
\end{cor}
\begin{proof}
  In a topos, any generator is a strong generator.
\end{proof}

Thus, in the presence of full classical logic, our notion of
constructive well-pointedness is equivalent to the usual notion of
well-pointedness.  However, \autoref{thm:set-wpt} shows that a
constructively well-pointed topos not satisfying full classical logic
need not be Boolean.  One can also construct examples showing that in
the absence of full classical logic, a Boolean topos in which $1$ is a
nonempty generator need not be \emph{constructively} well-pointed.  It
is true, however, even intuitionistically, that if $1$ is projective,
indecomposable, and nonempty in a Boolean topos, it must also be a
generator; see~\cite[V.3]{awodey:thesis}.


Other authors have also recognized the importance of explicitly
assuming projectivity and indecomposability of $1$ in an
intuitionistic context.  In~\cite{awodey:thesis} a topos in which $1$
is projective, indecomposable, and nonempty was called
\emph{hyperlocal} (but in~\cite{ptj:pp-bag-hloc,ptj:elephant} that
word is used for a stronger, non-elementary, property).  And
in~\cite{palmgren:cetcs}, indecomposability of $1$ is assumed
explicitly, while projectivity of $1$ is deduced from a factorization
axiom.  We note that most classical properties of a well-pointed topos
make use of projectivity and indecomposability of $1$, and many of
these are still true intuitionistically as long as the category is
\emph{constructively} well-pointed.

\begin{lem}\label{thm:pin-props}
  Let \bS\ be a constructively well-pointed Heyting category.  Then:
  \killspacingtrue
  \begin{enumerate}
  \item A morphism $p\colon Y\to X$ is regular epic if and only if
    every map $1\to X$ factors through it.\label{item:pp1}
  \item Likewise, $f\colon Y\to X$ is monic if and only if every map
    $1\to X$ factors through it in at most one way, and an isomorphism
    if and only if every map $1\to X$ factors through it
    uniquely.\label{item:pp1a}
  \item Given two subobjects $A\mono X$ and $B\mono X$, we have $A\le
    B$ if and only if every map $1\to X$ which factors through $A$
    also factors through $B$.\label{item:pp2}
  \item $X$ is initial if and only if there does not exist any
    morphism $1\to X$.\label{item:pp3}
  \end{enumerate}
  \killspacingfalse
\end{lem}
\begin{proof}
  The ``only if'' part of~\ref{item:pp1} follows because $1$ is
  projective.  For the ``if'' part, note that a map $p\colon Y\to X$
  in a regular category is regular epic iff its image is all of $X$,
  while if every $1\to X$ factors through $p$ then it must factor
  through $\mathrm{im}(p)$ as well; hence $\mathrm{im}(p)$ is all of
  $X$ since $1$ is a strong generator.

  The ``only if'' parts of~\ref{item:pp1a} are obvious.  If $f\colon
  Y\to X$ is injective on global elements, then the canonical
  monomorphism $Y\to Y\times_X Y$ is bijective on global elements, and
  hence an isomorphism; thus $Y$ is monic.  And if $f$ is bijective on
  global elements, then by this and by~\ref{item:pp1} it must be both
  monic and regular epic, hence an isomorphism.

  The ``only if'' part of~\ref{item:pp2} is also obvious.  If every $1\to
  X$ which factors through $A$ also factors through $B$, then in the
  pullback
  \[\vcenter{\xymatrix{A\cap B \pullbackcorner \ar[r]\ar[d] & B \ar[d]\\
      A\ar[r] & X}}
  \]
  the map $A\cap B\to A$ must be an isomorphism, since $1$ is a strong
  generator; hence its inverse provides a factorization of $A$ through
  $B$.

  Finally, the ``only if'' part of~\ref{item:pp3} is the nonemptiness
  assumption, while if $X$ has no global elements, then the map $0\to
  X$ is bijective on global elements and hence an isomorphism.
\end{proof}

\begin{rmk}
  In terms of the internal logic, indecomposability of $1$ corresponds
  to the \emph{disjunction property} (if $\ph\vee\psi$ is true, then
  either $\ph$ is true or $\psi$ is true), while projectivity of $1$
  corresponds to the \emph{existence property} (if $\exists x.\ph(x)$
  is true, then there is a particular $a$ such that $\ph(a)$ is true).
  We might also call nonemptiness of $1$ the \emph{negation property}
  ($\bot$ is not true).
\end{rmk}

\begin{rmk}
  (Constructive) well-pointedness has a very different flavor from all
  the other axioms of structural set theory we have encountered so
  far.  In particular, it is not preserved by slicing, or by most
  other categorical constructions.  This is to be expected, however,
  since ``being the category of sets'' is quite a fragile property.
\end{rmk}

We have now reached the point where we can define what we mean by a
\emph{structural set theory}: an extension of the axiomatic theory
describing a constructively well-pointed Heyting pretopos.  For
example:
\begin{blist}
\item Lawvere's \textbf{Elementary Theory of the Category of Sets
    (ETCS)}, defined in~\cite{lawvere:etcs}, is the theory of a
  well-pointed topos with a \nno\ satisfying full classical logic and
  the axiom of choice.
\item Palmgren's \textbf{Constructive ETCS (CETCS)}, defined
  in~\cite{palmgren:cetcs}, is the theory of a constructively
  well-pointed $\Pi$-pretopos with a \nno\ and enough projectives.
\item We will refer to the theory of a constructively well-pointed
  topos with a \nno\ as \textbf{Intuitionistic ETCS (IETCS)}.
\end{blist}

\begin{cnv}\label{cnv:elements}
  When working in a structural set theory, we usually write \Set\ for
  the category in question.  We call its objects \emph{sets} and its
  arrows \emph{functions}.  We speak of morphisms $1\to X$ as
  \emph{elements} of $X$, and we write $x\in X$ synonymously for
  $x\maps 1\to X$.  If $f\colon X\to Y$ is a function and $x\in X$, we
  write $f(x)$ for $f\circ x\colon 1\to Y$.  We speak of monomorphisms
  $S\mono X$ as \emph{subsets} of $X$ and write $S\subseteq X$.  For
  $x\in X$ and $S\subseteq X$, we write $x\in S$ and say ``$x$ is
  contained in $S$'' to mean that $1\too[x] X$ factors through $S\mono
  X$.  Similarly, for $S,T\subseteq X$ we write $S\subseteq T$ to mean
  that $S\mono X$ factors through $T\mono X$.  We also use the
  following terminology.
  \begin{blist}
  \item An arrow-variable $x\colon 1\to X$ whose domain is $1$ is a
    \emph{\ddo-variable}.
  \item An equality $(f=g)$ of arrow-terms $f,g\colon 1\to X$ with
    source $1$ is a \emph{\ddo-atomic formula}.
  \item A quantifier over a \ddo-variable is a \emph{\ddo-quantifier}.
  \item A formula whose only variables are \ddo-variables, whose
    only atomic subformulas are \ddo-atomic, and whose only
    quantifiers are \ddo-quantifiers is a \emph{\ddo-formula}.
  \end{blist}
\end{cnv}

These conventions make doing mathematics in a structural set theory
sound much more familiar.  They also make it sound very much like the
\emph{internal logic} of a category \bS.  Recall that the latter is a
type theory with types for all objects of \bS, function symbols for
all arrows in \bS, and relation symbols for all monics in \bS.  The
following lemma makes this correspondence precise.

\begin{lem}\label{thm:ddo-corresp}
  For any Heyting category \bS, there is a canonical bijection (up to
  provable equivalence) between
  \begin{enumerate}[(a)]
  \item \ddo-formulas in the language of categories with parameters in
    \bS, and\label{item:dc1}
  \item formulas in the internal first-order logic of
    \bS.\label{item:dc2}
  \end{enumerate}
\end{lem}
\begin{proof}
  Given a \ddo-formula as in~\ref{item:dc1}, we construct a formula in
  the internal logic by following formally the informal description in
  \ref{cnv:elements}.  Note that every arrow-term $1\to X$ must be
  constructed as a composite of some number of parameters, possibly
  beginning with a single \ddo-variable.  (The possibility of
  nontrivial terms of type $1\to 1$ can be excluded, up to provable
  equivalence.)  Thus, any such term such as $f\circ (g\circ x)\colon
  1\to X$ can be represented by a term $f(g(x)):X$ in the internal
  logic.  Similarly, \ddo-atomic formulas give atomic equalities
  between such terms, and connectives and \ddo-quantifiers translate
  in the obvious way.

  In the other direction, we have to do a little more work because the
  internal logic usually includes function-symbols of higher arity,
  corresponding to arrows in \bS\ whose domain is a cartesian product.
  Whenever a term such as $f(t_1,t_2):Y$ occurs, for terms $t_1:X_1$
  and $t_2:X_2$ and an arrow $f\colon X_1\times X_2\to Y$ in \bS, in
  the translation we introduce a new \ddo-variable $z\colon 1\to
  X_1\times X_2$ (in addition to those occurring in $t_1$ and $t_2$),
  represent $f(t_1,t_2)$ by $f\circ z$, and carry along an extra
  stipulation that $p_1\circ z = t_1$ and $p_2\circ z = t_2$, where
  $p_i\colon X_1\times X_2\to X_i$ are the product projections in \bS.
  This extra condition has then to be placed at an appropriate point
  in the resulting formula, but this process is familiar from the
  procedure of definitional extensions in first-order logic.

  The translation of atomic equalities between terms is obvious, but
  the internal logic includes relation symbols of arbitrary arity,
  with corresponding atomic formulas such as $R(x_1,x_2)$.  However,
  such an atomic formula can be replaced by the \ddo-formula $\exists
  z\in R. (p_1(z)=x_1 \meet p_2(z)=x_2)$, where $p_i\colon R\to X_i$ are
  the jointly monic projections of the relation $R$ in \bS.  Thus we
  can translate all atomic formulas, and the extension to connectives
  and quantifiers is straightforward.
\end{proof}

\begin{rmk}
  Note that in the translation of function symbols of higher arity,
  the object $X_1\times X_2$ is now a \emph{parameter} of the resulting
  \ddo-formula, as are $p_1$ and $p_2$.  Thus, we actually have to
  \emph{choose} some particular cartesian product of $X_1$ and $X_2$
  in \bS.  If \bS\ doesn't come with ``specified'' products, then this
  requires some axiom of choice to define the entire bijective
  correspondence.  But we will only need to apply the correspondence
  to particular formulas or finite sets of formulas, in which case
  there is no problem since we only need to make finitely many
  choices.
\end{rmk}

From now on we will identify the two types of formulas whose
equivalence is shown in \autoref{thm:ddo-corresp}.  The following
observation shows that not only do the formulas themselves correspond,
so do their ``extensions,'' in the senses appropriate to the internal
logic and to structural set theory respectively.

\begin{prop}\label{thm:ddo-sep}
  Let \bS\ be a category with finite limits in which $1$ is a strong
  generator.  Then the following are equivalent.
  \begin{enumerate}
  \item \bS\ is a constructively well-pointed Heyting category.\label{item:ddosep-phc}
  \item \bS\ satisfies the schema of \emph{\ddo-separation}: for any
    \ddo-formula $\ph(x)$ with free variable $x\colon 1\to X$, there
    exists a subobject $S\mono X$ such that any map $x\colon 1\to X$
    factors through $S$ if and only if $\ph(x)$
    holds.\label{item:ddosep-sep}
  \end{enumerate}
\end{prop}
\begin{proof}
  Assuming~\ref{item:ddosep-phc}, and given a \ddo-formula \ph\ as
  in~\ref{item:ddosep-sep}, we construct $S$ as the usual
  ``interpretation'' in the internal logic of \bS\ of the formula
  corresponding to \ph\ under \autoref{thm:ddo-corresp}.  That is, we
  build $S$ by induction on the structure of \ph, using intersections,
  unions, images, dual images, and so on in \bS.  That this satisfies
  the required property also follows from an inductive argument.  The
  cases of $\top$ and $\meet$ are clear, while $\imp$ and $\coim$
  follow from the assumption that $1$ is a strong generator.  Finally,
  the cases of $\im$, $\join$, and $\bot$ use the projectivity,
  indecomposability, and nonemptiness of $1$, respectively.

  Conversely, assume \ddo-separation~\ref{item:ddosep-sep} and that
  $1$ is a strong generator.  The argument of
  \autoref{thm:pin-props}\ref{item:pp1a} still applies to show that a
  morphism is monic iff it is injective on global elements.  Let \cM\
  denote the class of monics and \cE\ the class of morphisms that are
  surjective on global elements.  Because $1$ is a strong generator,
  any morphism in \cE\ is extremal epic (i.e.\ factors through no
  proper subobject of its codomain).  Moreover, \cE\ is evidently
  stable under pullback.

  For any map $f\colon Y\to X$, apply \ddo-separation to the formula
  $\exists y\in Y. (f(y)=x)$ to obtain a monic $m\colon S\mono X$.
  The pullback of $m$ along $f$ is a monic which is bijective on
  global elements, hence an isomorphism; thus $f$ factors through $m$,
  and the factorization $e\colon Y\to S$ is in \cE\ by construction.
  Therefore, $(\cE,\cM)$ is a pullback-stable factorization system,
  from which it follows that \bS\ is a regular category and that \cE\
  is exactly the class of regular epics---and hence $1$ is projective.

  Now given monics $m\colon A\mono X$ and $n\colon B\mono X$, apply
  \ddo-separation to the formula $\exists a\in A.(m(a)=x) \join
  \exists b\in B.(n(b)=x)$ to obtain a monic $S\mono X$.
  \autoref{thm:pin-props}\ref{item:pp2} still applies to show that
  $A\subseteq S$ and $B\subseteq S$ and that if $A\subseteq T$ and
  $B\subseteq T$ then $S\subseteq T$; thus $S=A\cup B$.  The defining
  property of $S$ (it contains precisely those $1\to X$ factoring
  through either $A$ or $B$) is evidently stable under pullback.
  Similarly, from the formula $\ph(x) = \bot$ we obtain a
  pullback-stable bottom element $0\mono X$, so \bS\ is a coherent
  category.  Moreover, by the construction of unions and empty
  subobjects, it follows that $1$ is indecomposable and nonempty.
  Finally, dual images can similarly be constructed by applying
  \ddo-separation to a formula with a universal \ddo-quantifier.
\end{proof}

Thus, the definition of a constructively well-pointed Heyting category, though it may seem
categorically technical, is equivalent to a very natural structural
analogue of the \ddo-separation axiom.


\begin{rmk}\label{thm:higher-ddo-sep}
  If \bS\ is a \Pi-pretopos or a topos, then the direction
  \ref{item:dc2}$\to$\ref{item:dc1} of \autoref{thm:ddo-corresp} can
  be extended to formulas in the appropriate higher-order logic
  involving dependent product types and/or power types, analogously to
  how we dealt with finite products, by choosing appropriate objects
  in \bS\ to introduce as extra parameters.  Therefore, the
  direction~\ref{item:ddosep-phc}$\to$\ref{item:ddosep-sep} of
  \autoref{thm:ddo-sep} can be applied to such formulas as well.  We
  will need this generalization in what follows.  Unfortunately, it
  seems difficult to precisely describe a class of formulas in the
  language of structural set theory which would satisfy higher-order
  versions of \autoref{thm:ddo-corresp} and \autoref{thm:ddo-sep}.
  This can be remedied by describing structural set theory as a type
  theory with quantification over types, as mentioned in
  footnote~\ref{fn:2} on page~\pageref{fn:2}.
\end{rmk}

\section{Structural separation, fullness, and induction}
\label{sec:more-sst}

There are several axioms from material set theory for which we have
not yet considered structural versions.  In this section, we consider
separation, fullness, and induction; in the next we consider
collection and replacement.

The structural separation axiom simply generalizes
\autoref{thm:ddo-sep}\ref{item:ddosep-sep} to unbounded quantifiers.
\begin{blist}
\item \emph{Separation:} For any formula $\ph(x)$, \qq{for any set
    $X$, there exists a subobject $S\mono X$ such that for any $x\in
    X$, we have $x\in S$ if and only if $\ph(x)$}.
\end{blist}

\begin{lem}
  If \bV\ satisfies the core axioms together with full separation,
  then $\bbSet(\bV)$ satisfies separation.
\end{lem}
\begin{proof}
  Note that any formula \ph\ in $\bbSet(\bV)$ in the language of
  categories may be translated into a formula \phhat\ in \bV\ in the
  language of material set theory.  Moreover, if \ph\ is \ddo, then so
  will \phhat\ be.  Now, if \bV\ satisfies material separation, it is
  easy to see that $S = \setof{x\in X | \phhat(x)}$ has the desired
  property for structural separation.
\end{proof}

The structural axiom of fullness is also a direct translation of the
material one.
\begin{blist}
\item \emph{Fullness:} \qq{for any sets $X,Y$ there exists a
    relation $R\mono M\times X\times Y$ such that $R\epi M\times X$ is
    regular epic, and for any relation $S\mono X\times Y$ such that
    $S\epi X$ is regular epic, there exists an $s\in M$ such that
    $(s,1)^*R \subseteq S$}.
\end{blist}

\begin{lem}
  If \bV\ satisfies the core axioms together with fullness, then
  $\bbSet(\bV)$ satisfies structural fullness.
\end{lem}
\begin{proof}
  Just like the proofs for exponentials and power sets in
  \autoref{thm:iz-topos}.
\end{proof}

\begin{rmk}
  The axiom of fullness appears quite different from our statements of
  the structural axioms of exponentiation and powersets in
  \S\ref{sec:struct-set}.  Specifically, the former refers explicitly
  to global elements and is thus only suitable in a well-pointed category,
  while the latter are phrased in a more ``category-theoretic'' way
  that makes sense in more generality.

  In fact, however, in a well-pointed category, the existence of power
  objects is easily shown to be equivalent to the following more
  ``set-like'' version: for any $X$ there is an object $P X$ and a
  relation $R_X \mono X\times P X$ such that for any subset $S\mono
  X$, there exists a unique $s\in P X$ such that $S \cong (1,s)^*R_X$.
  This is also true for exponentiation as long as we also assume the
  collection axioms from the next section.  Conversely, from our axiom
  of fullness one can derive a more category-theoretic version by
  interpreting it in the stack semantics.  So the difference is only
  an artifact of our chosen presentations.
\end{rmk}

The axiom of induction, of course, only makes sense in the presence of
the axiom of infinity.  The structural axiom of infinity (existence of
an \nno) asserts that functions can be constructed by recursion, which
implies Peano's induction axiom in the sense that any subset $S\mono
N$ which contains $o\colon 1\to N$ and is closed under $s\colon N\to
N$ must be all of $N$.  (The proof of
\autoref{thm:iz-topos}\ref{item:iz-topos-nno} essentially shows that
the converse holds in any \Pi-pretopos.)  It follows from
\ddo-separation that \ddo-formulas can be proven by induction; the
axiom of full induction extends this to arbitrary formulas.
\begin{blist}
\item \emph{Induction:} For any formula $\ph(x)$ with free variable
  $x\in N$, where $N$ is an \nno, \qq{if $\ph(0)$ and $\forall n\in
    N.(\ph(x)\imp \ph(s x))$, then $\ph(x)$ for all $x\in N$}.
\end{blist}
Just as in material set theory, infinity and full separation together
imply full induction.  We also have:

\begin{prop}
  If \bV\ satisfies the core axioms of material set theory, and also
  the axioms of infinity, exponentials, and induction, then
  $\bbSet(\bV)$ has an \nno\ and satisfies induction.
\end{prop}
\begin{proof}
  Any formula $\ph(x)$ in $\bbSet(\bV)$ with $x\in N = \omega$ can be
  rewritten as a formula in \bV, to which the material axiom of
  induction can be applied.
\end{proof}


There is no particularly natural structural ``axiom of foundation,''
although in \S\ref{sec:constr-mat} we will mention a somewhat related
property.  We can, however, formulate a structural axiom which is
closely related to the material axiom of set-induction.
\begin{blist}
\item \emph{Well-founded induction:} For any formula $\ph(x)$ with
  free variable $x\in A$, \qq{if $A$ is well-founded under the relation
  $\prec$, and moreover if $\ph(y)$ for all $y\prec x$ implies
  $\ph(x)$, then in fact $\ph(x)$ for all $x\in A$}.
\end{blist}
This axiom is implied by separation, since then we can form $\setof{
  x\in A | \ph(x) }$ and apply the definition of well-foundedness.  We
can also say:

\begin{prop}
  If \bV\ satisfies the core axioms of material set theory, and also
  the axioms of power set, set-induction, and Mostowski's principle,
  then $\bbSet(\bV)$ satisfies well-founded induction.
\end{prop}
\begin{proof}
  With power sets, any well-founded relation can be collapsed to a
  well-founded extensional one (an ``extensional quotient'' in the
  sense of \S\ref{sec:constr-mat}), which by Mostowski's principle
  will be isomorphic to a transitive set.  Thus, set-induction
  performed over the resulting transitive set can be used for
  inductive proofs over the original well-founded relation.
\end{proof}

\section{Structural collection and replacement}
\label{sec:strong-ax}

We now turn to structural versions of the collection and replacement
axioms.  Various such axioms have been proposed in the context of ETCS
(see~\cite{cole:cat-sets,osius:cat-setth,lawvere:etcs-long,mclarty:catstruct}),
but none of these seem to be quite appropriate in an intuitionistic or
predicative theory.  Hence our axioms must be different from all
previous proposals (though they are most similar to the replacement
axiom of~\cite{mclarty:catstruct}).

The intuition behind  structural collection is that
since the elements of a set in a structural theory are not
themselves sets, instead of ``collecting'' sets as \emph{elements} of
another set we must collect them as a family indexed \emph{over}
another set.  Also, since the language of category theory is
two-sorted, it is unsurprising that we have to assert collection for
objects and morphisms separately.  In fact, we find it conceptually
helpful to formulate \emph{three} axioms of collection, although the
third one is automatically satisfied.
\begin{blist}
\item \emph{Collection of sets:} For any formula $\ph(u,X)$, \qq{for
    any set $U$, if for every $u\in U$ there exists an $X$ with
    $\ph(u,X)$, then there exists a regular epi $V\xepi{p} U$ and an
    $A\in \Set/V$ such that for every $v \in V$ we have
    $\ph(pv,v^*A)$}.
\item \emph{Collection of functions:} For any formula $\ph(u,f)$,
  \qq{for any set $U$ and any $A,B\in\Set/U$, if for all $u\in U$
    there exists $u^*A\too[f] u^*B$ with $\ph(u,f)$, then there exists
    a regular epi $V\xepi{p} U$ and a function $p^*A\too[g] p^*B$ in
    $\Set/V$ such that for all $v\in V$, we have $\ph(pv,v^*g)$}.
\item \emph{Collection of equalities:} \qq{For any set $U$, any
    $A,B\in \Set/U$, and any functions $f,g\colon A\to B$ in $\Set/U$,
    if $u^*f = u^*g$ for every $u\in U$, then there is a regular epi
    $V\xepi{p} U$ such that $p^*f = p^*g$.}
\end{blist}
We will say that \Set\ ``satisfies collection'' if it satisfies all
three of these axioms.  However, one, and sometimes two, of these are
redundant.

\begin{prop}\label{thm:collmor}
    Let \Set\ be a constructively well-pointed Heyting category.  Then:
    \killspacingtrue
    \begin{enumerate}
    \item \Set\ satisfies collection of equalities.\label{thm:colleq}
    \item If \Set\ satisfies full separation and is a \Pi-pretopos,
      then it satisfies collection of functions.\label{item:collmor2}
    \end{enumerate}
    \killspacingfalse
\end{prop}
\begin{proof}
  In the situation of collection of equalities, let $E \xto{e} A$ be the
  equalizer of $f$ and $g$ and let $S = \coim_a E$, where $A\xto{a} U$
  is the structure map of $A\in \Set/U$.  The assumption that $u^*f =
  u^*g$ for every $1\too[u] U$ implies that every such $u$ factors
  through $S$.  Since $1$ is a strong generator, this implies $S\cong
  U$, and therefore $E\cong A$ and so $f=g$; thus we can take $V=U$
  and $p=1_U$.  This shows~\ref{thm:colleq}.

  In the situation of~\ref{item:collmor2}, let $C= (B\to U)^{(A\to
    U)}$ be the exponential in $\Set/U$, with projection $C \too[g]
  U$.  Then each $c\in C$ corresponds to a map $f_c\maps (gc)^*A \to
  (gc)^*B$.  Using the axiom of separation, find a subobject
  $V\xmono{m} C$ such that $c\in C$ is contained in $V$ iff
  $\ph(f_c)$.  By assumption, every $u\in U$ lifts to some $c\in V$,
  so since \Set\ is well-pointed, the projection $V\xepi{gm} U$ is regular
  epi.  Finally, there is an evident map $h\maps (gm)^*A\to (gm)^*B$
  such that $\ph(gmc,c^*h)$ for all $c\in V$.
\end{proof}


The appropriate structural formulation of the axiom of
\emph{replacement} is a bit more subtle than that of collection.  The
idea is that if we modify the hypotheses of collection by asserting
\emph{unique} existence, then passage to a cover $V\epi U$ should be
unnecessary.  As with collection, we may expect three versions for
sets, functions and equalities, and of these the second two are easy
and follow from collection.

\begin{prop}\label{thm:replacement}
  Let \Set\ be a constructively well-pointed Heyting category.
  \killspacingtrue
  \begin{enumerate}
  \item \Set\ always satisfies \emph{replacement of equalities}: \qq{for
      any set $U$, any $A,B\in \Set/U$, and any functions $f,g\colon
      A\to B$ in $\Set/U$, if $u^*f = u^*g$ for every $u\in U$, then
      $f=g$}.\label{item:repl-eq}
  \item If \Set\ satisfies collection of functions, then it satisfies
    \emph{replacement of functions}: \qq{for any set $U$ and any $A,B\in
      \Set/U$, if for every $1\too[u] U$ there exists a unique $u^*A
      \too[f] u^*B$ such that $\ph(u,f)$, then there exists $A\too[g] B$
      in $\Set/U$ such that for each $1\too[u] U$ we have
      $\ph(u,u^*f)$}.\label{item:repl-mor}
  \end{enumerate}
  \killspacingfalse
\end{prop}
\begin{proof}
  The proof of \autoref{thm:collmor}\ref{thm:colleq} already
  shows~\ref{item:repl-eq}.  For~\ref{item:repl-mor}, collection of
  functions gives us a regular epimorphism $V\xepi{p} U$ and a
  function $h\colon p^*A \to p^*B$ in $\Set/V$ such that $\ph(p v,
  v^*h)$ for any $v\in V$.  If $(r,s)\colon V\times_U V\toto V$ is the
  kernel pair of $p$, then for every $z = (v_1,v_2)\in V\times_U V$ we
  have $p v_1 = p v_2 = u$, say, and thus $\ph(u, v_1^*h)$ and $\ph(u,
  v_2^*h)$.  By uniqueness, $v_1^*h = v_2^*h$, and so
  \[z^*r^*h = v_1^* h = v_2^* h = z^* s^* h.
  \]
  By replacement of equalities, $r^*h = s^*h$.  So since the
  self-indexing of \Set\ is a prestack, $h$ descends to $g\colon A\to
  B$ with $p^*g = h$.  Since $p$ is regular epic, for any $u\in U$
  there exists a $v\in V$ with $p v = u$, so we have $u^* h = v^* p^*
  h = v^* g$, whence $\ph(u,u^*h)$.
\end{proof}

These two replacement axioms imply, in particular, that
\emph{universal properties are reflected by global elements}.  Rather
than give a precise statement of this, we present a paradigmatic
example.

\begin{prop}\label{thm:cc-lcc-mor}
  Let \Set\ be a constructively well-pointed Heyting category satisfying replacement of
  equalities and functions, and suppose we have $A,B,E\in \Set/X$ and
  a morphism $v\colon E\times_X A\to B$ such that for all $x\in X$,
  $x^*v$ exhibits $x^*E$ as an exponential $(x^*B)^{(x^*A)}$ in \Set.  Then
  $v$ exhibits $E$ as an exponential $B^A$ in $\Set/X$.
\end{prop}
\begin{proof}
  The assumption means that for any object $C$ and any morphism
  $f\colon C\times x^*A \to x^*B$, there exists a unique $g\colon C\to
  x^*E$ such that $f = x^*v \circ (g\times 1)$.  In particular, given
  any $D\in \Set/X$ and morphism $f\colon D\times_X A\to B$, for every
  $x\in X$ there is a unique $g_x\colon x^*D\to x^*E$ such that $x^*f
  = x^*v \circ (g_x\times 1)$.  By replacement of functions, we have
  $g\colon D\to E$ such that $x^*f = x^*v \circ (x^*g \times 1)$ for
  each $x\in X$, or equivalently $x^* f = x^*(v\circ (g\times 1))$.
  By replacement of equalities, $f=v\circ (g\times 1)$.  Moreover, if
  $f = v\circ (h\times 1)$ for some $h\colon D\to E$, then for each
  $x\in X$ we have $x^*f = x^*v \circ (x^*h \times 1)$.  By the
  universal property of the exponential $x^*E$, we have $x^*h = g_x =
  x^*g$, and hence $h=g$ by replacement of equalities.
\end{proof}

In future, we will invoke similar facts frequently, trusting the
reader to supply analogous arguments as necessary.  Note also that all
universal properties reflected in this way are automatically
pullback-stable.

We would also like an axiom of ``replacement of sets,'' which would
ensure that the \emph{existence} of an object satisfying a universal
property is also reflected by global elements.  Since no structural
theory can determine an object of a category more uniquely than up to
unique isomorphism, one natural such statement would be:
\begin{blist}
\item \emph{Replacement of sets:} \qq{if for every $u\in U$ there is a
    set $A$ with $\ph(u,A)$ which is unique up to unique isomorphism,
    then there is a $B\in \Set/U$ such that $\ph(u,u^*B)$ for all
    $u\in U$}.
\end{blist}
If \Set\ is additionally \emph{exact}, then replacement of sets
follows from collection, using the fact that in this case the
self-indexing $\Set/(-)$ a stack for the regular topology.  Moreover,
just as in material set theory, we have:

\begin{prop}\label{thm:collsep}
  If \bS\ is a well-pointed Heyting category which satisfies full
  classical logic and replacement of sets, then it also satisfies
  separation.
\end{prop}
\begin{proof}
  For any formula $\ph$ with a free variable $1\too[x] X$, let
  $\psi(x,A)$ be \qq{$A$ is terminal and $\ph(x)$, or $A$ is initial
    and $\neg\ph(x)$}.  Since $\forall x.(\ph(x)\vee\neg\ph(x))$ holds
  by classical logic, and initial and terminal objects are unique up
  to unique isomorphism, we can apply replacement of sets to obtain a
  a $B\in \Set/X$ such that for any $x\in X$, if $\ph(x)$ then $x^*B$
  is terminal, while if $\neg\ph(x)$ then $x^*B$ is initial.  It
  follows that $B\to X$ is monic, and is the subobject classifying
  $\ph$ required by the separation axiom.
\end{proof}


However, for most other purposes replacement of sets is basically
useless, since hardly ever is a set determined up to \emph{absolutely}
unique isomorphism, only an isomorphism which is compatible with some
additional functions whose existence is being asserted at the same
time.  For example, a cartesian product $A\times B$ is only determined
up to an isomorphism which is unique \emph{such that} it respects the
projections $A\times B\to A$ and $A\times B\to B$ with which the
product comes equipped.  In order to state a useful version of
replacement, we invoke the notion of \emph{context} from dependent
type theory (which also simplifies greatly in our particular
situation).

\begin{defn}
  Let \bS\ be a category.
  \killspacingtrue
  \begin{enumerate}
  \item A \textbf{context} \Gm\ in \bS\ is a finite (ordered) list of
    typed variables, with the property that whenever an arrow-variable
    $f\colon X\to Y$ occurs in \Gm, each of $X$ and $Y$ is either
    a parameter, or an object-variable occurring in \Gm\ prior to $f$.
  \item A \textbf{instantiation} $\vec{A}$ of a context \Gm\ in \bS\
    is a well-typed assignation of objects and arrows in \bS\ to the
    variables of \Gm.  We write $\vec{A}:\Gm$ and use the notations
    $X\mapsto X_A$ and $(f\colon Y\to Z) \mapsto (f_A\colon Y_A\to
    Z_A)$, where if $Y$ or $Z$ is a parameter then $Y_A$ or $Z_A$
    denotes simply that parameter.
  \item A \textbf{morphism} of instantiations
    $\vec{A}\to\vec{B}$ consists of
    \begin{enumerate}
    \item For each object-variable $X$ in \Gm, a morphism
      $\al_X\colon X_A\to X_B$, such that
    \item For each arrow-variable $f\colon X\to Y$ in \Gm, we have
      $\al_Y \circ f_A = f_B \circ \al_X$.  Here if $X$ or $Y$ is a
      parameter, then $\al_X$ or $\al_Y$ denotes the identity arrow of
      that parameter.
    \end{enumerate}
  \item If $(\Gm,\De)$ is a context (meaning the concatenation of \Gm\
    and \De), and $\vec{A}$ is an instantiation of \Gm, then an
    \textbf{extension} of $\vec{A}$ to \De\ is a $\vec{B}$ such that
    $(\vec{A},\vec{B})$ is an instantiation of \De.  Similarly we
    define extensions of morphisms.
  \end{enumerate}
  \killspacingfalse
\end{defn}

Of course, instantiations of a given, fixed, context can be discussed
(and quantified over) in the language of categories over \bS.  If \ph\
is a formula whose free variables are those in the context \Gm, we
write $\ph(\Gm)$, and similarly $\ph(\vec{A})$ for the instantiation
of \ph\ with these variables replaced by the parameters in $\vec{A}$.
The pullback of contexts and instantiations along a morphism $p\colon
V\to U$ is defined in the obvious way.

As a particularly important example, the context $(X,x\colon 1\to X)$
can be instantiated in $\Set/U$ by the pair $(U^*U,\Delta_U)$
consisting of the object $U^*U = U\times U$ and the morphism
$\Delta_U\colon 1_U\to U^*U$.  Moreover, this pair is the
\emph{universal} such instantiation over $U$, in that for any $u\colon
1\to U$ we have an isomorphism of instantiations $u^*(U^*U,\Delta_U)
\cong (U,u)$.  Similarly, for any $V\too[p] U$ we can consider
$(V^*U,(1_V,p))$, which has the analogous property that
$v^*(V^*U,(1_V,p)) \cong (U, p v)$ for any $v\colon 1\to V$.

\begin{prop}\label{thm:coll-ctxt}
  Any constructively well-pointed Heyting category satisfying
  collection also satisfies \emph{collection of contexts:} for any
  context $(X,x\colon 1\to X,\Gm)$ and any formula $\ph(X,x,\Gm)$,
  \qq{for any set $U$, if for every $u\in U$ there exists
    $\vec{A}:\Gm$ extending $(U,u)$ such that $\ph(U,u,\vec{A})$, then
    there exists a regular epi $V\xepi{p} U$ and a $\vec{B}:\Gm$
    extending $(V^*U,(1_V,p))$ in $\Set/V$ such that for every $v \in
    V$ we have $\ph(U,pv,v^*\vec{B})$}.
\end{prop}
\begin{proof}
  Simply apply collection of sets and functions repeatedly, using the
  fact that regular epimorphisms compose.
\end{proof}

We can now improve ``replacement of sets'' to a more useful axiom.
\begin{blist}
\item \emph{Replacement of contexts:} for any context $(X, x\colon
  1\to X, \Gm)$ and any formula $\ph(X,x,\Gm)$, \qq{for any set $U$,
    if for every $u\in U$ there is an extension $\vec{A}:\Gm$ of
    $(U,u)$ such that $\ph(U,u,\vec{A})$, and moreover $\vec{A}$ is
    unique up to a unique isomorphism of instantiations extending
    $1_U$, then there is an extension $\vec{B}:\Gm$ of $(U^*U,
    \Delta_U)$ in $\Set/U$ such that $\ph(U,u,u^*\vec{B})$ for all
    $u\in U$}.
\end{blist}

\begin{prop}\label{thm:repl-ctxt}
  Any constructively well-pointed Heyting pretopos satisfying
  collection also satisfies replacement of contexts.
\end{prop}
\begin{proof}
  By \autoref{thm:coll-ctxt}, there is a regular epi $p\colon V\epi U$
  and a $\vec{C}:\Gm$ in $\Set/V$ such that for every $v \in V$ we
  have $\ph(U,pv,v^*\vec{C})$.  Let $(r,s)\colon V\times_U V \toto V$ be
  the kernel pair of $f$, and consider $r^*\vec{C}$ and $s^*\vec{C}$
  in $\Set/V\times_U V$.  The assumption implies that when pulled back
  along any $z\colon 1\to V\times_U V$, these two instantiations of
  \Gm\ become uniquely isomorphic.  Thus, by replacement of functions
  and equalities applied some finite number of times, we actually have
  a unique isomorphism $r^*\vec{C}\cong s^*\vec{C}$ in $\Set/V\times_U
  V$.  Uniqueness implies that this isomorphism satisfies the cocycle
  condition over $V\times_U V\times_U V$.  Thus, since \Set\ is a
  stack for its regular topology (applied some finite number of
  times), the entire instantiation $\vec{C}:\Gm$ descends to some
  $\vec{B}:\Gm$ in $\Set/U$.  Since $p$ is surjective, the desired
  property for $\vec{B}$ follows.
\end{proof}

In general, it seems that replacement of contexts does not imply
collection, although we do not have a counterexample.  As in material
set theory, however, replacement suffices in the classical world.

\begin{prop}\label{thm:class-repcoll}
  Over ETCS, the following are equivalent.
  \begin{enumerate}
  \item Replacement of contexts.
  \item Collection of sets.
  \item Full separation and collection (of both sets and functions).
  \end{enumerate}
\end{prop}
\begin{proof}
  By Propositions \ref{thm:collmor}\ref{item:collmor2},
  \ref{thm:collsep}, and \ref{thm:repl-ctxt}, it suffices to prove
  that replacement of contexts implies collection of sets.  Given a
  formula $\ph(u,X)$, we can construct a context \Gm\ including both a
  set $X$ and a binary relation on it, and let $\psi(u,\Gm)$ assert
  that $\ph(u,X)$ holds and that the given relation is a
  well-ordering, which has the smallest possible order-type among
  well-ordered sets $X$ such that $\ph(u,X)$.  Since classically, any
  set admits a smallest well-ordering which is unique up to unique
  isomorphism, we can apply replacement of contexts to $\psi$ and
  thereby deduce collection of sets.
\end{proof}

As promised, replacement of contexts implies that \emph{the existence
  of objects satisfying a universal property is reflected by global
  elements}.  Continuing the example of \autoref{thm:cc-lcc-mor}, we
have the following.

\begin{prop}\label{thm:cc-lcc}
  Let \Set\ be a constructively well-pointed Heyting pretopos
  satisfying collection.  If \Set\ is cartesian closed, then it is
  locally cartesian closed.
\end{prop}
\begin{proof}
  Let $A,B\in\Set/X$; we want to construct the exponential $(B\to
  X)^{(A\to X)}$.  For $x\in X$, let $\ph(x,E,e)$ assert that $e\colon
  E\times x^*A \to x^*B$ exhibits $E$ as the exponential
  $(x^*B)^{x^*A}$.  This formula can be phrased as $\psi(X,x,\Gm)$ for
  some \Gm, although we note that \Gm\ involves not just $E$ and $e$
  but also object-variables for $x^*A$, $x^*B$, and $E\times x^*A$,
  and arrow-variables giving the projections that exhibit these as two
  pullbacks and a cartesian product, respectively.

  Now, since \Set\ is cartesian closed, for every $x\in X$ there
  exists $E$ and $e$ with $\ph(x,E,e)$, and such are unique up to a
  unique isomorphism which respects all the structure.  Therefore,
  replacement of contexts supplies an instantiation of \Gm\ extending
  $(X^*X, \Delta_X)$ in $\Set/X$.  This consists essentially of an
  object $C$ and a morphism $c\colon C\times_X A \to B$.  The
  conclusion of replacement of contexts implies that $C$ and $c$
  satisfy the hypotheses of \autoref{thm:cc-lcc-mor}, so they must
  be an exponential $(B\to X)^{(A\to X)}$ as desired.
\end{proof}




This is of course reminiscent of the way in which local cartesian
closure comes ``for free'' in
\autoref{thm:iz-topos}\ref{item:iz-topos-lcc}.  We end this section by
extending \autoref{thm:iz-topos} to the axiom of collection.



\begin{thm}\label{thm:iz-ext-seprepcoll}
  If \bV\ satisfies the core axioms of material set theory along with
  collection, then $\bbSet(\bV)$ also satisfies collection.
\end{thm}
\begin{proof}
  If \bV\ satisfies material collection, then given the setup of
  collection of sets for some formula $\ph$, let $\psi(u,X)$ assert
  that $X$ is a Kuratowski ordered pair of the form $(u,X')$ with
  $\ph(u,X')$.  By material collection, let $V$ be a set such that for
  any $u\in U$, there is an $X\in V$ with $\psi(u,X)$, and for any
  $X\in V$, there is a $u\in U$ with $\psi(u,X)$.  Then every element
  of $V$ is an ordered pair whose first element is in $U$, so there is
  a projection $V\to U$, which by assumption is surjective.  We can
  then form $B = \bigcup_{(u,X')\in V} X'$, and by \ddo-separation cut
  out
  \[ A = \setof{((u,X'),a) \in V\times B | a \in X'}.
  \]
  Then $A \in \bbSet(\bV)/V$, and for any $v=(u,X')\in V$, $v^*A$ is
  isomorphic to $X'$.  Thus, by isomorphism-invariance, we have
  $\ph(u,v^*A)$, as required.  Collection of functions is analogous.
\end{proof}

The material axiom of replacement, however, does not seem to imply the
structural one.  This, we feel, is one of the reasons (though not the
only one) that the material axiom of collection is necessary in
practice.



\section{Constructing material set theories}
\label{sec:constr-mat}

So far we have summarized the axioms of material set theory and
structural set theory, and explained how most axioms of material set
theory are reflected in structural properties of the resulting
category of sets.  We now turn to the opposite construction: how to
recover a material set theory from a structural one.  A ``set'' in
material set theory, of course, contains much more information than a
``set'' in structural set theory, namely the membership relations
between its elements, their elements, and so on.  This gives rise to
the idea of modeling a ``material set'' by a \emph{graph} or
\emph{tree} with nodes depicting sets and edges depicting membership.

This basic idea was used
by~\cite{cole:cat-sets,mitchell:topoi-sets,osius:cat-setth} in the
first equiconsistency proofs of ETCS with BZC, and nearly identical
constructions have been used for relative consistency proofs by others
such as~\cite{aczel:afa} and~\cite{mathias:str-maclane}.  We are not
aware of an exposition at our required level of generality, though.
The constructions are the same as always, though, as long as we are
careful about the definitions.

For the rest of this section, let \Set\ be a constructively
well-pointed \Pi-pretopos with an \nno.

\begin{defns}\ 
  \killspacingtrue
  \begin{enumerate}
  \item A \textbf{graph} is a set $X$, whose elements are called
    \textbf{nodes}, equipped with a binary relation $\prec$.
  \item If $x\prec y$ we say that $x$ is a \textbf{child} of $y$.
  \item A \textbf{pointed} graph is one equipped with a distinguished
    node $\star$ called the \textbf{root}.
  \item A pointed graph is \textbf{accessible} if for every node $x$
    there exists a path $x = x_n \prec \cdots \prec x_0 = \star$ to
    the root.
  \item For any node $x$ of a graph $X$, we write $X/x$ for the full
    subgraph of $X$ consisting of those $y$ admitting some path to
    $x$.  It is, of course, pointed by $x$, and accessible.
  \end{enumerate}
  \killspacingfalse
  ``Accessible pointed graph'' is abbreviated \apg.
\end{defns}

\begin{rmk}
  The hypothesis on \Set\ is necessary to formalize ``accessibility''
  and to define $X/x$.  Specifically, a pointed graph $X_{\prec} \mono
  X\times X$ with root $\star\colon 1\to X$ is accessible if for every
  $x\colon 1\to X$, there exists a nonzero finite cardinal $[n]$ and a
  map $[n] \to X_{\prec}$ realizing a path from $x$ to $\star$.
  (Recall that a finite cardinal is the pullback along a map $1\to N$
  of the second projection $N_< \to N$ of the strict order relation on
  the \nno.)  The definition of $X/x$ is similar, using the extension
  of \ddo-separation to functions described in \autoref{thm:higher-ddo-sep}.
\end{rmk}

We are using the terminology of~\cite{aczel:afa}.  The idea is that an
arbitrary graph represents a collection of material-sets with $\prec$
representing the membership relation between them, a pointed graph
represents a \emph{particular} set (the root) together with all the
data required to describe its hereditary membership relation, and an
\apg\ does this without any superfluous data (all the nodes bear some
relation to the root).  Thus, an arbitrary \apg\ can be considered a
picture of a possibly non-well-founded set: the root represents the
set itself, its children represent the elements of the set, and so on.

We will henceforth restrict attention to graphs for which $\prec$ is
well-founded, thereby ensuring that the models of material set theory
we construct satisfy the axiom of foundation.  It is also possible, by
changing this requirement, to construct models satisfying the various
axioms of anti-foundation (see~\cite{aczel:afa}), but we will not do
that here.  Recall the definition:

\begin{defn}
  A subset $S$ of a graph $X$ is \textbf{inductive} if for any node
  $x\in X$, if all children of $x$ are in $S$, then $x$ is also in
  $S$.  A graph $X$ is \textbf{well-founded} if any inductive subset
  of $X$ is equal to all of $X$.
\end{defn}

In the presence of classical logic, well-foundedness is equivalent to
saying that every nonempty subset of $X$ has a $\prec$-least element,
but in intuitionistic logic that version is both fairly useless and
rarely satisfied.

With the above definition of well-foundedness, we can perform proofs
by induction on a well-founded graph: by proving that a given subset
is inductive, we conclude it is the whole graph.  In order to prove a
\emph{statement} by well-founded induction, we must apply separation
to the statement to turn it into a subset; thus either the statement
must be \ddo\ or we must have a stronger separation axiom.  We can
also use well-founded induction in more than one variable, since if
$X$ and $Y$ are well-founded then so is $X\times Y$.

We record some observations about well-founded graphs.

\begin{lem}\label{thm:wf-hered}
  Any subset of a well-founded graph is well-founded with the induced
  relation.
\end{lem}
\begin{proof}
  For any $Z\subseteq X$, if $S\subseteq Z$ is inductive in $Z$, then
  the Heyting implication $(Z\imp S) \subseteq X$ is inductive in $X$.
\end{proof}

\begin{lem}\label{thm:wf-no-cycles}
  If $X$ is a well-founded graph, then there does not exist any cyclic
  path $x = x_n \prec \cdots \prec x_0 = x$ in $X$.
\end{lem}
\begin{proof}
  If there were, then the Heyting complement of $\{ x_0, \dots,
  x_n\}\subseteq X$ would be inductive but not all of $X$.
\end{proof}

\begin{lem}\label{thm:root-decid}
  If $X$ is a well-founded graph and $x\in X$, then $x\colon 1\to X/x$
  is a complemented subobject.  In particular, it is decidable whether
  or not a node of a well-founded \apg\ is the root.
\end{lem}
\begin{proof}
  By definition, every node $y\in X/x$ admits some path to $x$, of
  length $n\in N$ say.  Since $0$ is a complemented subobject of $N$,
  either $n=0$, in which case $y=x$, or $n>0$.  If $n>0$, then if
  $x=y$ there would be a cyclic path from $x$ to itself; hence in this
  case $x\neq y$.  Thus for all $y$, either $y=x$ or $y\neq x$, as
  desired.
\end{proof}

\begin{defn}
  We will write $X\csl x$ for the complement of $x$ in $X/x$, i.e.\
  the set of nodes admitting a path to $x$ of length $>0$.  In
  particular, $X\csl \star$ is the complement of the root.
\end{defn}

We need one more requirement on our \apgs, namely that they satisfy
the axiom of extensionality.

\begin{defn}
  An graph $X$ is \textbf{extensional} if whenever $x$ and $y$ are
  nodes such that $z\prec x \Leftrightarrow z\prec y$ for all $z$,
  then $x=y$.
\end{defn}

For non-well-founded graphs, this definition would have to be
strengthened in one of various possible ways.

\begin{rmk}
  An equivalent characterization of the universe of extensional
  well-founded \apgs\ can be obtained by working with well-founded
  \emph{rigid trees} instead.  A \emph{tree} is an \apg\ in which
  every node $x$ admits a \emph{unique} path to the root, and it is
  \emph{rigid} if for any node $z$ and any children $x\prec z$ and
  $y\prec z$, if $X/x \cong X/y$, then $x=y$.
  Every extensional well-founded \apg\ $X$ has an ``unfolding'' into a
  well-founded rigid tree $X^t$, whose nodes are the paths in $X$, and
  conversely every well-founded rigid tree is the unfolding of some
  extensional well-founded \apg.  Rigid trees are used
  by~\cite{cole:cat-sets,mitchell:topoi-sets} and~\cite{mm:shv-gl},
  while extensional relations are used by~\cite{osius:cat-setth}
  and~\cite{ptj:topos-theory}, but they result in essentially
  equivalent theories.  (In the non-well-founded case, however,
  extensional relations seem more generally applicable than rigid
  trees.)
\end{rmk}

We now record some observations about extensionality and its
interaction with well-foundedness.

\begin{defn}
  An \textbf{initial segment} of a graph $X$ is a subset $Y\subseteq
  X$ such that $x\prec y \in Y$ implies $x\in Y$.
\end{defn}

\begin{lem}\label{thm:ext-hered}
  Any initial segment of an extensional graph is extensional with the
  induced relation.\qed
\end{lem}

\begin{lem}\label{thm:wf-rigid}
  If $X$ is a well-founded extensional graph and $X/x \cong X/y$, then
  $x=y$.
\end{lem}
\begin{proof}
  We prove this by well-founded induction.  If $g\colon X/x\toiso
  X/y$, then for all $x'\prec x$ we have $X/x' \cong X/g(x')$, where
  $g(x')\prec y$; hence by induction $x' = g(x')$.  Similarly, we have
  $y' = g^{-1}(y')$ for all $y'\prec y$.  By extensionality, $x=y$.
\end{proof}

In the language of~\cite{aczel:afa}, \autoref{thm:wf-rigid} says that
well-founded extensional graphs are ``Finsler-extensional.''

\begin{lem}\label{thm:wf-rigid-global}
  Any automorphism of a well-founded extensional graph is the
  identity.  Therefore, any two parallel isomorphisms of well-founded
  extensional graphs are equal.
\end{lem}
\begin{proof}
  Let $f\colon X\toiso X$ be an isomorphism; we prove by well-founded
  induction that $f(x)=x$ for all $x\in X$.  But if $f(x')=x'$ for all
  $x'\prec x$, then extensionality immediately implies $f(x)=x$.
\end{proof}

\begin{defn}
  We write $\bbV(\Set)$ for the class of well-founded extensional
  \apgs\ in \Set.
\end{defn}

Our goal is now to show that $\bbV(\Set)$ is a model of a material set
theory.  We consider two \apgs\ to be \emph{equal}, i.e.\ to represent
the same material-set, when they are isomorphic.  And we model the
membership relation $\ordin$ in the expected way:

\begin{defn}
  If $X$ is an \apg, the children of its root are called its
  \textbf{members}.  We write $|X|$ for the set of members of $X$.  If
  $X$ and $Y$ are \apgs, we write $X\iin Y$ to mean that $X\cong Y/y$
  for some member $y\in |Y|$.
\end{defn}

Here is our omnibus theorem.

\begin{thm}\label{thm:str->mat}
  Let \Set\ be a constructively well-pointed \Pi-pretopos with a \nno,
  and let \ph\ be a formula of material set theory with parameters in
  $\bbV(\Set)$.  Then $\bbV(\Set)\ss\ph$ whenever any of the following
  holds.
  \begin{blist}
  \item \ph\ is the axiom of extensionality, empty set, pairing,
    union, exponentiation, infinity, foundation, transitive closures,
    Mostowski's principle, or an instance of \ddo-separation.
  \item \Set\ satisfies structural fullness and \ph\ is material fullness.
  \item \Set\ is a topos and \ph\ is the power set axiom.
  \item \Set\ satisfies structural separation and \ph\ is an instance
    of material separation.
  \item \Set\ satisfies structural collection and \ph\ is an instance
    of material collection.
  \item \Set\ satisfies replacement of contexts and \ph\ is an instance
    of material replacement.
  \item \Set\ is Boolean and \ph\ is an instance of \ddo-classical
    logic.
  \item \Set\ satisfies full classical logic and \ph\ is an instance
    of full classical logic.
  \item \Set\ satisfies induction and \ph\ is an instance of full
    induction.
  \item \Set\ satisfies the axiom of choice and \ph\ is the material
    axiom of choice.
  \item \Set\ satisfies well-founded induction and \ph\ is an instance
    of set-induction.
  \end{blist}
\end{thm}

We will prove this theorem with a series of lemmas, but first we need
to introduce some auxiliary notions leading up to the construction of
\emph{extensional quotients}.

\begin{defn}
  Let $X$ and $Y$ be graphs.  A \textbf{simulation} from $X$ to $Y$ is
  a function $f\colon X\to Y$ such that
  \begin{enumerate}
  \item if $x'\prec x$, then $f(x')\prec f(x)$, and
  \item if $y\prec f(x)$, then there exists an $x'\prec x$ with $f(x')=y$.
  \end{enumerate}
  A \textbf{bisimulation} from $X$ to $Y$ is a relation $R\mono
  X\times Y$ such that both projections $R\to X$ and $R\to Y$ are
  simulations (where $R$ is considered as a full subgraph of $X\times
  Y$).  A bisimulation is \textbf{bi-entire} if $R\to X$ and $R\to Y$
  are surjective.
\end{defn}

Note that if $f\colon X\to Y$ is a simulation, then the relation
$(1,f)\colon X\to X\times Y$ is a bisimulation, which is bi-entire iff
$f$ is surjective.

An obvious example of a simulation is the inclusion of an initial
segment.  The following lemma says that for well-founded extensional
graphs, these are the only simulations.

\begin{lem}\label{thm:sim-inj}
  If $X$ is well-founded and extensional, then any simulation $f\colon
  X\to Y$ is injective, and isomorphic to the inclusion of an initial
  segment.
\end{lem}
\begin{proof}
  We show by well-founded induction that $f(x_1)=f(x_2)$ implies
  $x_1=x_2$.  Suppose $f(x_1)=f(x_2)$; then for any $z_1\prec x_1$, we
  have $f(z_1)\prec f(x_1) = f(x_2)$, so since $f$ is a simulation
  there is a $z_2\prec x_2$ with $f(z_2)=f(z_1)$.  Hence $z_1=z_2$ by
  induction, so $z_1\prec x_2$.  Dually, $z_2\prec x_2$ implies
  $z_2\prec x_1$, so by extensionality $x_1=x_2$.  Finally, any
  injective simulation must be an initial segment.
\end{proof}

\begin{lem}\label{thm:bisim-wf}
  If $R$ is a bi-entire bisimulation from $X$ to $Y$, then $X$ is
  well-founded if and only if $Y$ is so.
\end{lem}
\begin{proof}
  Suppose $X$ is well-founded and $S\subseteq Y$ is inductive.  Let
  $T\subseteq X$ consist of those $x\in X$ such that $R(x,y)$ implies
  $y\in S$; we show that $T$ is inductive.  Suppose $x\in X$ is such
  that $x'\prec x$ implies $x'\in T$, and suppose that $R(x,y)$.  Then
  for any $y'\prec y$ there is an $x'\prec x$ with $R(x',y')$, whence
  $x'\in T$ and thus $y'\in S$.  So since $S$ is inductive, $y$ must
  be in $S$.  Thus $x\in T$, so $T$ is inductive.  Since $X$ is
  well-founded, $T=X$, and then since $R$ is bi-entire, $S=Y$; thus
  $Y$ is well-founded.
\end{proof}

\begin{lem}\label{thm:bisim-iso}
  Any bi-entire bisimulation between extensional well-founded graphs
  must be an isomorphism.
\end{lem}
\begin{proof}
  Let $R\mono X\times Y$ be such.  We show by well-founded induction
  that if $R(x,y_1)$ and $R(x,y_2)$, then $y_1=y_2$.  For if $z_1\prec
  y_1$, then since $R$ is a bisimulation, there is a $w\prec x$ with
  $R(w,z_1)$, and again since $R$ is a bisimulation, there is a
  $z_2\prec y_2$ with $R(w,z_2)$.  By induction, $z_1=z_2$, so that
  $z_1\prec y_2$ for any $z_1\prec y_1$.  By symmetry, for any
  $z_2\prec y_2$ we have $z_2\prec y_1$, hence by extensionality
  $y_1=y_2$.

  By symmetry, if $R(x_1,y)$ and $R(x_2,y)$ then $x_1=x_2$, so $R$ is
  functional in both directions.  Since it is also bi-entire, it must
  be an an isomorphism.
\end{proof}

\begin{lem}\label{thm:wfext-bisid}
  If $X$ is well-founded, then it is extensional if and only if every
  bisimulation from $X$ to $X$ is contained in the identity.
\end{lem}
\begin{proof}
  Suppose first that $X$ is well-founded and extensional and that $R$
  is a bisimulation with $R(x_1,x_2)$.  Then $R$ is a bi-entire
  bisimulation from $X/x_1$ to $X/x_2$ (by ordinary induction on
  length of paths), so by \autoref{thm:bisim-iso} it must be an
  isomorphism $X/x_1\cong X/x_2$.  But $X$ is well-founded and
  extensional, so by \autoref{thm:wf-rigid}, $x_1=x_2$.

  Now suppose that every bisimulation from $X$ to $X$ is contained in
  the identity, and also that $x,y\in X$ are such that $z\prec x
  \Leftrightarrow z\prec y$ for all $z$.  Define $R(a,b)$ to hold if
  either $a=b$, or $a=x$ and $b=y$.  Then $R$ is a bisimulation, and
  if it is contained in the identity, then $x=y$; hence $X$ is
  extensional.
\end{proof}

In the language of~\cite{aczel:afa}, \autoref{thm:wfext-bisid} says
that for well-founded graphs, extensionality is equivalent to ``strong
extensionality.''

\begin{cor}
  If $Y$ is well-founded and extensional, then any two simulations
  $f,g\colon X\toto Y$ are equal.
\end{cor}
\begin{proof}
  The image of $(f,g)\colon X\to Y\times Y$ is a bisimulation,
  hence contained in the identity.
\end{proof}

Therefore, well-founded extensional graphs and simulations form a
(large) preorder.  In the language of material set theory, this
preorder represents the partial order of transitive-sets and subset
inclusions.


\begin{lem}\label{thm:bisim-quot}
  If $X$ is a graph and $R$ is a bisimulation from $X$ to $X$ which is
  an equivalence relation, then its quotient $Y$ inherits a graph
  structure such that the quotient map $[-]\colon X\epi Y$ is a
  simulation.  Also, if $X$ is an \apg, then so is $Y$.
\end{lem}
\begin{proof}
  Define $\prec$ on $Y$ to be minimal such that $[-]$ preserves
  $\prec$, i.e.\ $y_1\prec y_2$ if there exist $x_1\prec x_2$ in $X$
  with $[x_1] = y_1$ and $[x_2]=y_2$.  Now suppose that $y\prec [x]$.
  By definition this means that there exist $z_1\prec z_2$ with
  $[z_1]=y$ and $[z_2] = [x]$, i.e.\ $R(z_2,x)$ holds.  But $R$ is a
  bisimulation, so there exists $x'\prec x$ with $R(z_1,x')$,
  i.e. $[x']=y$; hence $[-]$ is a simulation.  If $X$ is an \apg, we
  define the root of $Y$ to be $[\star]$; accessibility of $Y$ follows
  directly from that of $X$.
\end{proof}

Of course, by \autoref{thm:bisim-wf}, if $X$ is well-founded, then so
is the quotient $Y$.  If $R$ is the \emph{largest} bisimulation on
$X$, then its quotient is easily verified to be extensional.  The
largest bisimulation exists if \Set\ is a topos, or if it satisfies
full separation, so in these situations every well-founded graph has
an extensional quotient.  In general, this seems not to be provable,
but we can still construct extensional quotients in a useful amount of
generality.

\begin{lem}\label{thm:ext-quotient}
  Let $n$ be a fixed \emph{external} natural number, let $X$ be a
  well-founded \apg, and assume that $X/x$ is extensional whenever $x$
  is a node that admits a path of length $n$ to the root.  Then there
  is an extensional well-founded \apg\ $\Xbar$ and a surjective
  simulation $q\colon X\to \Xbar$.
\end{lem}
\begin{proof}
  The proof is by external induction on $n$.  (We will only need this
  lemma for $n\le 3$, so the reader is encouraged not to worry too
  much about what this induction requires of the metatheory.)  The
  base case is easy: since the root $\star$ admits a path of length
  $0$ to itself and $X\cong X/\star$, we can take $\Xbar=X$.

  Now suppose the statement is true for some $n$, and let $X$ satisfy
  the hypothesis for $n+1$.  For any $k$, write $X_{k}$ for the set of
  nodes admitting a path of length $k$ to the root.  Let the relation
  $R$ on $X$ be defined by $R(x,y)$ if there exists an isomorphism
  $X/x \toiso X/y$.  (This can be constructed using \ddo-separation
  and local exponentials in \Set.)  Then $R$ is a bisimulation and an
  equivalence relation, so by \autoref{thm:bisim-quot} it has a
  quotient $Y$ which is again a well-founded \apg.

  We claim that $Y$ satisfies the hypothesis for $n$.  Let $y\in Y_n$
  and suppose that $y_1, y_2\in Y/y$ satisfy $z\prec y_1 \Iff z\prec
  y_2$; we must show $y_1=y_2$.  Applying the simulation property
  inductively, we have $y = [x]$ for $x\in X_n$, and $y_i = [x_i]$
  with $x_i \in X/x$ for $i=1,2$.  By \autoref{thm:root-decid}, for
  each $i$ either $x_i = x$ or $x_i \in X/w_i$ for some $w_i\prec x$.
  If $x_1 = x$ and $x_2=x$, then of course $y_1 = [x_1]=[x_2]=y_2$.
  If $x_1 = x$ and $x_2\in X/w_2$ with $w_2\prec x = x_2$, then there
  would be a cyclic path in $X$ from $x_2$ to itself, contradicting
  \autoref{thm:wf-no-cycles}.  Hence the only remaining case is when
  $x_i \in X/w_i$ with $w_i\prec x$ for both $i=1,2$.

  Now this implies that $w_i\in X_{n+1}$, so each $X/w_i$ is
  extensional, and thus so is each $X/x_i$ by \autoref{thm:ext-hered}.
  Therefore, by \autoref{thm:sim-inj}, the quotient map $[-]\colon
  X\to Y$ induces an isomorphism $X/x_i \cong Y/[x_i]= Y/y_i$.  But
  $z\prec y_1 \Iff z\prec y_2$ means that $Y\csl y_1 \cong Y\csl y_2$,
  and hence (by \autoref{thm:root-decid}) also $Y/y_1 \cong Y/y_2$;
  thus we also have $X/x_1 \cong X/x_2$.  Thus, by definition,
  $R(x_1,x_2)$, and so $y_1 = [x_1]=[x_2]=y_2$.

  We have shown that $Y$ satisfies the hypothesis for $n$.  Thus it
  has an extensional quotient $\Ybar$, and so the composite $X\to Y\to
  \Ybar$ is an extensional quotient of $X$.
\end{proof}

We can now start verifying the axioms of material set theory.

\begin{lem}
  The axiom of extensionality holds.  That is, two well-founded
  extensional \apgs\ $X$ and $Y$ are isomorphic iff $Z\iin X \Iff
  Z\iin Y$ for all $Z$.
\end{lem}
\begin{proof}
  The ``only if'' direction is clear, so suppose that $Z\iin X \Iff
  Z\iin Y$ for all $Z$.  Then for every $x\in |X|$, we have $X/x \iin
  Y$, hence $X/x \cong Y/y$ for some $y\in Y$.  By
  \autoref{thm:wf-rigid}, this $y$ must be unique, and conversely as
  well; hence we have a bijection $g\colon |X|\toiso |Y|$ such that
  for any $x\in |X|$ there exists an isomorphism $h_x\colon X/x \toiso
  Y/g(x)$ (which must be unique, by \autoref{thm:wf-rigid-global}).
  Define a relation $R$ from $X$ to $Y$ such that $R(a,b)$ holds if:
  \begin{enumerate}
  \item $a=\star$ and $b=\star$, or
  \item there exists $x\in |X|$ such that $a\in X/x$, $b\in Y/g(x)$,
    and $h_x(a)=b$.
  \end{enumerate}
  Then $R$ is a bi-entire bisimulation, so by \autoref{thm:bisim-iso} it
  is an isomorphism.
\end{proof}

\begin{lem}\label{thm:struct-emppairun}
  The axioms of empty set, pairing, and union hold.
\end{lem}
\begin{proof}
  The empty set is represented by the \apg\ with one node and no
  $\prec$ relations, which has no members.

  If $X$ and $Y$ are extensional well-founded \apgs, let $Z = X + Y +
  1$ with $\prec$ induced from $X$ and $Y$ along with $\star_X \prec
  \star$ and $\star_Y\prec \star$, where $\star$ is the new point
  added.  Since $X$ and $Y$ are extensional, $Z$ satisfies the
  hypothesis of \autoref{thm:ext-quotient} with $n=1$, and its
  extensional quotient represents the pair $\{X,Y\}$.

  Finally, if $X$ is an extensional well-founded \apg, let $\Vert X
  \Vert$ denote the subset of those $x\in X$ such that $x\prec y\prec
  \star$ for some $y$, let $Y$ be the subset of $X$ consisting of
  those nodes admitting some path to a node in $\Vert X \Vert$, and
  define $Z = Y + 1$ with \prec\ inherited from $Y$ and with $y\prec
  \star$ for each $y\in \Vert X\Vert \subseteq Y$, where $\star$ is
  the new element added.  Then $Z$ satisfies the hypotheses of
  \autoref{thm:ext-quotient} with $n=1$, so it has an extensional
  quotient, which is the desired union $\bigcup X$.
\end{proof}

\begin{lem}\label{thm:cart-prod}
  Cartesian products (using Kuratowski ordered pairs) exist in
  $\bbV(\Set)$.
\end{lem}
\begin{proof}
  Let $X$ and $Y$ be well-founded \apgs, and consider the set
  \[ Z =
  (X\csl\star) + (Y\csl\star) +
  |X| + (|X|\times |Y|) +
  (|X|\times |Y|)
  + 1.
  \]
  For $x\in |X|$ we write $x$ for its image in $(X\csl\star)$ and $x'$
  for its image in the first copy of $|X|$.  Similarly, we write $y$ for
  images in $Y\csl\star$, $(x,y)$ for images in the first copy of
  $|X|\times |Y|$, $(x,y)'$ for images in the second copy, and $\star$
  for the final point.  We define $\prec$ on $Z$ as follows:
  \begin{blist}
  \item $\prec$ on $X\csl\star$ and $Y\csl\star$ is induced from
    $X$ and $Y$.
  \item $x\prec x'$ for all $x\in |X|$.
  \item $x\prec (x,y)$ and $y\prec (x,y)$ for all $y\in |Y|$.
  \item $x' \prec (x,y)'$ and $(x,y)\prec (x,y)'$ for all $x\in |X|$
    and $y\in |Y|$.
  \item $(x,y)'\prec \star$ for all $x\in |X|$ and $y\in |Y|$.
  \end{blist}
  It is straightforward to verify that $Z$ is then a well-founded
  \apg.  Since $X$ and $Y$ are extensional, $Z$ satisfies the
  hypothesis of \autoref{thm:ext-quotient} with $n=3$.  Its
  extensional quotient then represents the cartesian product of $X$
  and $Y$.
\end{proof}

\begin{lem}\label{thm:struc-exp}
  $\bbV(\Set)$ satisfies the exponentiation axiom.
\end{lem}
\begin{proof}
  If $X$ and $Y$ are extensional well-founded \apgs, let $Z$ be their
  material cartesian product as above, and define
  \[W =
  (Z\csl \star) +
  |Y|^{|X|} +
  1
  \]
  with $\prec$ induced from $Z$ along with (using the notation of
  \autoref{thm:cart-prod}) $(x,y)' \prec f$ whenever $f(x)=y$, and
  $f\prec \star$ for any $f\in {|Y|}^{|X|}$.  Then $W$ is a
  well-founded \apg, and in fact is already extensional; we claim it
  represents the material function-set.

  It is clear that if $F\iin W$, then $F\in \bbV(\Set)$ is a function
  from $X$ to $Y$ in the sense of material set theory.  Conversely,
  from any $F\in \bbV(\Set)$ which is a function from $X$ to $Y$,
  consider the subset of $|X|\times |Y|$ determined by those $(x,y)$
  such that the Kuratowski ordered pair $\{ \{ X/x\}, \{X/x,Y/y\}\}$
  is $\iin F$.  This defines a function $|X|\to |Y|$ in \Set, which
  therefore induces an $f\in |W|$ such that $F\cong W/f$.
\end{proof}

Observe that in particular, we have shown that for $X,Y\in\bbV(\Set)$,
there is a 1-1 correspondence between functions $|X|\to |Y|$ in \Set\
and isomorphism classes of \apgs\ $F\in\bbV(\Set)$ which represent
functions from $X$ to $Y$ in the sense of material set theory.  In
fact, although $\bbV(\Set)$ need not satisfy limited \ddo-replacement,
so that \autoref{thm:iz-topos} does not apply to it directly, we still
have:

\begin{lem}\label{thm:struc-mat-set}
  The sets and functions in $\bbV(\Set)$ form a category
  $\bbSet(\bbV(\Set))$, which can naturally be identified with a full
  subcategory of $\Set$ closed under finite limits, subsets,
  quotients, and local exponentials.  In particular,
  $\bbSet(\bbV(\Set))$ is a \Pi-pretopos.
\end{lem}
\begin{proof}
  We have observed closure under products and non-local exponentials.
  Closure under subsets is easy: if $U\subseteq |X|$, then the
  sub-graph $Y$ of $X$ consisting of the root and all nodes admitting
  a path to $U$ is a well-founded extensional \apg\ with $|Y|\cong U$.
  This then implies closure under finite limits.

  For quotients, if $R$ is an equivalence relation on $|X|$, let
  $|X|\xepi{q}Y$ be the quotient of $|X|$ by $R$ in \Set; then the
  \apg\
  \[ Z = (X\csl\star) + Y + 1,\]
  with $\prec$ inherited from $X$ along with $x\prec q(x)$ and $y\prec
  \star$ for all $x\in |X|$ and $y\in Y$, is well-founded and
  extensional and has $|Z|\cong Y$.

  Finally, given $f\colon |A|\to |B|$ and $g\colon |X|\to |A|$, we
  have the local exponential $h\colon \Pi_f(g) \to |B|$ in \Set, with
  counit $e\colon \Pi_f(g)\times_{|B|} |A| \to |X|$.  Let $Z$
  represent the material cartesian product of $A$ and $X$ as in
  \autoref{thm:cart-prod}, and consider the \apg\
  \[ W= (Z\csl\star) + \Pi_f(g) + 1 \]
  with $\prec$ inherited from $Z$, along with $(x,y)' \prec j$
  whenever $f(x) = h(j)$ and $e(j,x)=y$, and $j\prec \star$ for all
  $j\in \Pi_f(g)$.  Then $W$ is well-founded and extensional and
  $|W|\cong \Pi_f(g)$; hence $\bbSet(\bbV(\Set))$ is closed under
  local exponentials.
\end{proof}


We return to verifying the axioms of material set theory.

\begin{lem}
  If \Set\ satisfies structural fullness, then $\bbV(\Set)$ satisfies
  material fullness.
\end{lem}
\begin{proof}
  Analogously to exponentiation, if $R\mono M\times X\times Y$ is a
  generic set of multi-valued functions from $|X|$ to $|Y|$, consider
  \[W = (Z\csl\star) + M + 1
  \]
  with $\prec$ induced from $Z$ along with $(x,y)'\prec m$ if
  $R(m,x,y)$ and $m\prec \star$ for all $m\in M$.  Then $W$ is a
  well-founded \apg\ and satisfies the hypothesis of
  \autoref{thm:ext-quotient} with $n=1$.  Its extensional quotient
  represents a generic set of multi-valued functions, by a similar
  argument as for exponentiation.
\end{proof}

\begin{lem}\label{thm:struc-pow}
  If \Set\ is a topos, then $\bbV(\Set)$ satisfies the power set
  axiom, and $\bbSet(\bbV(\Set)) \subseteq \Set$ is a logical
  subtopos.
\end{lem}
\begin{proof}
  For an extensional well-founded \apg\ $X$, define
  \[ Y = (X\csl \star) + P |X| + \star
  \]
  with $\prec$ induced from $X$ along with $x\prec A$ whenever $x\in
  |X|$, $A\in P |X|$, and $x\in A$; and also of course $A\prec
  \star$.  This is an extensional well-founded \apg\ that represents
  the material power set of $X$.  The second statement is immediate.
\end{proof}

\begin{lem}\label{thm:inf}
  $\bbV(\Set)$ satisfies the axiom of infinity.
\end{lem}
\begin{proof}
  Let $\omega = N + 1$, with $\prec$ defined to be $<$ on $N$ together
  with $n\prec \star$ for all $n\in N$.  This is an extensional
  well-founded \apg\ which represents the von Neumann ordinal \omega.
  The infinity axiom follows from the universal property of the \nno.
\end{proof}

\begin{lem}\label{thm:struc-ac}
  If \Set\ satisfies the axiom of choice, then so does $\bbV(\Set)$.
\end{lem}
\begin{proof}
  Let $X$ be a well-founded extensional \apg\ such that $X\csl x$ is
  inhabited for each $x\in |X|$.  If $Y\subseteq X\times |X|$ consists
  of those $(y,x)$ such that $y\in X\csl x$, then the projection $Y\to
  |X|$ is surjective, and hence has a section, say $s$.  From $s$ we
  can easily construct a material choice function for $X$.
\end{proof}

\begin{lem}\label{thm:struc-tc}
  \Set\ has transitive closures.
\end{lem}
\begin{proof}
  If $X$ is a well-founded extensional \apg, let $T = (X\csl \star) +
  1$ with $\prec$ inherited from $X$ along with $x\prec \star$ for all
  nodes $x\in X$ (not just all members).  Then $T$ is a well-founded
  extensional \apg\ which represents the transitive closure of $X$.
\end{proof}

\begin{lem}\label{thm:struc-most}
  \Set\ satisfies Mostowski's principle.
\end{lem}
\begin{proof}
  Since $\bbSet(\bbV(\Set))$ is closed in \Set\ under finite limits
  and subsets, any well-founded extensional graph $X$ constructed in
  the material set theory $\bbV(\Set)$ will induce such a graph in
  $\Set$.  Therefore, $X+1$, with $\prec$ induced from $X$ along with
  $x\prec \star$ for all nodes $x\in X$, is a well-founded extensional
  \apg, i.e.\ an object of $\bbV(\Set)$.  We then verify that it is
  transitive and isomorphic to $X$ in $\bbV(\Set)$.
\end{proof}

We now turn to the axiom schemata.  For these, we need to be able to
translate material formulas into structural ones.  There is an obvious
way to do this: if \ph\ is a formula in $\bbV(\Set)$, we define
$\ph_{\ordiin}$ in \Set\ as follows:
\begin{blist}
\item We replace the material ``equality'' symbol $=$ by isomorphism
  $\cong$ of \apgs.
\item We replace the material ``membership'' symbol $\in$ by the
  relation $\iin$.
\item The connectives are unchanged.
\item We replace quantifiers over material-sets by quantifiers over
  well-founded extensional \apgs.  For example, $\exists x. \ph(x)$
  becomes \qq{there exists a well-founded extensional \apg\ $X$ such
    that $\ph_{\ordiin}(X)$}.
\end{blist}
It is easy to see that $\bbV(\Set) \ss \ph$ if and only if $\Set \ss
\ph_{\ordiin}$.  This translation works quite well for the schemata
involving arbitrary formulas.

\begin{lem}\label{thm:struc-sep}
  If \Set\ satisfies separation, then so does $\bbV(\Set)$.
\end{lem}
\begin{proof}
  Let $\ph(x)$ be a formula and $A\in\bbV(\Set)$.  Using separation,
  let $U \subseteq |A|$ consist of precisely those $a\in A$ such that
  $\ph_{\ordiin}(A/a)$.  Define $B$ to consist of the root of $A$
  together with all nodes admitting a path to some node in $U$.  Then
  $B$ is a well-founded extensional \apg, and for any $C\iin A$ we
  have $C\iin B$ iff $\ph_{\ordiin}(C)$.
\end{proof}

\begin{lem}\label{thm:struc-fullbool}
  If \Set\ satisfies full classical logic, then so does $\bbV(\Set)$.
\end{lem}
\begin{proof}
  Classical logic for \Set\ implies $\ph_{\ordiin} \vee \neg
  \ph_{\ordiin}$ for any formula $\ph$ in $\bbV(\Set)$.
\end{proof}

\begin{lem}\label{thm:struc-coll}
  If \Set\ satisfies collection, then so does $\bbV(\Set)$.
\end{lem}
\begin{proof}
  Suppose that $A\in\bbV(\Set)$ and that \ph\ is a formula such that
  for any $X\iin A$, there exists a $Y\in\bbV(\Set)$ with $\ph(X,Y)$.
  This means that for any $x\in |A|$, there exists a well-founded
  extensional \apg\ $Y$ such that $\ph_{\ordiin}(A/x,Y)$.  By
  collection in \Set, there is a surjection $V\xepi{p} |A|$ and a
  pointed graph $B\in \Set/V$ (here ``pointed'' means we have a
  section $s\colon V\to B$ over $V$) such that for each $v\in V$,
  $v^*B$ is a well-founded extensional \apg\ and
  $\ph_{\ordiin}(A/p(v),v^*B)$.  It is easy to show that $B$,
  considered as a graph in \Set, is still well-founded.  And since $B$
  is a graph in $\Set/V$, its relation $\prec$ is fiberwise; thus for
  each $v\in V$ we have $B/s(v) \cong v^*B$, which is therefore
  extensional and accessible.  It follows that $B+1$, with $s(v)\prec
  \star$ for all $v\in V$, satisfies the hypotheses of
  \autoref{thm:ext-quotient} with $n=1$.  Its extensional quotient is
  then the set desired by the material collection axiom.
\end{proof}

\begin{lem}\label{thm:struc-rep}
  If \Set\ satisfies replacement of contexts, then $\bbV(\Set)$
  satisfies material replacement.
\end{lem}
\begin{proof}
  Suppose that $A\in\bbV(\Set)$ and that \ph\ is a formula such that
  for any $X\iin A$, there exists a unique $Y\in\bbV(\Set)$ with
  $\ph(X,Y)$.  Thus, for any $x\in |A|$, there exists a well-founded
  extensional \apg\ $Y$ such that $\ph_{\ordiin}(A/x,Y)$, and any two
  such $Y$ are isomorphic.  Since any such isomorphism is unique by
  \autoref{thm:wf-rigid-global}, replacement of contexts in \Set\
  supplies a pointed graph $B\in \Set/|A|$ such that for each $x\in
  |A|$ we have $\ph_{\ordiin}(A/x,x^*B)$.  The extensional quotient of
  $B+1$ then represents the set desired by material replacement.
\end{proof}

\begin{lem}\label{thm:struc-ind}
  If \Set\ satisfies full induction, then so does $\bbV(\Set)$.
\end{lem}
\begin{proof}
  If $\ph(x)$ is as in the statement of the material induction axiom,
  then $\ph_{\ordiin}(X)$ is a statement about some $X\iin \omega$,
  i.e.\ $X\in N$.  But every $X\iin N$ is isomorphic to $N/n$ for a
  unique $n\in N$, so $\ph_{\ordiin}(X)$ is equivalent to a statement
  $\ph_{\ordiin}'(n)$ about some $n\in N$, which can then be proven by
  the structural induction axiom.
\end{proof}

\begin{lem}\label{thm:struc-setind}
  If \Set\ satisfies well-founded induction, then $\bbV(\Set)$
  satisfies set-induction.
\end{lem}
\begin{proof}
  Just like \autoref{thm:struc-ind}, using induction over the
  well-founded (extensional) relation in \Set\ that underlies any
  object of $\bbV(\Set)$.
\end{proof}

The schemata involving \ddo-formulas require a little more work, since
if \ph\ is \ddo\ then $\ph_{\ordiin}$ need not be.  Thus, we need to
define a different translation for \ddo-formulas.  Suppose \ph\ is a
\ddo-formula in $\bbV(\Set)$ with parameters $A_1,\dots,A_k$, each of
which is a well-founded extensional \apg.  Then
\[A_1 + \dots + A_k + 1,
\]
with $\prec$ inherited from the $A_i$ along with $\star_i \prec \star$
for all $1\le i\le k$, satisfies the hypothesis of
\autoref{thm:ext-quotient} with $n=1$.  Let $T$ be its extensional
quotient.  We can now translate \ph\ into a \ddo-formula as follows,
interpreting material-set variables not by \apgs\ but by elements of
$T$.
\begin{blist}
\item Each parameter $A_i$ is replaced by $[\star_i]\in T$.
\item The material ``equality'' symbol $=$ is replaced by
  equality in $T$.
\item The material ``membership'' symbol $\in$ is replaced by
  $\prec$ in $T$.
\item The connectives are unchanged.
\item A \ddo-quantifier of the form $\exists x\in y.\ph$ is replaced
  by a \ddo-quantifier of the form $\exists x\in T.(x\prec y\meet
  \ph)$.  Similarly, $\forall x\in y.\ph$ is replaced by $\forall x\in
  T.(x\prec y\imp \ph)$.
\end{blist}
We call the formula produced in this way $\ph_{\ordiin}^{\ddo}$.

\begin{lem}
  $\ph_{\ordiin}$ is equivalent to $\ph_{\ordiin}^{\ddo}$.
\end{lem}
\begin{proof}
  \autoref{thm:wf-rigid} implies that $T/x \cong T/y$ if and only if
  $x=y$.  Similarly, if $T/x \iin T/y$, then $T/x \cong (T/y)/y' =
  T/y'$ for some $y'\prec y$, whence $x=y'$ and so $x\prec y$.  The
  converse is easy, so $T/x \iin T/y$ if and only if $x\prec y$.  Thus
  the atomic formulas correspond, and the connectives evidently do, so
  it remains to observe that \ddo-quantifiers are adequately
  represented by quantifiers over $T$, since by definition $X\iin
  T/y$ if and only if $X\cong T/x$ for some $x\prec y$.
\end{proof}

\begin{lem}\label{thm:struc-ddosep}
  $\bbV(\Set)$ satisfies \ddo-separation.
\end{lem}
\begin{proof}
  Just like \autoref{thm:struc-sep}, but using $\ph_{\ordiin}^{\ddo}$
  instead of $\ph_{\ordiin}$, and the \ddo-separation property of
  \autoref{thm:ddo-sep} instead of the separation axiom.
\end{proof}

\begin{lem}\label{thm:struc-fdn}
  $\bbV(\Set)$ satisfies foundation.
\end{lem}
\begin{proof}
  For any \ddo-formula $\ph$ in $\bbV(\Set)$ and any well-founded
  \apg\ $X\in\bbV(\Set)$, we can form $\setof{ x\in X | \ph(X/x) }$
  using \ddo-separation, since $\ph(X/x) \equiv
  \ph_{\ordiin}^{\ddo}(x)$.  The hypothesis on \ph\ implies that this
  is an inductive subset of $X$, hence all of it.  This implies the
  desired conclusion, since for every well-founded extensional \apg\
  $Y$ there is another one $X$ with $Y\iin X$.
\end{proof}

\begin{lem}\label{thm:struc-bool}
  If \Set\ is Boolean, then $\bbV(\Set)$ satisfies \ddo-classical logic.
\end{lem}
\begin{proof}
  For any \ddo-formula \ph\ in $\bbV(\Set)$, \ddo-separation supplies
  a subset $\setof{ \emptyset | \ph }\subseteq \{\emptyset\}$.  If
  this is complemented in \Set, then we must have $\ph\vee\neg\ph$.
\end{proof}

This completes the proof of \autoref{thm:str->mat}.

Of course, it is natural to ask to what extent the constructions
$\bbV(-)$ and $\bbSet(-)$ are inverse.  We have already seen the
canonical inclusion $\bbSet(\bbV(\Set))\into \Set$, and the following
is easy to verify.

\begin{lem}
  The inclusion $\bbSet(\bbV(\Set))\into \Set$ is an equivalence if
  and only if every object of \Set\ can be embedded into some
  well-founded extensional graph.\qed
\end{lem}

I propose to call this property of \Set\ the \emph{axiom of
  well-founded materialization}.  (Other axioms of materialization
would arise from using various kinds of non-well-founded graphs.)

Note that well-founded materialization follows from the axiom of
choice, since then every object can be well-ordered and thus given the
structure of an \apg\ representing a von Neumann ordinal.  We remark
in passing that this implies that ``all replacement schemata for ETCS
are equivalent.''

\begin{prop}\label{thm:etcs-rep}
  Let $\mathcal{T}_1$ and $\mathcal{T}_2$ be axiom schemata for
  structural set theory such that for each $i=1,2$,
  \begin{enumerate}
  \item if \bV\ satisfies ZFC, then $\bbSet(\bV)$ satisfies
    $\mathcal{T}_i$, and\label{item:er1}
  \item if \Set\ satisfies ETCS+$\mathcal{T}_i$, then $\bbSet(\bV)$
    satisfies ZFC.\label{item:er2}
  \end{enumerate}
  Then $\cT_1$ and $\cT_2$ are equivalent over ETCS.
\end{prop}
\begin{proof}
  Since ETCS includes AC, for any model of ETCS we have $\Set \simeq
  \bbSet(\bbV(\Set))$.  Thus, if $\Set\ss\cT_1$, then
  by~\ref{item:er2} we have $\bbSet(\bV)\ss$ ZFC, hence
  by~\ref{item:er1} we have $\Set = \bbSet(\bbV(\Set))\ss\cT_2$.  The
  converse is the same.
\end{proof}

We have seen that our axioms of collection and replacement from
\S\ref{sec:strong-ax} satisfy~\ref{item:er1} and~\ref{item:er2}, so
over ETCS they are equivalent to any other such schema.  This includes
the replacement axiom of McLarty~\cite{mclarty:catstruct}, the axiom
CRS of Cole~\cite{cole:cat-sets}, the axiom RepT of
Osius~\cite{osius:cat-setth}, and the reflection axiom of
Lawvere~\cite[Remark 12]{lawvere:etcs-long}.  However, our axioms seem
to be more appropriate in an intuitionistic context.

Returning to the material-structural comparison, on the other side we
have:

\begin{prop}\label{thm:fdn-eqv}
  If \bV\ satisfies the core axioms of material set theory along with
  infinity, exponentials, foundation, and transitive closures, then
  there is a canonical embedding $\bV \to \bbV(\bbSet(\bV))$.  This
  map is an isomorphism if and only if \bV\ additionally satisfies
  Mostowski's principle.
\end{prop}
\begin{proof}
  By assumption, any $x\in \bV$ has a transitive closure
  $\mathrm{TC}(x)$, which is a well-founded extensional graph.  If we
  define $Y = \mathrm{TC}(x)+1$, with $\prec$ induced by $\in$ on
  $\mathrm{TC}(x)$ and with $z\prec \star$ for all $z\in x$, then $Y$
  is a well-founded extensional \apg, i.e.\ an object of
  $\bbV(\bbSet(\bV))$.  This construction gives a map $\bV \to
  \bbV(\bbSet(\bV))$, and it is straightforward to verify that it
  preserves and reflects membership and equality.

  Now if this embedding is an isomorphism, then clearly \bV\ must
  satisfy Mostowski's principle, since $\bbV(\bbSet(\bV))$ does so.
  Conversely, if \bV\ satisfies Mostowski's principle, then every
  $X\in \bbV(\bbSet(\bV))$ is isomorphic in \bV\ to a transitive set,
  and therefore equal in $\bbV(\bbSet(\bV))$ to something in the image
  of the embedding.
\end{proof}


\begin{rmk}
  The theory M$_0$ of~\cite{mathias:str-maclane} consists of the core
  axioms together with power sets and full classical logic.  If
  $\bV\ss$ M$_0$, then $\bbV(\bbSet(\bV))$ is precisely the model
  $W_1$ constructed in~\cite[\S2]{mathias:str-maclane}, which
  satisfies M$_0$ plus regularity, transitive closures, and
  Mostowski's principle, and inherits infinity and choice from \bV.
  (In the presence of classical logic and power sets, the axiom of
  infinity is unnecessary for the construction $\bbV(-)$.)
  In particular, if $\bV\ss$ ZBQC, then $\bbV(\bbSet(\bV))\ss$ MOST,
  while $\bbV(\bbSet(\bV))\cong \bV$ if we already had $\bV\ss$ MOST.
\end{rmk}



So far we have concentrated on building models of pure sets only.
However, we can also allow an arbitrary set of atoms: we fix some
$A\in \Set$ and modify our definitions as follows.  (The definitions
not listed below need no modification.)
\begin{blist}
\item An \textbf{$A$-graph} is a graph $X$ together with a partial
  function $\ell\colon X \rightharpoonup A$, such that
  $\mathrm{dom}(\ell)$ is a complemented subobject of $X$, and if
  $x\prec y$ then $y\notin \mathrm{dom}(\ell)$.
\item Isomorphisms between $A$-graphs are, of course, required to
  preserve the labeling functions $\ell$.
\item An $A$-graph is \textbf{$A$-extensional} if
  \begin{enumerate}
  \item $\ell$ is injective on its domain, and
  \item for any $x,y\notin\mathrm{dom}(\ell)$, if $z\prec x
    \Leftrightarrow z\prec y$ for all $z$ then $x=y$.
  \end{enumerate}
\item An \textbf{$A$-simulation} between $A$-graphs is a simulation
  $f\colon X\to Y$ such that for any $x\in X$, if either $\ell(x)$ or
  $\ell(f(x))$ is defined, then both are, and they are equal.
\end{blist}
Note that if an $A$-graph $X$ is an \apg\ and
$\star\in\mathrm{dom}(\ell)$, then $X$ can have no nodes other than
the root, so $X$ is simply an element of $A$.  Thus, accessible
pointed $A$-graphs model atoms themselves in addition to sets that can
contain atoms.

It is straightforward to verify that all the above lemmas, and the
main theorem, still hold when extensional well-founded \apgs\ are
replaced by $A$-extensional well-founded accessible pointed
$A$-graphs.  The resulting material set theories, of course, satisfy
the modified axioms for theories with atoms listed at the end of
\S\ref{sec:sets}.

\section{The stack semantics}
\label{sec:internal-logic}

We now introduce the stack semantics of a Heyting pretopos \bS, no
longer \mbox{assumed} to be well-pointed.  (In fact, we will work most of
the time with a positive Heyting category, although a few facts
require exactness.)  In the introduction, we described two approaches
to this semantics: first, as a direct generalization of the usual
Kripke-Joyal semantics, and second, as a fragment of the internal
logic of the category of sheaves or stacks.  In this section we will
take the first, more explicit, viewpoint; in \S\ref{sec:cofc} we will
show how it agrees with the second.

Recall that our goal is to embed the usual internal type theory in a
structural set theory, along the lines of \autoref{thm:ddo-corresp}.
The interpretation of the internal type theory is usually defined by
constructing, for any formula \ph\ with (say) one free variable $x$ of
type $A$ (an object of \bS), a subobject $\mm{\ph}\mono A$ which we
think of as the subset $\setof{x\in A | \ph(x)}$.  This approach is
very powerful and flexible---for instance, it generalizes to the
internal logic of any fibered poset---but it seems inadequate to deal
directly with unbounded quantifiers.

Thus, we start instead with the \emph{Kripke-Joyal semantics} (see
e.g.~\cite[VI.6]{mm:shv-gl}), according to which $\mm{\ph}\mono A$ can
be defined indirectly by characterizing the sub-presheaf
$\bS(-,\mm{\ph})\mono \bS(-,A)$, i.e.\ characterizing which maps $U\to
A$ factor through $\mm{\ph}$.  The usual construction of $\mm{\ph}$
then becomes a \emph{proof} that this subfunctor is representable.
The stack semantics consists of defining a sub-presheaf
``$\bS(-,\mm{\ph})$'' for any formula \ph\ in a similar way, though it
will not in general be representable.  (This also makes the connection
with $\bSh(\bS)$ clear, since this sub-presheaf is in fact a sub-sheaf
and thus a genuine subobject in $\bSh(\bS)$.)


If $U$ is an object of \bS\ and \ph\ is a formula of category theory
with parameters in $\bS/U$, we say that \ph\ is a \textbf{formula over
  $U$}.  Note that a formula over $1$ is the same as a formula in \bS\
itself.  We can think of a formula over $U$ as an assertion taking
place in each fiber.  If moreover $V\too[p] U$ is any map, its
\textbf{pullback} $p^*\ph$ is a formula over $V$ obtained by replacing
each parameter of \ph\ by its pullback along $p$.

\begin{rmk}
  Since pullbacks are only defined up to isomorphism, the notation
  $p^*\ph$ is, strictly speaking, ambiguous.  However, by
  isomorphism-invariance (Lemmas \ref{thm:isoinvar-truth} and
  \ref{thm:forcing-basics}\ref{item:forcing-4}), the particular
  choices made are irrelevant.  Moreover, since only a finite number
  of choices are ever needed for any particular formula, making such
  choices requires no axiom of choice in the metatheory.
\end{rmk}



\begin{defn}[The stack semantics]\label{defn:ff}
  Let \bS\ be a positive Heyting category, and let \ph\ be a sentence over $U$
  in \bS.  The relation $U\ff\ph$ is defined recursively as follows.
  \begin{labellist}{}{\labelwidth=2.4em\leftmargin=3em}
  \item{\rm(\ff\fa)} $U\ff (f=g)$ iff in fact $f=g$.\label{item:ff-atomic}
  \item{\rm(\ff\top)} $U\ff \top$ always.
  \item{\rm(\ff\bot)} $U\ff \bot$ iff $U$ is an initial object.
  \item{\rm(\ff\meet)} $U\ff (\ph\meet\psi)$ iff $U\ff \ph$ and $U\ff \psi$.
  \item{\rm(\ff\join)} $U\ff (\ph\join \psi)$ iff $U=V\cup W$, where
    $V\xmono{i} U$ and $W\xmono{j} U$ are subobjects such that $V\ff
    i^*\ph$ and $W\ff j^*\psi$.\label{item:ff-or2}
  \item{\rm(\ff\imp)} $U\ff (\ph\imp\psi)$ iff for any $V\too[p] U$ such that
    $V\ff p^*\ph$, also $V\ff p^*\psi$.\label{item:ff-imp2}
  \item{\rm(\ff\neg)} $U\ff\neg\ph$ iff $U\ff(\ph\imp\bot)$.
  \item{\rm(\ff$\im_0$)} $U\ff \exists X. \ph(X)$ iff there is a
    regular epimorphism $V\xepi{p} U$ and an object $A\in\bS/V$ such that
    $V\ff p^*\ph(A)$.\label{item:ff-existsobj2}
  \item{\rm(\ff$\im_1$)} Similarly, $U\ff \exists f\maps A\to B.
    \ph(f)$ iff there is a regular epimorphism $V\xepi{p} U$ and an
    arrow $g\maps p^*A\to p^*B$ in $\bS/V$ such that $V\ff
    p^*\ph(g)$.\label{item:ff-existsarr2}
  \item{\rm(\ff$\coim_0$)} $U\ff \forall X. \ph(X)$ iff for any
    $V\too[p] U$ and any object $A\in \bS/V$, we have $V\ff
    p^*\ph(A)$.\label{item:ff-forallobj2}
  \item{\rm(\ff$\coim_1$)} Similarly, $U\ff \forall f\maps A \to B.
    \ph(f)$, where $A$ and $B$ are objects of $\bS/U$, iff for any
    $V\too[p] U$ and any arrow $p^*A\too[j] p^*B$ in $\bS/V$, we have
    $V\ff p^*\ph(j)$.\label{item:ff-forallarr2}
  \end{labellist}
  If \ph\ is a formula over $1$ (i.e.\ a formula in \bS), we say \ph\
  is \textbf{valid} if $1\ff\ph$.
\end{defn}

We occasionally write $U\ff_\bS \ph$ if we wish to emphasize the
category \bS.  We may also write $V\ff \ph$ instead of $V\ff p^*\ph$
if the map $p$ is obvious from context.

We now prove a couple of basic lemmas about the stack semantics.

\begin{lem}\label{thm:forcing-basics}
  Let \ph\ be a sentence over $U$ in a positive Heyting category \bS.
  \killspacingtrue
  \begin{enumerate}
  \item If $\ph'$ is an isomorph of \ph, then $U\ff \ph'$ if and only
    if $U\ff \ph$.\label{item:forcing-4}
  \item If $U\ff \ph$, then for any $V\too[p] U$ we have $V\ff
    p^*\ph$.\label{item:forcing-1}
  \item If $V\xepi{p} U$ is a regular epimorphism and $V\ff p^*\ph$, then
    $U\ff \ph$.\label{item:forcing-2}
  \item If $U$ is initial, then $U\ff\ph$ for any
    \ph.\label{item:forcing-0}
  \item If $U=V\cup W$ with $V\ff\ph$ and $W\ff\ph$, then $U\ff
    \ph$.\label{item:forcing-3}
  \end{enumerate}
  \killspacingfalse
\end{lem}
\begin{proof}
  These are all inductions on formulas.  \ref{item:forcing-4} is easy,
  just like \autoref{thm:isoinvar-truth}, and \ref{item:forcing-1}
  needs only the fact that regular epis and unions are stable under
  pullback and initial objects are strict (that is, any map $V\to 0$
  is an isomorphism).  Strictness of initial objects also immediately
  implies \ref{item:forcing-0}.

  For~\ref{item:forcing-2}, \top, \bot, and \meet\ are trivial.
  Atomic formulas, \imp, and \coim\ follow from pullback-stability of
  regular epimorphisms.  For \join, we take images, and for \im, we
  just compose regular epis.

  \ref{item:forcing-3} is only slightly more involved.  Once again,
  \top, \bot, and \meet\ are trivial, while atomic formulas, \imp, and
  \coim\ follow from pullback-stability of unions.  \join\ is also
  easy since unions distribute over themselves.  For $\im_0$, if we
  have $U=V\cup W$ with regular epis $V'\epi V$ and $W'\epi W$, and
  objects $X\in\bS/V'$ and $Y\in\bS/W'$ such that $V'\ff\ph(X)$ and
  $W'\ff\ph(Y)$, then because coproducts are disjoint and stable, $X+Y
  \in\bS/(V'+W')$ pulls back to $X$ and $Y$ over $V$ and $W$,
  respectively; thus by the inductive hypothesis, $V'+W'\ff\ph(X+Y)$.
  And $V'+W'\to U$ is regular epic, so $U\ff\exists X.\ph(X)$.  The
  case of $\im_1$ is analogous.  
\end{proof}

We refer to \autoref{thm:forcing-basics}\ref{item:forcing-2},
\ref{item:forcing-0}, and~\ref{item:forcing-3} as \textbf{descent} of
forcing.  Collectively, they say that ``forcing descends along finite
jointly effective-epimorphic families''.

\begin{lem}\label{thm:deduction}
  The usual rules of deduction for intuitionistic logic are sound for
  \ff.  In other words, if $\ph\pp\psi$ is provable and $U\ff\ph$,
  then also $U\ff\psi$.
\end{lem}
\begin{proof}
  Straightforward verification; see, for
  instance,~\cite[D1.3.1]{ptj:elephant} for the axioms we have to
  prove.  The only slightly nontrivial axioms are the rules for
  falsity, disjunction, and existential quantification.  The axiom
  \[\bot\imp\ph\]
  for falsity follows from
  \autoref{thm:forcing-basics}\ref{item:forcing-0}.  For the axiom
  \[(\ph\imp\chi)\impp((\psi\imp\chi)\imp((\ph\join\psi)\imp\chi))\]
  we must check that if $U\ff(\ph\imp\chi)$, $U\ff(\psi\imp\chi)$, and
  $U\ff(\ph\join\psi)$, then $U\ff\chi$.  But $U\ff(\ph\join\psi)$
  means $U=V\cup W$ with $V\ff\ph$ and $W\ff\psi$, while the other two
  hypotheses imply $V\ff\chi$ and $W\ff\chi$; so $U\ff\chi$ follows by
  descent of forcing
  (\autoref{thm:forcing-basics}\ref{item:forcing-3}).  Finally, for
  the axiom
  \[\forall X.(\ph(X)\imp\psi) \impp (\exists X.\ph(X)\imp\psi)\]
  we must check that if $U\ff\forall X.(\ph(X)\imp\psi)$ and
  $U\ff\exists X.\ph(X)$, then $U\ff\psi$.  But $U\ff\exists X.\ph(X)$
  means we have a regular epimorphism $V\xepi{p} U$ and $A\in \bS/V$ with
  $V\ff p^*\ph(A)$.  The other assumption then implies that $V\ff
  p^*\psi$, so $U\ff\psi$ again follows by descent of forcing
  (\autoref{thm:forcing-basics}\ref{item:forcing-2}).  Quantification
  over arrows is identical.
\end{proof}

Now we consider the relationship of the stack semantics to the usual
internal logic, by way of representing objects.

\begin{defn}\label{def:complete-logic}
  Let \bS\ be a positive Heyting category and \ph\ a sentence in \bS\
  over $U$.  We say that a subobject $\mm{\ph}\mono U$
  \textbf{represents} or \textbf{classifies} \ph\ if for any $V\too[p]
  U$,
  \[V\ff p^*\ph \quad\iff\quad p \text{ factors through } \mm{\ph}.
  \]
\end{defn}

Evidently \mm{\ph} is unique up to isomorphism, if it exists.  As
suggested earlier, we can then prove that the usual constructions used
in the internal logic do produce such classifying subobjects.

\begin{prop}\label{thm:ddo-classif}
  Every \ddo-sentence in a positive Heyting category is classified.
\end{prop}
\begin{proof}
  First note that a \ddo-sentence over $U$ can be translated into a
  \ddo-sentence in \bS\ with one free variable $u\in U$.  We interpret
  parameters in $\bS/U$ as parameters in \bS\ by ignoring their
  structure maps to $U$, while for any $(X\too[f] U) \in \bS/U$ we
  modify the \ddo-quantifiers $\forall x\in X.\ph$ and $\exists x\in
  X.\ph$ to read instead $\forall x\in X.((f x = u) \imp \ph)$ and
  $\exists x\in X.((f x = u) \meet \ph)$.  Now we further translate
  this formula along the equivalence of \autoref{thm:ddo-corresp} to
  obtain a formula $\check{\ph}(u)$ in the usual internal logic of
  \bS, with one free variable $u:U$.  Comparing definitions shows that
  for any $p\colon V\to U$, we have $V\ff p^*\ph$ in the stack
  semantics exactly when $V\ff\check{\ph}(p)$ in the usual
  Kripke-Joyal semantics.  It then follows, by the usual
  characterization of Kripke-Joyal semantics, that the standard
  representing subobject of $\check{\ph}$ also classifies \ph\ in the
  above sense.
\end{proof}

In particular, the usual internal logic of a positive Heyting category
is exactly the \ddo-fragment of the stack semantics.  More precisely:

\begin{cor}
  A formula in the internal logic of \bS\ is valid if and only if its
  translate under the correspondence of \autoref{thm:ddo-corresp} is
  valid in the stack semantics.
\end{cor}

If \bS\ is a \Pi-pretopos, then classifiers for all arrow-quantifiers
can be constructed from those for \ddo-quantifiers, since morphisms
$1\to Y^X$ are the same as morphisms $X\to Y$.  Likewise, if \bS\ is a
topos, we can classify quantifiers over subobjects.  However, there is
no hope to classify arbitrary object-quantifiers in this way, and in
general not all sentences will be classified.  This suggests the
following definition.

\begin{defn}
  A positive Heyting category is \textbf{autological} if all sentences
  over all objects of \bS\ are classified.
\end{defn}

In \autoref{thm:clog-sep} we will show that \bS\ is autological
precisely when its stack semantics validates the structural axiom of
separation from \S\ref{sec:more-sst}.  But first, we need to
investigate more generally how validity in the stack semantics
relates to ``external'' properties of \bS\ itself.

Like the ordinary internal logic which it generalizes, validity in the
stack semantics can be interpreted as ``local'' truth.  In particular,
for ``statements which are their own localizations,'' validity in the
stack semantics is equivalent to external truth.
Chief among ``statements which are their own localizations'' are
assertions of \emph{pullback-stable universal properties}.  More
generally, by interpreting universal properties in the stack semantics
we transform them into their ``localizations.''  In order to make this
precise, we will reuse the notions of context and instantiation from
the (closely related) discussion of replacement axioms in
\S\ref{sec:strong-ax}.

\begin{thm}\label{thm:prestack-univ}
  Let \bS\ be a positive Heyting category, let $(\Gm,\De)$ be a
  context in \bS\ such that $\De$ consists only of arrow-variables,
  and let $\ph(\Gm)$ and $\psi(\Gm,\De)$ be formulas.  Then the
  following are equivalent.
  \killspacingtrue
  \begin{enumerate}
  \item $1\ff$ \qq{for any $\vec{A}:\Gm$ such that $\ph(\vec{A})$,
      there exists a unique $\vec{f}:\De$ extending $\vec{A}$ such
      that $\psi(\vec{A},\vec{f})$.}\label{item:psu1}
  \item For any $U$ and any $\vec{A}:\Gm$ in $\bS/U$ such that $U\ff
    \ph(\vec{A})$, there exists a unique $\vec{f} : \De$ in $\bS/U$
    extending $\vec{A}$ such that $U\ff
    \psi(\vec{A},\vec{f})$.\label{item:psu2}
  \end{enumerate}
  \killspacingfalse
\end{thm}
\begin{proof}
  By the forcing interpretations of $\coim$ and $\im$,
  statement~\ref{item:psu1} is equivalent to saying that for any $U$
  and any $\vec{A}:\Gm$ in $\bS/U$ such that $U\ff \ph(\vec{A})$,
  there is a regular epi $V\xepi{p} U$ and an $\vec{f}:\De$ in
  $\bS/V$ extending $p^*\vec{A}$ such that $V\ff
  \ph(p^*\vec{A},\vec{f})$, and moreover for any $W\xto{q} V$ and
  $\vec{h}: \De$ in $\bS/W$ extending $q^*p^*\vec{A}$ such that $W\ff
  \ph(q^*p^*\vec{A},\vec{h})$, we have $\vec{f}=\vec{h}$.

  In particular, $\vec{f}$ is unique in $\bS/V$ such that $V\ff
  \ph(p^*\vec{A},\vec{f})$, and therefore all the arrows that make it
  up will satisfy the descent conditions over the kernel pair of $p$.
  Since any Heyting category is a prestack for its coherent topology,
  $\vec{f}$ must descend to an instantiation $\vec{g}$ of \De\
  extending $\vec{A}$ over $U$, and descent of forcing implies that
  $U\ff \ph(\vec{A},\vec{f})$; this proves~\ref{item:psu2}.  The
  converse is easy.
\end{proof}

The following example will hopefully help clarify in what sense this
theorem is about pullback-stable universal properties.

\begin{eg}\label{eg:prestack-coeq}
  Let $\xymatrix{ R \ar@<1mm>[r]^r \ar@<-1mm>[r]_s & P \ar[r]^q & Q}$
  be a diagram in \bS\ such that $q r = q s$.  Let \Gm\ be the context
  $(X, f\colon P\to X)$, let $\ph(X,f)$ assert that $f r = f s$, let
  \De\ be the context $g\colon Q\to X$ extending \Gm, and let
  $\psi(X,f,g)$ assert that $g q = f$.  (Note that $R$, $P$, $Q$, $r$,
  $s$, and $q$ are \emph{parameters} of these contexts.)  Then
  \autoref{thm:prestack-univ}\ref{item:psu1} becomes $1\ff$ \qq{$q$ is
    a coequalizer of $r$ and $s$}, while~\ref{item:psu2} becomes the
  assertion that $q$ is a pullback-stable coequalizer of $r$ and $s$.
\end{eg}

Generalizing from this example in the evident way, we can interpret
\autoref{thm:prestack-univ} as saying that $1\ff$ \qq{$\vec{B}$ has
  some universal property} if and only if $U^*\vec{B}$ has the given
universal property in $\bS/U$ for any $U\in\bS$.  Similarly, for
asserting the \emph{existence} of objects with universal properties,
we have the following.  Note that here we require exactness, just as
we did for \autoref{thm:repl-ctxt}.

\begin{thm}\label{thm:stack-univ}
  Let \bS\ be a Heyting pretopos, let $(\Gm,\De)$ be a context in \bS,
  and let $\ph(\Gm)$ and $\psi(\Gm,\De)$ be formulas.  Then the
  following are equivalent.
  \killspacingtrue
  \begin{enumerate}
  \item $1\ff$ \qq{for any $\vec{A}:\Gm$ such that
      $\ph(\vec{A})$, there exists a $\vec{B}:\De$
      extending $\vec{A}$ such that $\psi(\vec{A},\vec{B})$, and any
      two such $\vec{B}$ are isomorphic by a unique isomorphism
      extending $1_{\vec{A}}$.}\label{item:su1}
  \item For any $U$ and any $\vec{A}:\Gm$ in $\bS/U$ such that $U\ff
    \ph(\vec{A})$, there exists a $\vec{B}:\De$ in $\bS/U$ extending
    $\vec{A}$ such that $U\ff \psi(\vec{A},\vec{B})$, and any two such
    $\vec{B}$ are isomorphic by a unique isomorphism extending
    $1_{\vec{A}}$.\label{item:su2}
  \end{enumerate}
  \killspacingfalse
\end{thm}
\begin{proof}
  Just like the proof of \autoref{thm:prestack-univ}, but using the
  facts that any pretopos is a stack for its coherent topology, and
  that uniqueness of isomorphisms implies that they satisfy the
  cocycle condition.
\end{proof}

Following on from \autoref{eg:prestack-coeq}, we have the following.

\begin{eg}\label{eg:stack-coeq}
  Let \Gm\ be the context $(R,P,r\colon R\to P, s\colon R\to P)$, let
  $\ph = \top$, let \De\ be the context $(Q,q\colon P\to Q)$ extending
  \Gm, and let $\psi$ be the formula \qq{$q$ is a coequalizer of $r$
    and $s$} considered in \autoref{eg:prestack-coeq}.  Then
  \autoref{thm:stack-univ}\ref{item:su1} becomes $1\ff$ \qq{every
    parallel pair has a coequalizer}, while~\ref{item:su2} asserts
  that every slice category of \bS\ has pullback-stable coequalizers
  (for which it suffices that \bS\ itself has them).
\end{eg}

Thus, we can interpret \autoref{thm:stack-univ} as saying that $1\ff$
\qq{for all $\vec{A}$, there exists $\vec{B}$ with some universal
  property} if and only if for any $U$ and any $\vec{A}$ in $\bS/U$,
there exists $\vec{B}$ in $\bS/U$ such that $p^*(\vec{A},\vec{B})$ has
the specified universal property in $\bS/V$ for any $p\colon V\to U$.
In the future, we will apply these theorems to all sorts of universal
properties without further comment.


For statements \ph\ that do not express simple universal properties,
however, valid\-ity in the stack semantics (expressed by $1\ff\ph$) and
external truth (expressed by $\bS\ss\ph$) can have quite different
meanings.  For example, if \ph\ is \qq{every regular epi splits}, then
$\bS\ss\ph$ iff \bS\ satisfies the external axiom of choice (AC),
while $1\ff \ph$ iff \bS\ satisfies the internal axiom of choice
(IAC).  The most important example of the divergence of \ss\ and \ff,
however, is the following.

\begin{lem}\label{thm:int-wpt}\label{thm:int-coll}
  For any positive Heyting category \bS, we have $1\ff$ \qq{\bS\ is
    constructively well-pointed and satisfies collection}.
\end{lem}
\begin{proof}
  We first show that $1\ff$ \qq{\bS\ is constructively well-pointed}.  If $X\too[f] U$ is
  in $\bS/U$, $m\maps A\mono X$ is monic, and $U\ff$ \qq{every
    $1\too[x] X$ factors through $A$}, then in particular the generic
  element $\Delta_X\maps X\to f^*X$ in $\bS/X$ factors through $m$;
  hence $m$ is split epic and thus an isomorphism.  Thus $1\ff$
  \qq{$1$ is a strong generator}.  Now since the initial object is
  strict, the existence of a map $1\to 0$ is the same as saying that
  $1$ is an initial object.  Then if $U\ff$ \qq{$1_U$ is initial}, $U$
  must be initial; hence $1\ff$ \qq{$1$ is not initial}.  Similarly,
  if $U\ff$ \qq{$V\to 1_U$ is regular epic}, then $V\to U$ is in fact
  regular epic, and thus (since $V$ has a section over itself) $U\ff
  \exists v\colon 1\to V.\top$.  Finally, if $U\ff$ \qq{$U=V\cup W$},
  then $U=V\cup W$, and so since $V$ and $W$ have sections over
  themselves, we have $U\ff \exists v\colon 1\to V.\top \join \exists
  w\colon 1\to W.\top$.

  We now show $1\ff$ \qq{\bS\ satisfies collection}.  Suppose that
  $1\ff \forall u\in U.\exists X. \ph(u,X)$.  Then $U\ff\exists
  X.\ph(\Delta_U,X)$, so we have a regular epi $V\xepi{p} U$ and an
  $A\in\bS/V$ with $V\ff \ph(p,A)$.  Thus, for any $W$ and any
  $W\too[v] V$, we have $W \ff \ph(p v, v^*A)$, which is the desired
  conclusion.  Collection of functions is analogous.
\end{proof}

This validates our assertion that \emph{the stack semantics of any
  Heyting pretopos models a structural set theory}.  Therefore, from
now on we avail ourselves of \autoref{cnv:elements} when speaking in
the stack semantics.

\begin{rmk}
  Under the second approach to stack semantics, \autoref{thm:int-wpt}
  says that if \bS\ is small, any strictification of its self-indexing
  is well-pointed and satisfies collection as an internal category in
  the topos $\bSh(\bS)$ of sheaves for the coherent topology of \bS.
  Since $\bSh(\bS)$ is a coherent topos, it has enough
  points, so this is equivalent to saying that each stalk has these
  properties.  This version of \autoref{thm:int-wpt} was proven
  in~\cite{awodey:thesis}, and we will prove the analogous fact for the
  2-category of stacks on \bS\ in~\cite{shulman:2cofc}; see also \S\ref{sec:cofc}.
\end{rmk}

\begin{rmk}
  \autoref{thm:stack-univ} implies that $1\ff$ \qq{the self-indexing
    of \bS\ is a stack} is true in any Heyting pretopos.  Thus, by
  Propositions \ref{thm:replacement} and \ref{thm:repl-ctxt},
  replacement of contexts is also always valid in the stack semantics of a
  Heyting pretopos.
\end{rmk}

Moreover, any additional axiom of structural set theory that is
expressible as a pullback-stable universal property will be faithfully
represented in the stack semantics.  This includes all the structure
mentioned in \autoref{thm:iz-topos}, except for AC and full classical
logic.  In particular, we can say:
\begin{blist}
\item If \bS\ is a \Pi-pretopos with a \nno\ satisfying the ``internal
  presentation axiom'' (there are enough internal projectives), then
  its stack semantics models CETCS.
\item If \bS\ is any topos with a \nno, then its stack semantics
  models IETCS.
\end{blist}

Of course, \autoref{thm:int-coll} also says that the stack semantics
of a Heyting pretopos is not just any structural set theory, but also
satisfies collection.  In particular, from any constructively well-pointed Heyting
pretopos, we can construct another one which satisfies
collection---albeit in an exotic logic (namely, the stack semantics of
the original category).  This implies the well-known fact that
intuitionistically, the addition of collection does not change the
consistency strength of a theory.

\begin{rmk}
  Since by \autoref{thm:collsep}, adding collection to full classical
  logic \emph{does} change the consistency strength, it follows that
  full classical logic is not in general preserved by passage to the
  stack semantics.  It is not hard to show that if the stack
  semantics of \bS\ satisfies full classical logic, then \bS\ is
  necessarily autological.
\end{rmk}

In fact, well-pointedness and collection are ``precisely'' the
characteristic properties of the stack semantics, in the following
sense.

\begin{thm}\label{thm:collpin}
  Let \bS\ be a positive Heyting category; the following are equivalent.
  \killspacingtrue
  \begin{enumerate}
  \item \bS\ is constructively well-pointed and satisfies collection.\label{item:collpin1}
  \item For any $U\in\bS$ and any sentence \ph\ over $U$, we have
    $U\ff\ph$ if and only if $\bS\ss u^*\ph$ for all $1\too[u]
    U$.\label{item:collpin2}
  \end{enumerate}
  \killspacingfalse
\end{thm}
\begin{proof}
  First assume~\ref{item:collpin2}.  In particular, this means that
  for any sentence \ph\ in \bS\ we have $1\ff\ph$ if and only if
  $\bS\ss\ph$.  Since the stack semantics is always constructively
  well-pointed and satisfies collection, it follows
  from~\ref{item:collpin2} that \bS\ is so,
  proving~\ref{item:collpin1}.

  Now assume~\ref{item:collpin1}; we prove~\ref{item:collpin2} by
  structural induction on \ph.  For an atomic formula $(f=g)$ where
  $f,g\colon X\to Y$ in $\bS/U$, we can verify that $(f=g)$ is
  equivalent to $\forall x\in X. (fx=gx)$, which is a \ddo-formula and
  hence classified by \autoref{thm:ddo-classif}.  Thus,
  strong-generation of $1$ implies the result for atomic formulas.

  The cases of \top\ and \meet\ are obvious, as usual.
  If $U\ff\bot$, then $U$ is initial, and thus (since $1$ is nonempty)
  it has no global elements $1\to U$; thus $\bS\ss\bot$ for every such
  global element.  Conversely, if $U$ has no global elements, then the
  map $0\to U$ is a bijection on global elements, hence an isomorphism
  since $1$ is a strong generator.

  If $U\ff(\ph\join\psi)$, then $U=V\cup W$ with $V\ff\ph$ and
  $W\ff\psi$; hence by the inductive hypotheses, $\bS\ss v^*\ph$ for
  all $1\too[v] V$ and $\bS\ss w^*\psi$ for all $1\too[w] W$.  But
  since $1$ is indecomposable, every $1\too[u] U$ factors through
  either $V$ or $W$, hence either $\bS \ss u^*\ph$ or $\bS \ss
  u^*\psi$, so $\bS \ss u^*(\ph\join \psi)$.  Conversely, suppose $\bS
  \ss u^*(\ph\join \psi)$ for each $1\too[u] U$, and consider the
  objects $U$ and $U+U$ over $U$, which pull back to $1$ and $1+1$
  along any $1\too[u] U$.  Applying collection of functions to the
  statement \qq{either $u^*\ph$ and $1\too[f] 1+1$ is the first
    injection, or $u^*\psi$ and $1\too[f] 1+1$ is the second
    injection}, we obtain a regular epimorphism $V\xepi{p} U$ and a
  suitable map $V \to V+V$ over $V$.  This map decomposes $V$ as a
  coproduct $W+X$, and by the inductive hypothesis we have $W\ff
  p^*\ph$ and $X\ff p^*\psi$; hence $U\ff (\ph\join\psi)$.

  If $U\ff (\ph\imp\psi)$, then for any $1\too[u] U$, if $\bS\ss
  u^*\ph$, then by the inductive hypothesis $1\ff u^*\ph$ and so $1\ff
  u^*\psi$, hence $\bS \ss u^*\psi$ by the inductive hypothesis.  Thus
  $\bS \ss u^*(\ph\imp\psi)$.  Conversely, if $\bS \ss
  u^*(\ph\imp\psi)$ for all $1\too[u] U$, then for any $V\too[p] U$
  with $V\ff p^*\ph$, for any $1\too[v] V$ we have $1 \ff (p v)^*\ph$,
  hence (by the inductive hypothesis) $\bS \ss (p v)^*\ph$, so $\bS \ss
  (p v)^*\psi$, and thus $1\ff (p v)^*\psi$.  By the inductive
  hypothesis, we have $V\ff p^*\psi$; hence $U\ff (\ph\imp\psi)$.

  If $U\ff \exists X. \ph(X)$, then we have a regular epimorphism $V
  \xepi{p} U$ and an $A\in \bS/V$ with $V\ff \ph(A)$.  But for any
  $1\too[u] U$, by projectivity of $1$ there is a $1\too[v] V$ with $p
  v = u$, and $1 \ff v^*p^*\ph(v^* A)$, hence $\bS \ss (p
  v)^*\ph(v^*A)$, and thus $\bS \ss \exists X. (p v)^*\ph(X)$.
  Conversely, if $\bS \ss \exists X. u^*\ph(X)$ for each $1\too[u] U$,
  then by collection of sets, there exists a regular epimorphism
  $V\xepi{p} U$ and an $A\in \bS/V$ such that $\bS \ss (p v)^*
  \ph(v^*A)$ for all $1\too[v] V$, whence $V\ff p^*\ph(A)$ by the
  inductive hypothesis, and so $U\ff \exists X.\ph(X)$.  Existential
  quantification of morphisms is analogous.

  Finally, if $U\ff \forall X.\ph(X)$, then for any $1\too[u] U$ and
  $A\in \bS$ we have $1\ff u^*\ph(A)$, hence $\bS \ss u^*\ph(A)$ by
  the inductive hypothesis, and so $\bS \ss \forall X. u^*\ph(X)$.
  Conversely, if $\bS \ss \forall X. u^*\ph(X)$ for all $1\too[u] U$,
  then for any $V\too[p] U$ and $B\in \bS/V$, we have $\bS \ss (p v)^*
  \ph(v^*B)$ for any $1\too[v] V$, hence $V\ff p^*\ph(B)$ by the
  inductive hypothesis; hence $U\ff \forall X.\ph(X)$.
\end{proof}


\begin{cor}
  The theory of a constructively well-pointed Heyting pretopos satisfying collection is
  complete for stack semantics over Heyting pretoposes.  More
  precisely, if a statement is valid in the stack semantics of any
  Heyting pretopos (in any intuitionistic metatheory), then it can be
  proved from the axioms of a constructively well-pointed Heyting pretopos satisfying
  collection.
\end{cor}
\begin{proof}
  By standard completeness theorems for intuitionistic logic, it
  suffices to show that if $1\ff_\bS \ph$ for any Heyting pretopos
  \bS\ (in any intuitionistic metatheory), then $\Set\ss\ph$ for any
  constructively well-pointed Heyting pretopos \Set\ that satisfies collection.  However,
  for any such \Set, \autoref{thm:collpin} implies that $\Set \ss \ph$
  if and only if $1\ff_{\Set} \ph$, and the latter is true by
  assumption.
\end{proof}

We now turn to our promised characterization of the axiom of separation in the
stack sem\-antics.  First we need to know that the stack semantics is
``local'' and ``idempotent''.

\begin{lem}\label{thm:ss-loc}
  For any $V\too[p] U$ and any sentence \ph\ over $V$ in \bS, we have
  $V\ff_{\bS}\ph$ if and only if $(V\xto{p} U) \ff_{\bS/U} \ph$.
\end{lem}
\begin{proof}
  Straightforward from the definition, since unions and regular epis
  in $\bS/U$ are created in \bS.
\end{proof}

\begin{lem}\label{thm:int-idem}
  For any sentence \ph\ over $V$ and any $V\too[p] U$, we have $U\ff$
  \qq{$V\ff\ph$} if and only if $V\ff\ph$.
\end{lem}
\begin{proof}
  By \autoref{thm:ss-loc}, it suffices to assume that $U=1$.  And
  since the stack semantics of \bS\ satisfies collection, by
  \autoref{thm:collpin} we have
  \begin{equation}
    1\ff \tqq{$V\ff\ph$ if and only if $v^*\ph$ holds for all $1\too[v] V$}\label{eq:ii}
  \end{equation}
  In the ``only if'' direction,~\eqref{eq:ii} says that if $W$ is such that for
  any $Z\too[q] W$ and $Z\too[v] V$ we have $Z\ff v^*\ph$, then $W\ff
  W^*($\qq{$V\ff\ph$}$)$, i.e.\ $W\ff$ \qq{$W^* V \ff W^*\ph$}.  In
  particular, for $W=1$, this says that if for any $Z\too[v] V$ we
  have $Z\ff v^*\ph$, then $1\ff$ \qq{$V\ff\ph$}.  But this hypothesis
  is satisfied as soon as $V\ff\ph$, so $V\ff\ph$ implies $1\ff$
  \qq{$V\ff\ph$}.

  In the ``if'' direction,~\eqref{eq:ii} says that for any $W\too[v]
  V$, if $W\ff$ \qq{$W^*V\ff W^*\ph$}, then $W\ff v^*\ph$.  In
  particular, for $W=V$ and $v=1_V$, it says that if $V\ff$
  \qq{$V^*V\ff V^*\ph$}, then $V\ff \ph$.  But if $1\ff$
  \qq{$V\ff\ph$}, then also $V\ff$ \qq{$V^*V\ff V^*\ph$} by pullback,
  so $1\ff$ \qq{$V\ff\ph$} implies $V\ff \ph$ as desired.
\end{proof}

Recall that we define \bS\ to be \emph{autological} if all sentences
over all objects of \bS\ are classified in the sense of
\autoref{def:complete-logic}.

\begin{thm}\label{thm:clog-sep}
  Let \bS\ be a Heyting pretopos; the following are equivalent.
  \killspacingtrue
  \begin{enumerate}
  \item \bS\ is autological.\label{item:clogsep1}
  \item $1\ff$ \qq{\bS\ is autological}\label{item:clogsep2}
  \item $1\ff$ \qq{\bS\ satisfies the axiom of separation}.\label{item:clogsep3}
  \item (If \bS\ is constructively well-pointed) \bS\ satisfies separation and collection.\label{item:clogsep4}
  \end{enumerate}
  \killspacingfalse
\end{thm}
\begin{proof}
  \autoref{thm:int-idem} implies that autology is a pullback-stable
  universal property, so \autoref{thm:stack-univ} implies
  that~\ref{item:clogsep1}$\Iff$\ref{item:clogsep2}.  And by
  \autoref{thm:collpin}, the stack semantics is always constructively well-pointed and
  satisfies collection, so if we show
  that~\ref{item:clogsep1}$\Iff$\ref{item:clogsep4}, then
  internalizing the same argument will
  show~\ref{item:clogsep2}$\Iff$\ref{item:clogsep3}.

  Suppose \bS\ is constructively well-pointed and autological; we show it satisfies
  \autoref{thm:collpin}\ref{item:collpin2}.  Note that the inductive
  proof of
  \ref{thm:collpin}\ref{item:collpin1}$\imp$\ref{item:collpin2} only
  used collection in two cases, $\join$ and $\im$, and only in the
  $(\ss)\imp(\ff)$ direction.  Thus it will suffice to prove
  \ref{thm:collpin}\ref{item:collpin2} for formulas of the form
  $(\ph\vee\psi)$, $\exists X.\ph(X)$, or $\exists f\colon X\to
  Y.\ph(f)$, in each case assuming inductively that
  \ref{thm:collpin}\ref{item:collpin2} holds for the constituent
  formulas.

  Note first that all of these statements are obvious if $U=1$.  In
  the case of $\vee$, if $\bS \ss (\ph\vee\psi)$, then $\bS\ss\ph$ or
  $\bS\ss\psi$, whence by induction $1\ff\ph$ or $1\ff\psi$, and thus
  $1\ff (\ph\vee\psi)$ via one of the decompositions $1= 1\cup 0$ or
  $1=0\cup 1$.  Similarly, in the second case, if $\bS\ss \exists
  X.\ph(X)$, then there is some $A\in\bS$ with $\bS\ss\ph(A)$, whence
  by induction $1\ff\ph(A)$, and so $1\ff \exists X.\ph(X)$ via the
  cover $1\epi 1$.  Quantification over arrows is identical.

  Now if \chi\ is one of the formulas in question over a general $U$,
  by autology we can form $\mm{\chi}\mono U$.  By hypothesis, for any
  $1\too[u] U$ we have $\bS \ss u^*\chi$, so the above remark shows
  that $1\ff u^*\chi$, whence $u$ factors through $\mm{\chi}$.  But
  $1$ is a strong generator, so $\mm{\chi}$ is all of $U$ and thus
  $U\ff\chi$, as desired.

  Thus, by \autoref{thm:collpin}, \bS\ satisfies collection, so it
  remains to prove separation.  Suppose given $U$ and $\ph(u)$, and
  let $\psi \equiv U^*\ph(\Delta_U)$.  Then \psi\ is a sentence over
  $U$ such that $u^*\psi$ is equivalent to $\ph(u)$ for any $1\too[u]
  U$.  Now consider $\mm{\psi}$.  Then for any $1\too[u] U$, by
  \autoref{thm:collpin}, $\bS\ss\ph(u)$ is equivalent to $1\ff
  \ph(u)$, and thus to $1\ff u^*\psi$, which is true iff $u$ factors
  through $\mm{\psi}$; thus $\mm{\psi}$ satisfies the conclusion of
  the separation axiom.  Thus we
  have~\ref{item:clogsep1}$\imp$\ref{item:clogsep4}.

  Now assume \bS\ is well-pointed and satisfies separation and collection.
  For any sentence \ph\ over $U$, we apply separation to find a
  subobject $S\mono U$ such that $1\too[u] U$ factors through $S$ if
  and only if $\bS\ss u^*\ph$.  By \autoref{thm:collpin}, $S\ff \ph$.
  And if $V\too[p] U$ is such that $V\ff p^*\ph$, then for any
  $1\too[v] V$, $pv$ factors through $S$, which (since \bS\ is well-pointed)
  implies that $V$ factors through $S$.  Thus, $S=\mm{\ph}$, so \bS\
  is autological.
\end{proof}

In particular, the implication
\ref{item:clogsep4}$\imp$\ref{item:clogsep1} of \autoref{thm:clog-sep}
shows that if \bV\ satisfies IZF, then $\bbSet(\bV)$ is autological,
giving us our first examples of autological topoi.  We now present a
quick proof that any complete topos over IZF (including any
Grothendieck topos) is autological.

Fix a model \bV\ of IZF.  As usual in material set theory, by a
\emph{large category} we mean a category whose objects and arrows form
proper classes in the sense of \bV, i.e.\ consist of the sets in \bV\
satisfying some first-order formulas.  Recall that a topos is
well-powered if and only if it is locally small, and that if it is
complete it is also cocomplete.  Note that any cocomplete topos also
has a \nno, namely the countable copower of $1$.  The examples of most
interest are of course Grothendieck toposes, i.e.\ categories of
sheaves on a small site.

\begin{lem}\label{thm:forcing-5}
  If \bS\ is a complete and well-powered topos over a model of IZF and
  $U = \bigcup_{i\in I} V_i$ in \bS\ with $V_i\ff \ph$ for all $i\in
  I$, then $U\ff\ph$.
\end{lem}
\begin{proof}
  This is, of course, an extension of
  \autoref{thm:forcing-basics}\ref{item:forcing-3} to the infinitary
  case, but the proof requires the use of collection in \bV.  The
  cases of atomic formulas, \top, \bot, and \meet\ are again easy,
  while \join\ follows from distributivity of unions over themselves,
  and \imp\ and \coim\ follow from pullback-stability of unions.

  For $\im$, if we have $V_i\ff \exists X.\ph(X)$ for all $i\in I$,
  then for each $i$ there is an epimorphism $W\epi V$ and an $A \in
  \bS/W$ such that $W\ff \ph(A)$.  By collection in \bV, we have a set
  \cQ\ such that for every $i\in I$ there is a quadruple $(i,W,p,A)\in
  \cQ$ such that $p\colon W\epi V_i$ is an epimorphism, $A\in\bS/W$,
  and $W\ff\ph(A)$.  Let $B = \coprod_{(i,W,p,A)\in\cQ} A$ and $Z =
  \coprod_{(i,W,p,A)\in\cQ} W$, which exist since \bS\ is cocomplete.
  We then have an induced epimorphism $q\colon Z\epi U$, and since $Z
  = \bigcup_{(i,W,p,A)\in\cQ} W$ and $B$ pulls back over each $W$ to
  the corresponding $A$, the inductive hypothesis implies that
  $Z\ff\ph(B)$.  Therefore, $U\ff\exists X.\ph(X)$.  Existential
  quantification over arrows is analogous.
\end{proof}

\begin{thm}\label{thm:cplt-aut}
  Any complete and well-powered topos over a model of IZF is auto\-logical.
\end{thm}
\begin{proof}
  Let \bS\ be such a topos and \ph\ a sentence over $U$ in \bS.  Using
  separation in \bV\ and well-poweredness of \bS, let $K$ be a set of
  subobjects of $U$ representing each isomorphism class $V\mono U$
  such that $V\ff \ph$.  Since \bS\ is complete, $K$ has a union; call
  it $S$.  By \autoref{thm:forcing-5}, $S\ff\ph$.  And given any
  $W\too[p] U$ such that $W\ff\ph$, if $R$ is the image of $p$, then
  $R\ff\ph$, so $R\iso V$ for some $V\in K$; hence $R\subseteq S$ and
  thus $p$ factors through $S$.  Thus, $S$ classifies \ph, so \bS\ is
  autological.
\end{proof}

\begin{rmk}
  This implies that in a Grothendieck topos of sheaves on some site,
  the definition of the stack semantics can be rephrased in terms of a
  forcing relation defined only over the site, just as for the usual
  internal logic.  We omit the proof, which is very much like those of
  Theorems \ref{thm:collpin} and \ref{thm:cofc-logic}.
\end{rmk}

We will see another way of proving autology in \S\ref{sec:cofc}: the
category of small objects in any category of classes which satisfies
separation is autological.  In~\cite{shulman:uqsa} we will prove
directly that autology is preserved by many other topos-theoretic
constructions.

\begin{rmk}
  We have not mentioned the interpretation of fullness, induction, or
  well-founded induction in the stack semantics.  Since these do not
  express simple universal properties, their stack-semantics versions
  are not generally equivalent to their ``external'' versions.  One
  can, of course, write out explicitly in terms of \bS\ what each of
  these axioms means in the stack semantics, and the result will be the
  appropriate version of that axiom for a non-well-pointed category.  For
  example, the assertion that $1\ff$ \qq{\bS\ satisfies fullness} is
  equivalent to the categorical version of fullness given
  in~\cite{vdbdb:nonwellfounded,vdbm:pred-i-exact} (minus their
  `smallness' conditions).
\end{rmk}

\section{Material set theories in the stack semantics}
\label{sec:mater-set-theor}

We can now combine \S\S\ref{sec:constr-mat} and
\ref{sec:internal-logic} in the expected way: from any Heyting
pretopos \bS, we obtain an interpretation of material set theory by
interpreting the theory of \S\ref{sec:constr-mat} in the stack
semantics of \bS.  In particular, we can construct interpretations of
material set theory by starting with a model \bV, passing to the
category $\bbSet(\bV)$, performing some category-theoretic
construction on $\bbSet(\bV)$ (which generally destroys
well-pointedness), then constructing the model of
\S\ref{sec:constr-mat} in the stack semantics of the resulting
category.

There are several important questions to ask about this construction.
The first is, which of the relevant category-theoretic properties of
$\bbSet(\bV)$ are preserved by the constructions in question, and
which others can be ``forced'' to hold or fail by a clever choice of
such a construction?  To study this in generality is not our aim in
this paper, but we remark briefly on what is known.
\killspacingtrue
\begin{enumerate}
\item The most-studied case is that of elementary topoi, which are
  well-known to be preserved by all sorts of constructions, such as
  sheaves (on internal sites), coalgebras for left-exact comonads,
  Artin gluing, realizability, and filterquotients.  Natural numbers
  objects are also usually preserved.
\item Realizability and sheaf constructions on \Pi-pretoposes (often
  also satisfying some additional axioms, such as fullness) are
  studied
  in~\cite{mp:wftrees,mp:ttcst,vdb:ind-exact,aglw:shvs-ast,vdb:pred-sheaves,vdbm:pred-i-exact,vdbm:pred-ii-realiz,vdbm:pred-iii-shvs},
  among other places (see also~\cite{ast-website}).
\item Of course, Booleanness and the axiom of choice are not usually
  preserved by any of these constructions.
\item On the other hand, we have seen that collection is always
  satisfied in the stack semantics.
\item We will prove in~\cite{shulman:uqsa} that the axiom of
  separation, in its stack-semantics form (i.e.\ \emph{autology} as
  defined in \S\ref{sec:internal-logic}) is also preserved by the
  constructions of sheaves, coalgebras, gluing, realizability, and
  filterquotients.
\item We are not aware of any results regarding the preservation of
  full induction or well-founded induction.
\end{enumerate}
\killspacingfalse

The second important question to ask is, how does this approach
compare to existing category-theoretic constructions of
interpretations of material set theory?  As mentioned in the
introduction, such interpretations fall into two groups: the older
approach of Fourman~\cite{fourman:sheaf-setthy} and
Hayashi~\cite{hayashi:setthy-topos}, and the newer approach of
algebraic set theory begun by Joyal and Moerdijk~\cite{jm:ast}.  The
next two sections are devoted to comparing our models with these;
under general hypotheses they turn out to be equivalent.

The third important question to ask is, how does this approach compare
with well-known constructions of models that remain entirely within
the world of material set theory?  We will not study this question
here, but a partial answer to it can be extracted from our answers to
the second question.

For instance, in~\cite{bs:freyd-ac,bs:cplt-rep-setth} it is shown,
using a construction of Freyd~\cite{freyd:all-localic}, that the
Fourman-Hayashi interpretation in any Boolean Grothendieck topos
(defined over a model of ZFC) can be identified with a certain
symmetric submodel of a Boolean-valued model of ZF.
(Fourman~\cite{fourman:sheaf-setthy} already observed this in
particular cases.)  It then follows that the stack-semantics model in
any such Grothendieck topos can also be so identified.
Hayashi~\cite{hayashi:setthy-topos} also asserted that his
interpretation in a topos of sheaves on a complete Heyting algebra
agrees with the corresponding Heyting-valued model of IZF, as
in~\cite{grayson:heyt-ist}.  Since Freyd's construction also applies
to non-Boolean topoi, it seems likely that the Fourman-Hayashi
interpretation in an arbitrary Grothendieck topos can similarly be
compared to symmetric submodels of Heyting-valued models as studied
in~\cite{tt:heytval-ist}, but to my knowledge this has not been
written down.  Likewise, in~\cite{kgvo:ast-eff,vdbm:pred-ii-realiz} it
is shown that the model of material set theory constructed via
algebraic set theory from a realizability topos agrees with McCarty's
realizability model of IZF~\cite{m:realiz-recurs}, and thus so does
the stack-semantics model in a realizability topos.  In principle, it
should be possible to give a direct comparison between stack-semantics
models and material-set-theoretic models, but we will not do so here.

\section{The cumulative hierarchy in a complete topos}
\label{sec:fourman}

In this section we will show that in good cases, our interpretation of
material set theory in the stack semantics is equivalent to the
interpretations of Fourman~\cite{fourman:sheaf-setthy} and
Hayashi~\cite{hayashi:setthy-topos}.  In fact, Hayashi defines his
interpretation by way of a Kripke-Joyal semantics, which is thus quite
evidently the same as the stack-semantics approach.  Fourman uses
instead representing subobjects, which he can only construct in a
complete and well-powered topos.  Additionally, both focus on a
version of the ``von Neumann hierarchy'' constructed internally to the
topos, which our approach shows to be unnecessary.

Let us begin by reformulating the stack-semantics interpretation of
material set theory in terms of elements rather than slice categories.
Let \bS\ be a topos with a \nno, let $A$ be a well-founded extensional
graph in \bS, and let $x\colon U\to A$ be a morphism.  Let $\xbar$
denote the subobject of $U^* A$ such that in the internal logic of
$\bS/U$, \xbar\ contains precisely those $y\in A$ which admit a path
to $x$.

\begin{lem}\label{thm:hayashi-apg}
  Let $x$ and \xbar\ be as above, and similarly $y\colon U\to A$
  and \ybar.
  \killspacingtrue
  \begin{enumerate}
  \item $U\ff$ \qq{\xbar\ is a well-founded extensional \apg}.\label{item:fa1}
  \item $U \ff (\xbar\cong\ybar)$ if and only if $x=y$.\label{item:fa2}
  \item $U \ff (\xbar \iin\ybar)$ if and only if $x\prec y$.\label{item:fa3}
  \item If $f\colon A \into B$ is a simulation, then $U \ff (\xbar
    \cong \overline{f x})$.\label{item:fa4}
  \end{enumerate}
  \killspacingfalse
  Furthermore, every well-founded extensional \apg\ in $\bS/U$ is
  isomorphic to one of the form \xbar, for some well-founded
  extensional graph $A$ in \bS\ and some $x\colon U\to A$.
\end{lem}
\begin{proof}
  Since $U^*\colon \bS\to\bS/U$ is a logical functor, and
  well-foundedness and extensionality are \ddo-properties in a topos,
  $U\ff$ \qq{$U^* A$ is a well-founded extensional graph}.  The
  proof of~\ref{item:fa1} is then obvious by arguing internally,
  while~\ref{item:fa4} follows from \autoref{thm:sim-inj}.  The ``if''
  directions of~\ref{item:fa2} and~\ref{item:fa3} are easy, while the
  ``only if'' directions can be shown by well-founded induction.

  Finally, suppose $X$ is a well-founded extensional \apg\ in
  $\bS/U$.  Arguing in the stack semantics of \bS, consider $X$ as a
  graph in \bS\ itself and form $X+1$, with $\prec$ inherited from $X$
  and with $\star_u \prec \star$, where for each $u\in U$, $\star_u$
  denotes the root of the fiber $u^*X$.  Then $X+1$ satisfies the
  hypotheses of \autoref{thm:ext-quotient} with $n=1$; let $A$ be its
  extensional quotient.  The section $U\to X$ (which equips $X$ with
  its root, as an \apg\ in $\bS/U$) induces a map $x\colon U\to A$ and
  it is easy to verify that $X\cong \xbar$.
\end{proof}

\begin{rmk}
  Since isomorphisms between well-founded extensional graphs are
  unique when they exist, $U\ff (\xbar\cong\ybar)$ if and only if in
  fact $\xbar\cong\ybar$ in $\bS/U$, and similarly $U \ff (\xbar
  \iin\ybar)$ if and only if there exists a $z\colon U\to \ybar$ such
  that $\xbar \cong \ybar/z$ in $\bS/U$.
\end{rmk}

\begin{rmk}
  Combining \autoref{thm:hayashi-apg}\ref{item:fa2}
  with~\ref{item:fa4}, we see that for $x\colon U\to A$ and $y\colon
  U\to B$, we have $\xbar\cong\ybar$ iff $f x = g y$ for some (hence
  any) well-founded extensional graph $C$ with simulations $f\colon
  A\to C$ and $g\colon B\to C$, and similarly for membership.  Note
  that for any any such pair $A,B$ there exists such a $C$ with $f$
  and $g$, such as the extensional quotient of $A+B+1$.
\end{rmk}

\begin{thm}\label{thm:hayashi-stack}
  The stack semantics interpretation of material set theory in a
  topos can equivalently be described as follows.  Instead of
  well-founded extensional \apgs\ in $\bS/U$, the parameters at
  stage $U$ are morphisms $x\colon U\to A$, where $A$ is any
  well-founded extensional graph in \bS.  The forcing semantics of
  formulas \ph\ is defined inductively by:
  \begin{blist}
  \item For $x\colon U\to A$ and $y\colon U\to B$, we say $U\ff (x=y)$
    if $f x = g y$ for some (hence any) simulations $f\colon A\to C$
    and $g\colon B\to C$.
  \item Likewise, $U\ff (x\in y)$ if $f x \prec g y$ for $f$ and $g$ as above.
  \item The connectives are interpreted in the usual Kripke-Joyal way.
  \item We say $U\ff \exists x.\ph(x)$ if there is an epimorphism
    $V\epi U$ and an $x\colon V\to A$, for some well-founded
    extensional graph $A$, such that $V\ff \ph(x)$.
  \item Likewise, $U\ff \forall x.\ph(x)$ if $V\ff \ph(x)$ for any
    $V\to U$, any well-founded extensional graph $A$, and any map
    $x\colon V\to A$.
  \end{blist}
\end{thm}
\begin{proof}
  Straightforward induction on formulas, using \autoref{thm:hayashi-apg}.
\end{proof}

This is exactly the interpretation of
Hayashi~\cite{hayashi:setthy-topos}, except that he considers only
morphisms $x\colon U\to A$ where $A$ belongs to some specified class
\sU\ of well-founded extensional graphs.  Of course, some closure
conditions on \sU\ are then necessary to ensure that the axioms of
material set theory are satisfied.  Hayashi assumes that:
\begin{blist}
\item For any $A,B\in\sU$ there is a $C\in \sU$ with simulations $A\into
  C$ and $B\into C$, and
\item For any $A\in \sU$ there is a $B\in\sU$ with a simulation $P A
  \into B$.
\end{blist}
Here $P A$ is the power-object of $A$, equipped with the well-founded
extensional relation $\prec$ defined internally as follows: for
subsets $H,K\subseteq A$, we have $H\prec K$ iff there exists $x\in K$
such that $H = \setof{y | y\prec x}$.  Hayashi calls a \sU\ satisfying
these conditions a \textbf{pre-universe}, and proves that the above
semantics valued in any pre-universe validates the core axioms
together with power sets, foundation, and transitive closures, plus
infinity if there is a \nno, and \ddo-classical logic if the topos is
Boolean.  We thus record:

\begin{thm}\label{thm:hayashi-interp}
  Let \bS\ be a topos and let \sU\ be the maximal pre-universe
  consisting of all well-founded extensional graphs.  Then the
  interpretation of material set theory in the stack semantics of \bS\
  is identical to Hayashi's interpretation valued in \sU.\qed
\end{thm}

Hayashi also defines a \textbf{universe} to be a pre-universe such
that
\begin{blist}
\item For any (small) set $S\subseteq \sU$, there is a $B\in \sU$ such
  that for all $A\in S$ there exists a simulation $A\to B$.
\end{blist}
Of course, this is no longer an elementary property; it only makes
sense when working in some ambient set theory.  Note that the class of
\emph{all} well-founded extensional graphs need not, in general, be a
universe.  It is a universe, however, if \bS\ is \emph{cocomplete}
relative to the external set theory, since then $B$ can be obtained as
the extensional quotient of $\coprod_{A\in S} A + 1$.

Hayashi proves that the axiom of collection is satisfied if \bS\ is
well-powered and \sU\ is a universe.  Our approach via the stack
semantics shows that in fact, these hypotheses are unnecessary: the
axiom of collection is \emph{always} satisfied, at least for the
maximal pre-universe.  Hayashi also claims that the axiom of
separation is valid as soon as the subobject lattices of \bS\ are
complete Heyting algebras, even if \sU\ is only a pre-universe, but I
don't think this can be correct, since it would imply the relative
consistency of collection assuming separation.

The one remaining difference between Hayashi's approach and ours is
that rather than the ``maximal'' pre-universe consisting of all
well-founded extensional graphs, he is more interested in the
``minimal'' ones obtained by transfinitely iterating the power set
operation, mimicking the construction of the von Neumann hierarchy in
material set theory.  This is also the approach taken by Fourman.


In order to make sense of this transfinite iteration, we of course
need an external set theory containing ordinals along which we can
perform induction.  For convenience and familiarity, we take this to
be a material set theory.  Thus, for the rest of this section, let
\bV\ be a model of IZF, with $\Set = \bbSet(\bV)$ its category of
sets.  (Fourman and Hayashi assumed the metatheory to be ZFC, but this
is unnecessary.)

Let \bS\ be a complete and well-powered topos over \bV.  We will
describe the cumulative hierarchy in \bS, as constructed by Fourman
and Hayashi.  Recall that a (von Neumann) \textbf{ordinal} in \bV\ is
a transitive set on which $\in$ is a transitive relation (it is
automatically well-founded, by the axiom of foundation).  Since
set-induction is valid in IZF, we can construct objects of \bS\ by
transfinite recursion on ordinals.  Thus we can inductively define a
well-founded extensional graph $V_\al$ for each ordinal $\al$ of $\bV$
such that $V_\al$ is the extensional quotient of $\coprod_{\be<\al} P
V_\be$.  (The well-founded relation on $P V_\be$ is defined as above,
by setting $H\prec K$ if there exists an $x\in K$ such that $H =
\setof{y | y\prec x}$.)

\begin{rmk}
  Fourman and Hayashi defined the hierarchy $V_\al$ in a more
  traditional way, by conditioning on whether \al\ is zero, a
  successor, or a limit:
  \[V_0 = 0 \qquad V_{\al+1} = P V_\al \qquad V_{\al} = \colim_{\be<\al}
  V_\be \quad\text{(\al\ a limit)}
  \]
  Intuitionistically, this classification of ordinals is no longer
  valid.  We could still try to define $V_\al = \colim_{\be<\al} P
  V_\be$, but since intuitionistic ordinals are not necessarily
  linearly ordered, in general this colimit need not be extensional.
  However, in the classical case, all three definitions are
  equivalent.
\end{rmk}

Evidently, the class $\{V_\al\}$ is a universe in \bS, in Hayashi's
sense.  Thus, it induces a class of well-founded extensional \apgs\ in
all slices of \bS, namely those of the form \xbar\ for some $x\colon
U\to V_\al$.  It is natural to ask whether this includes \emph{all}
well-founded extensional \apgs, and thus whether Hayashi's
interpretation valued in $\sU=\{V_\al\}$ is the same as the
stack-semantics interpretation.  The answer is yes.

\begin{lem}\label{thm:fourman-apg2}
  In any complete and well-powered topos \bS, every (internally)
  well-founded extensional \apg\ in $\bS/U$ is of the form \xbar\ for
  some ordinal $\al$ and some $x\colon U\to V_\al$.
\end{lem}
\begin{proof}
  Since every slice topos of a complete and well-powered topos is also
  such, and pullback functors $U^*$ preserve the construction of the
  $V_\al$ hierarchy, we may assume that $U=1$.  Thus, let $X$ be a
  well-founded extensional \apg\ in \bS; the proof is a modified
  version of the standard recursion theorem for well-founded
  relations.

  For a fixed ordinal \al, define an \emph{\al-attempt} to be a
  partial function $X\rightharpoonup V_\al$ which is a simulation and
  whose domain is an initial segment of $X$.  Using extensionality of
  $V_\al$, we can prove by induction on $X$ in the internal logic of
  \bS\ that any two \al-attempts agree on their common domain, and
  that the union of any family of \al-attempts is an \al-attempt.
  Therefore, there is a unique largest \al-attempt $f_\al\colon
  X\rightharpoonup V_\al$.  Furthermore, for any $\be\ge\al$ there is
  a unique simulation $i_\al^\be\colon V_\al \into V_\be$, and we must
  have $i_\al^\be f_\al = f_\be$.

  We now show that there exists an \al\ such that $f_\al$ is defined
  on all of $X$.  Consider the set $D$ of subobjects $S\mono X$ such
  that $S$ is the domain of $f_\al$ for some \al; this exists since
  \bS\ is well-powered and \bV\ satisfies separation.  By collection
  in \bV, there is a set $A$ of ordinals such that for any $S\in D$,
  $S = \mathrm{dom}(f_\al)$ for some $\al\in A$.

  Let $\be = \bigcup_{\al \in A} \al$ and $\gm = \be^+$; we claim that
  $f_\gm$ is a total function.  Let $S_\gm$ be its domain.  First,
  note that by the construction of $A$, there exists an $\al\in A$
  such that $S_\gm$ is the domain of $f_\al$.  Since the domains of
  $f_\al$ are increasing in \al, it follows that $S_\gm$ is the domain
  of $f_\be$ as well.

  We will prove that $S_\gm$ is an inductive subset of $X$, and hence
  must be all of $X$.  Applying the Kripke-Joyal semantics, suppose
  given some $U$ and some $q\colon U\to X$ such that for any $p\colon
  U'\to U$ and any $q'\colon U'\to X$ such that $q'\prec q p$, $q'$
  factors through $S_\gm$; we must show that $q$ factors through
  $S_\gm$.

  Let $T \subseteq U^*X$ be the subset defined in the internal logic
  by $\setof{ x\in X | x\prec q}$.  Then by assumption, $T$ factors
  through $U^*S_\gm$ and therefore the domain of $U^* f_\be$ contains
  $T$.  Then $T$ is classified by a map $t\colon U\to P X$.  Now the
  partial function $f_\be$ determines a map $P f_\be \colon P X \to P
  V_\be$ by taking images, so by composition we have a map $P f_\be
  \circ t\colon U\to P V_\be$.  But $P V_\be \cong V_\gm$, so we have
  a map $r\colon U\to V_\gm$.

  Let $\Ubar \subseteq X$ be the image of $q\colon U\to X$.  We can
  prove in the internal logic that if $q(u) = q(u')$ for some $u,u'\in
  U$, then $r(u)=r(u')$; hence $r$ descends to $\rbar\colon \Ubar\to
  V_\gm$.  We can also show, again internally, that if $x\in S_\gm$
  and $x = q(u)$, then $\rbar(x) = f_\gm(x)$, since both are equal to
  $\setof{ f_\be(y) | y \prec x}$.  Therefore, if we let $\Sbar =
  S_\gm \cup \Ubar$, then $f_\gm$ and $\rbar$ induce a map $g\colon
  \Sbar \to V_\gm$.  Finally, we can show that $g$ is a simulation,
  hence its domain is a subset of $S_\gm$; thus $\Ubar\subseteq
  S_\gm$, i.e.\ $q$ factors through $S_\gm$.  Hence $S_\gm$ is
  inductive, and thus all of $X$, so $f_\gm$ is a total function.
\end{proof}

Therefore, we have:

\begin{thm}\label{thm:hayashi-interp2}
  Let \bS\ be a complete and well-powered topos, and let \sU\ be the
  above von Neumann universe in \bS.  Then the interpretation of
  material set theory in the stack semantics of \bS\ is identical to
  Hayashi's interpretation valued in \sU.\qed
\end{thm}

Finally, we relate this to Fourman's
approach~\cite{fourman:sheaf-setthy}.  Instead of a forcing semantics,
he works with representing subobjects.  Thus, we still consider the
parameters to be morphisms $x\colon U\to A$ for some well-founded
extensional graph $A$ in a pre-universe \sU, but now instead of a
forcing relation we define, for each formula \ph\ at stage $U$, a
subobject $\mm{\ph}\mono U$, as follows.
\begin{blist}
\item If $x\colon U\to A$ and $y\colon U\to B$, then $\mm{x=y}$ is the
  equalizer of $f x$ and $g y$ where $f\colon A\to C$ and $g\colon
  B\to C$ are a pair of simulations.
\item Similarly, $\mm{x\in y}$ is the pullback of $C_{\prec}$ along
  $(fx ,gy)$.
\item The connectives \bot, \top, \meet, \join, \imp, and \neg\ are
  interpreted using the Heyting algebra connectives in $\Sub(U)$, as
  usual.
\item Given a formula $\ph(x)$ at stage $U$ with one free variable
  $x$, for each well-founded extensional graph $A$ consider the image
  of
  \[ \mm{\ph(\pi_2)} \mono U\times A \too[\pi_1] U. \]
  This is a subobject of $U$, call it $\mm{\exists x\in A.\ph(x)}$.
  Since \bS\ is well-powered, we can apply separation in \bV\ to
  obtain the set of all subobjects of $U$ which are of the form
  $\mm{\exists x\in A.\ph(x)}$ for some $A\in\sU$.  Since \bS\ is
  cocomplete, we can then take the union of this set; call it
  $\mm{\exists x.\ph(x)}$.
\item Universal quantification is similar, using the dual image
  $\coim_{\pi_1} \mm{\ph(\pi_2)}$ instead of the image $\im_{\pi_1}
  \mm{\ph(\pi_2)}$, and an intersection instead of a union.
\end{blist}
These subobjects are essentially the same as those defined by Fourman,
although he also considered only the von Neumann universe.  Thus it
remains only to check that they do, in fact, represent the forcing
relation defined by Hayashi.

\begin{lem}\label{thm:fourman=hayashi}
  Let \bS\ be a complete and well-powered topos and \sU\ a
  pre-universe.  Then for any formula $\ph$ at stage $U$ and any map
  $W\too[p] U$, we have $W\ff \ph$ in Hayashi's sense if and only if
  $p$ factors through $\mm{\ph}$ in Fourman's sense.
\end{lem}
\begin{proof}
  Since \bS\ is autological by \autoref{thm:cplt-aut}, the forcing
  relation of Hayashi is representable.  Thus, it remains only to
  check that the subobjects defined by Fourman are the same as those
  produced by autology.  This is evident for atomic formulas and
  connectives, so we consider only the quantifiers.

  In Fourman's construction of $\mm{\exists x.\ph(x)}$, it is clear
  that $\mm{\exists x\in A.\ph(x)} \ff \exists x.\ph(x)$ for each $A$,
  so by \autoref{thm:forcing-5} we also have $\mm{\exists x.\ph(x)}
  \ff \exists x.\ph(x)$.  On the other hand, given $W\to U$ such that
  $W\ff \exists x.\ph(x)$, we have an epimorphism $Z\epi W$ and a map
  $x\colon Z\to A$, for some $A\in\sU$, such that $Z\ff \ph(x)$.
  Thus, $Z\to U\times A$ factors through $\mm{\ph(\pi_2)}$, so $W\to
  U$ factors through $\mm{\exists x\in A.\ph(x)}$, and hence also
  through $\mm{\exists x.\ph(x)}$.  Thus, Fourman's subobject
  $\mm{\exists x.\ph(x)}$ is in fact a representing object for
  $\exists x.\ph(x)$.

  Universal quantification is even easier.  If $W\ff \forall
  x.\ph(x)$, then for any $A\in\sU$, $p\colon W\to U$ must factor
  through $\mm{\forall x\in A.\ph(x)}$, whence it factors through
  their intersection $\mm{\forall x.\ph(x)}$.  Conversely, it suffices
  to show that $\mm{\forall x.\ph(x)} \ff \forall x.\ph(x)$, but this
  follows since by definition, we have $\mm{\forall x.\ph(x)} \ff
  \forall x\in A.\ph(x)$ for any $A\in\sU$.
\end{proof}

\begin{thm}\label{thm:fourman-interp}
  Let \bS\ be a complete and well-powered topos, and let \sU\ be the
  above von Neumann universe.  Then the interpretation of material set
  theory in the stack semantics of \bS\ is identical to Fourman's
  interpretation valued in \sU.\qed
\end{thm}

As remarked previously, it follows from known facts about the
Fourman-Hayashi interpretation that the stack-semantics interpretation
in Grothendieck toposes, such as sheaves on a locale or continuous
group actions, can also be identified with the logic of standard
material-set-theoretic constructions, such as Boolean- and
Heyting-valued, permutation, and symmetric models.
See~\cite{fourman:sheaf-setthy,bs:freyd-ac,bs:cplt-rep-setth}.

\section{Categories of classes}
\label{sec:cofc}

We now turn to a comparison between the stack semantics and the
methods of algebraic set theory, a field which was begun
by~\cite{jm:ast} and has been carried forward by many others;
see~\cite{ast-website} for a number of references.  The idea of
algebraic set theory is to consider, rather than a category whose
objects represent sets, a larger category whose objects represent
\emph{classes}, equipped with a ``notion of smallness'' specifying
which objects (and more generally which families) should be regarded
as sets.  In particular, we require that there is a ``class of all
sets.''  We can then interpret ``unbounded'' quantifiers over all sets
as bounded quantifiers in the internal logic of the category of
classes.

There is not yet a universally accepted set of axioms for a category
of classes, so we will at first restrict ourselves to a small core set
of axioms.

\begin{defn}
  A \textbf{(representable) notion of smallness} on a positive Heyting
  category \bC\ consists of a distinguished map $E\too[\pi] S$.  Given
  \pi, a morphism $X\to U$ in \bC\ is called \textbf{small} if there
  exist two pullback squares
  \[\xymatrix{X \ar[d] &
    X' \ar[l] \ar[r]\ar[d]
    \save*!/dl-1.4pc/dl:(-1,1)@^{|-}\restore
    \save*!/dr-1.4pc/dr:(-1,1)@^{|-}\restore
    &
    E \ar[d]^\pi\\
    U & U'\ar@{->>}[l]^p\ar[r]_{\th_X} & S.}\]
  in which $U'\epi U$ is a regular epimorphism.  An object $X$ is
  called \textbf{small} if $X\to 1$ is a small map.
\end{defn}


Observe that a morphism $X\to U$ is small if and only if $U\ff$
\qq{there exists an $x\in S$ such that $X\cong x^*E$} in the stack
semantics of \bC.  We thus abbreviate this statement as \qq{$X$ is
  small}.  Note also that a map $f\colon X\to Y$ in $\bC/U$ is small
if and only if $U\ff$ \qq{for every $y\in Y$, the fiber $y^*X$ is
  small}, which is precisely the intended intuition.  Thus we
abbreviate this statement as \qq{$f$ is small}.  The possibility of
this sensible ``logical'' interpretation depends, of course, on
allowing passage to a cover in the definition of small maps, since we
cannot expect classifying maps of small objects to be uniquely defined.

In the literature, it is more common to consider \bC\ as equipped with
a class of maps called ``small'' and then assert the existence of a
representing map $\pi$ as an axiom.  As observed in~\cite{mp:ttcst},
however, defining the small maps from $\pi$ as above has the advantage
that it immediately implies a few of the other common axioms for small
maps, including:
\begin{blist}
\item Small maps are closed under pullback.
\item Small maps descend along regular epis, i.e.\ if the pullback of
  $f\colon X\to U$ along a regular epi $V\epi U$ is small, so is $f$.
\item If $X\to U$ and $Y\to V$ are small maps, then so is $X+Y\to U+V$.
\end{blist}

Let \bS\ be the full subcategory of \bC\ determined by the small
objects; we intend to define a translation from the stack semantics of
\bS\ to the internal logic of \bC.  However, for \bS\ to have a
well-behaved stack semantics, it must be itself a positive Heyting
category.  The following definition is closely related to those
commonly assumed in the literature, but weaker (and thus more general)
than most.

\begin{defn}\label{defn:cofc}
  A \textbf{category of classes} is a positive Heyting category
  equipped with a notion of smallness which additionally satisfies the
  following axioms.
  \killspacingtrue
  \begin{enumerate}
  \item $1\ff_\bC$ \qq{if $X$ is small, then a map $f\colon Y\to X$ is 
      small if and only if $Y$ is small}.\label{item:cofc1}
  \item $1\ff_\bC$ \qq{the full subcategory of small objects is closed
      under finite limits, finite coproducts, images, and dual images
      (and is therefore itself a positive Heyting category)}.\label{item:cofc2}
  \item Every small map is exponentiable.\label{item:cofc3}
  \end{enumerate}
  \killspacingfalse
\end{defn}

We have stated the first two conditions in terms of the stack
semantics, since we believe it makes their real meaning clearer.  Of
course, if we work out what they mean explicitly in terms of \bC, we
see that they reduce to some well-known axioms for categories of
classes.  Specifically, the first is equivalent to:
\begin{blist}
\item if $g$ is small, then $g f$ is small if and only if $f$ is small.
\end{blist}
Assuming this (and the representable definition of small maps as
above), the second is equivalent to:
\begin{blist}
\item For any $U$, the full subcategory of $\bC/U$ determined by the
  small maps is a positive Heyting category, and the inclusion functor
  preserves all the structure.
\end{blist}
More explicitly, this is equivalent to the following axioms.
\begin{blist}
\item All isomorphisms are small.
\item $0$ and $1+1$ are small objects.
\item If $g\colon B\to C$ and $g f\colon A\to C$ are small, then so is
  the composite $\mathrm{im}(f) \mono B \to C$.  (In particular, this
  follows if small maps are closed under quotients.)
\item If $m\colon A\mono C$ is monic and small, and $f\colon C\to D$
  is small, then $\coim_f(A)\mono D$ is small.
\end{blist}
Finally, if \bC\ is exact, then the third condition (exponentiability
of small maps) is equivalent to
\begin{blist}
\item $1\ff$ \qq{every small object is exponentiable}
\end{blist}
but if \bC\ is not exact, then \autoref{thm:stack-univ} does not apply,
and so this latter property does not imply the actual existence of
exponentials.



\begin{rmk}
  \autoref{defn:cofc}\ref{item:cofc1} is sometimes called the ``axiom
  of replacement,'' but it actually has no connection to our axiom of
  replacement from \S\ref{sec:strong-ax}.  We prefer to think of it as
  saying that the notions of ``small family'' in \bS\ and \bC\ agree.
\end{rmk}

Now the exponentiability of the representing map $\pi\colon E\to S$
implies that it determines an \emph{internal full subcategory} of \bC,
i.e.\ an internal category whose object of objects is $S_0 = S$ and
whose object of morphisms is the exponential $S_1 = (p_1^*E)^{p_2^*E}$
in $S_0\times S_0$.  (See~\cite[B2.3.5]{ptj:elephant} and~\cite[\S
I.5]{jm:ast}.)  Regarding this as ``the category of sets,'' we can
then reason about it in the ordinary internal logic of \bC.  That is,
for any $U\in\bS$, if \ph\ is a formula in the language of categories
with object-parameters given by maps $U\to S_0$ and arrow-parameters
given by maps $U\to S_1$, then \ph\ has a meaning in the internal
first-order logic of \bC\ and is classified by a subobject
$\mm{\ph}\mono U$.  Of course, in general $\mm{\ph}$ may not be small.

Now suppose that \ph\ is a formula over $U$ in \bS\ in the sense of
\S\ref{sec:internal-logic}.  Since each object-parameter of \ph\ is a
small map $X\to U$, there is a regular epi $U'\epi U$ over which $X$
is a pullback of $\pi$, and similarly for arrow-parameters.  Since
\ph\ has only finitely many parameters, we can find a single cover
$U'\epi U$ over which \emph{all} parameters of \ph\ are represented by
maps into $S_0$ or $S_1$, and thereby translate \ph\ into a formula
\phhat\ at stage $U'$ of the internal language of \bC\ about the
internal category $S_1\toto S_0$.

However, without further axioms, there is little we can say about the
relationship of \ph\ to \phhat.  Basically, the problem is that \bS\
only knows about the small objects over $1$, but \phhat\ knows about
small maps over all objects of \bC.  So we need to assume that \bC\ is
``generated by small objects'' in a suitable sense.

\begin{defn}
  A category of classes \bC\ is \textbf{well-generated by small
    objects} if the following hold.
  \killspacingtrue
  \begin{enumerate}
  \item \emph{The small objects are a strong generator:} if
    $A\xmono{f} B$ is a monomorphism in \bC\ and every map $U\to B$
    factors through $A$ when $U$ is small, then $f$ is an
    isomorphism.\label{item:sp1}
  \item \emph{The small objects are relatively projective:} if $U$ is
    small and $A\xepi{p} U$ is regular epi, there is a small $V$
    and a regular epi $V\xepi{q} U$ which factors through
    $p$.\label{item:sp2}
  \item \emph{The small objects are relatively indecomposable:} if $U$
    is small and $U=A\cup B$, there are small objects $V$ and $W$
    such that $U=V\cup W$, $V\subseteq A$, and $W\subseteq
    B$.\label{item:sp3}
  \item \emph{The small objects are collectively nonempty:} if $X$ is
    small and there is a map $X\to 0$, then $X\cong 0$.\label{item:sp4}
  \end{enumerate}
  \killspacingfalse
\end{defn}

\begin{rmk}
  Two of these conditions are well-known:~\ref{item:sp1} is called the
  axiom of \emph{small generators} and~\ref{item:sp2} is called the
  axiom of \emph{small covers} (see~\cite{abss:rfosttcoc}).  Note
  that~\ref{item:sp1} is always true if interpreted in the stack
  semantics, since all identities are small.  By contrast,
  when~\ref{item:sp2} is interpreted in the stack semantics of \bC, it
  becomes precisely the usual ``collection axiom'' of algebraic set
  theory (see~\cite{jm:ast}).

  Condition~\ref{item:sp3} is rarely stated explicitly, since it
  follows from~\ref{item:sp2} if all complemented monics are small,
  which in turn follows from the common assumption that all diagonals
  are small.  And of course,~\ref{item:sp4} is always true.  However,
  we have chosen to state all four axioms explicitly in order to
  complete the analogy with (constructive) well-pointedness.
\end{rmk}

\begin{rmk}\label{thm:nonelem-smpin}
  If \bS\ is small in some external set theory, then \bC\ has a
  restricted Yoneda embedding $\bC \into \bSh(\bS)$ into the category
  of sheaves for the coherent topology on \bS.  In this case, \bC\ is
  well-generated by small objects if and only if this functor is
  coherent and conservative.  This should be compared with
  \autoref{rmk:nonelem-wpt}, and likewise the proof of
  \autoref{thm:cofc-logic} should be compared with that of
  \autoref{thm:collpin}.
\end{rmk}

\autoref{thm:nonelem-smpin} suggests that $\bSh(\bS)$ itself should be
a ``canonical'' choice of a category of classes containing \bS, and in
fact this is the case.

\begin{prop}
  For any small Heyting pretopos \bS, the category $\bSh(\bS)$ of
  sheaves for the coherent topology on \bS\ is a category of classes
  in which \bS\ is a well-generating category of small objects.
\end{prop}
\begin{proof}[Sketch of proof]
  The self-indexing of \bS\ is a stack for its coherent topology, so
  if we rectify it to a strict functor $\bS\op\to\Cat$, it will become
  an internal category in $\bSh(\bS)$ by the results of~\cite[\S
  V.1]{awodey:thesis}.  With $S$ the object of objects of this
  internal category and a suitable definition of $E$, this defines a
  representable notion of smallness.  We then show that a map $Y\to
  X$ is small iff for any $U\in\bS$ and any map $y(U)\to X$, where
  $y\colon \bS\into \bSh(\bS)$ is the Yoneda embedding, the pullback
  $y(U)\times_X Y$ is representable.  In particular, \bS\ is the
  category of small objects.  It is then straightforward to verify the
  axioms.
\end{proof}

With this example in mind, the following theorem realizes our claim in
the introduction that when \bS\ is small, the stack semantics of \bS\
can equivalently be defined as a fragment of the internal logic of
$\bSh(\bS)$, which is ``canonical'' modulo the chosen strictification
of the self-indexing.  Recall the definition of a formula \phhat\ in
the internal logic of \bC\ from any formula \ph\ in \bS\ in the
language of a category.

\begin{thm}\label{thm:cofc-logic}
  Let \bC\ be a category of classes which is well-generated by small
  objects, and \bS\ the full subcategory of small objects.  Then for
  any $U\in\bS$, any formula \ph\ over $U$ in \bS, and any $p\colon
  V\to U$ in \bS, we have
  \[ V \ff_{\bS} \ph \quad\iff\quad V'\ff_{\bC}\phhat. \]
  Here $V' = V\times_U U'$, where $U'\epi U$ is a cover chosen as in
  the definition of \phhat.
\end{thm}

Note that on the left, $\ff$ refers to the stack semantics notion of
forcing, while on the right it refers to the Kripke-Joyal semantics
for the usual internal logic of \bC\ (which is, of course, the same as
the \ddo-fragment of the stack semantics of \bC, so there is no real
ambiguity in the notation).

\begin{proof}
  The proof is by induction on \ph.  The cases of atomic formulas,
  \top, \bot, and \meet\ are evident, so it remains to deal with
  \join, \imp, \im, and \coim.  We adopt the convention that adding a
  prime symbol denotes passing to a regular-epi cover over which all
  parameters are represented by morphisms to $S_0$ or $S_1$.

  \textit{Disjunction:} if $V\ffs(\ph\join\psi)$, then we have
  $V=W_1\cup W_2$ in \bS\ with $W_1\ffs\ph$ and $W_2\ffs\psi$.  By the
  inductive hypothesis, $W_1'\ffc\phhat$ and $W_2'\ffc\psihat$; thus
  $V'\ffc (\phhat\join\psihat) = \widehat{(\ph\join\psi)}$.

  Conversely, if $V'\ffc(\phhat\join\psihat)$, then $V'=A_1'\cup A_2'$
  with $A_1'\ffc \phhat$ and $A_2'\ffc\psihat$.  Let $A_i$ be the
  image of $A_i'$ in $V$, for $i=1,2$, so that $V=A_1\cup A_2$.  Since
  small objects are relatively indecomposable, we have $V= W_1\cup
  W_2$ in \bS\ with $W_1\subseteq A_1$ and $W_2\subseteq A_2$.  Therefore,
  $W_1'\ffc \phhat$ and $W_2'\ffc\psihat$, and so by the inductive
  hypothesis, $W_1\ffs\ph$ and $W_2\ffs\psi$; hence
  $V\ffs(\ph\join\psi)$.

  \emph{Implication:} suppose first $V\ffs(\ph\imp\psi)$ and let
  $A'\to V'$ be a map in \bC\ such that $A'\ffc \phhat$.  Form the
  classifying subobject $\mm{\psihat}\mono A'$ in \bC.  Then for any
  $W\to A'$ where $W\in\bS$, we have $W\ffs(\ph\imp\psi)$ and $W'\ffc
  \phhat$, hence by the inductive hypothesis $W\ffs\ph$; thus
  $W\ffs\psi$ and so (by the inductive hypothesis again)
  $W'\ffc\psihat$.  Therefore, $W'$ factors through $\mm{\psihat}$,
  and so $W$ factors through $\mm{\psihat}$.  Since the small objects
  are a strong generator, this implies $\mm{\psihat}\iso A'$, and thus
  $A'\ffc \psihat$; hence $V'\ffc (\phhat\imp\psihat)$.

  Conversely, suppose $V'\ffc (\phhat\imp\psihat)$ and let $W\to V$ be
  such that $W\ffs\ph$.  Then, by the inductive hypothesis, $W'\ffc
  \phhat$, and so $W'\ffc \psihat$, hence $W\ffs\psi$.  Thus,
  $V\ffs(\ph\imp\psi)$.

  \emph{Existential quantification:} suppose first that $V\ffs \exists
  X.\ph(X)$.  Then we have $W\in \bS$, a regular epi $W\epi V$, and an
  object $Y\in\bS/W$ such that $W\ffs\ph(Y)$.  Passing to a cover
  $W'\epi W$ which represents $Y$ (by a map $y\colon W'\to S_0$, say)
  as well as all other parameters of \ph, the inductive hypothesis
  implies that $W'\ffc \phhat(y)$.  Since $W'\to V'$ is regular epi,
  this implies $V'\ffc \exists y.  \phhat(y)$.

  Conversely, suppose $V'\ffc \exists \th. \phhat(\th)$.  Then we have
  a cover $A'\epi V'$ and a map $y\maps A'\to S_0$ with
  $A'\ffc\phhat(y)$.  Since small objects are relatively projective,
  there is a regular epi $W\epi V$ in \bS\ which factors through $A'$.
  Hence $W \ffc \phhat(y)$, so by the inductive hypothesis, $W\ffs
  \ph(Y)$ where $Y=y^*E$.  Hence, $V\ffs \exists X. \ph(X)$.
  Existential quantification over arrows is analogous.

  \emph{Universal quantification:} suppose first that $V'\ffc \forall
  x.\phhat(x)$.  Then for any $W\to V$ in \bS\ and $Y\in\bS/W$, we can
  pass to a further cover $W'' \epi W' = W\times_V V'$ over which $Y$
  is represented by some $y\colon W'\to S_0$.  Since $W''\ffc \forall
  y.\phhat(y)$, we have $W''\ffc \phhat(y)$, and hence by the
  inductive hypothesis $W\ffs \ph(Y)$.  Thus, $V\ffs \forall
  X. \ph(X)$.

  Conversely, suppose $V\ffs \forall X. \ph(X)$, let $A'\to V'$ be a
  map, take any $y\maps A'\to S_0$, and form the classifying subobject
  $\mm{\phhat(y)}\mono A'$.  Then for any map $W\too[q] A'$ with $W$
  in \bS, we have $W\ffs \forall X.\ph(X)$, and thus, by the inductive
  hypothesis, $W\ffc \phhat(y)$.  Hence $W$ factors through
  $\mm{\phhat(y)}$, so since the small objects are a strong generator,
  we have $\mm{\phhat(y)}\iso A'$, and so $A'\ffc \phhat(y)$.
  Therefore, $V'\ffc \forall x.\phhat(x)$.  Universal quantification
  over arrows is analogous.
\end{proof}

In particular, we have another way of producing autological
(pre)topoi.

\begin{defn}\label{defn:cofc-sep}
  A category of classes satisfies the \textbf{separation axiom} if every
  monomorphism is small.
\end{defn}

\begin{cor}\label{thm:sep-aut}
  If \bC\ is a category of classes which is well-generated by small
  objects and which satisfies the separation axiom, then its category
  \bS\ of small objects is autological.
\end{cor}
\begin{proof}
  By \autoref{thm:cofc-logic}, any sentence \ph\ over $U\in\bS$
  induces a sentence \phhat\ in the internal logic of \bC\ at stage
  $U'$ for some regular epi $U'\epi U$.  Since small objects are
  relatively projective, we may as well assume that $U'$ is small.
  By the separation axiom, it follows that the usual classifying
  object $\mm{\phhat}$ in the internal logic of \bC\ is small, and so
  by closure under images, the image of the composite $\mm{\phhat}\to
  U' \to U$ is a small subobject of $U$.  The correspondence of
  \autoref{thm:cofc-logic} then directly implies that this image
  classifies \ph\ in the stack semantics of \bS, so \bS\ is autological.
\end{proof}





There remains the question of how a given category \bS\ can be
embedded as the category of small objects in some category of classes
\bC.  We have seen that there is a ``canonical'' answer, namely
$\bSh(\bS)$, but this category has some problems.  In particular, it
fails to satisfy some desirable axioms of algebraic set theory, such
as that any quotient of a small object is small, or that all diagonals
are small.  We didn't need these axioms to prove
\autoref{thm:cofc-logic}, but we will need some of them below to
construct and compare models of material set theory.  This problem can
be solved by restricting to the subcategory of sheaves whose diagonals
\emph{are} small; these are the \emph{ideals}
of~\cite{af:ast-classes,afw:ast-ideals}.


Another problem, however, which afflicts both sheaves and ideals, is
that not all axioms satisfied by \bS\ extend to corresponding
properties of these categories of classes.  Chief among the
problematic axioms is separation; i.e.\ we have no converse to
\autoref{thm:sep-aut}.  The problem is that not all sheaves or ideals
are ``first-order definable'' from \bS, so the first-order axiom
schema of autology for \bS\ cannot be extended to the separation axiom
for sheaves or ideals.

One solution to this problem is proposed
in~\cite{abss:foss-et,abss:rfosttcoc}: if \bS\ can be equipped with an
extra non-elementary structure called a \emph{superdirected structural
  system of inclusions (sdssi)}, then it can be embedded in a category
of \emph{super\-ideals} which does satisfy separation.  (The word
``structural'' in the term sdssi is backwards from our usage---an
sdssi is a very \emph{non-structural} notion, in that it distinguishes
between isomorphic objects.)  This includes both complete toposes and
realizability toposes, providing a unified reason why such toposes can
model IZF.

We prefer to view \emph{autology} as the property characterizing those
toposes which can model IZF, since unlike an sdssi, it is an
elementary first-order property and can be phrased in a purely
structural (i.e.\ category-theoretic) way.  (By \autoref{thm:sep-aut},
any topos admitting an sdssi must be autological.)  However, this
would be most satisfying if any autological topos could be embedded in
a category of classes satisfying the axiom of separation.  One wants
to consider the category of ``first-order definable classes'' over
\bS, but it is nearly impossible to give this meaning in a purely
structural way.  However, we will show in~\cite{shulman:2cofc} that
there is a canonical answer if we generalize to a \emph{2-category} of
classes---namely, we can consider the 2-category of \emph{stacks} for
the coherent coverage on on \bS\ which are ``definable'' from the
stack semantics in a precise sense.  This modification is quite
reasonable from the point of view of structural set theory, since
there a ``proper class'' is something quite different from a set: at
best it is a groupoid, having only a notion of isomorphism between its
objects, rather than a notion of equality.

\bigskip

We now turn to a comparison between models of material set theory in
the stack semantics and in a category of classes.  Since
\autoref{thm:cofc-logic} shows that the underlying logics agree, it
remains to compare the ``material sets'' in the two models.  This is
quite similar to the situation in \S\ref{sec:fourman}: the main
difference is that in a category of classes, the ``material sets'' are
represented by morphisms into a single object $\sV$ (which is not, of
course, small), rather than by morphisms into a family of objects
making up a ``pre-universe.''

(In fact, there is a formal relationship: a pre-universe in a topos
\bS, in Hayashi's sense, naturally forms a single object of the
category of ideals, in which we can apply the methods of this section.
This is the jumping-off point of~\cite{abss:foss-et,abss:rfosttcoc},
which not only constructs a category of ideals, but a corresponding
forcing semantics that generalizes Hayashi's.)

As remarked above, in order to define the object $\sV$ and perform
this comparison, we need a little more structure on our category of
classes.  By an \emph{$X$-indexed family of small subobjects of $A$}
we mean a monic $R\mono A\times X$ such that the projection $R\to X$
is a small map.  A \textbf{powerclass object} of $A$ is an object $P_s
A$ equipped with a universal $P_s A$-indexed family of small
subobjects of $A$.  Note that in general, $P_s A$ need not be small
even if $A$ is.  It is shown in~\cite{jm:ast} that powerclass objects
can be constructed when \bC\ is exact.  Authors who do not assume that
\bC\ is exact often assume directly that powerclass objects exist.

The following \emph{ad hoc} definition summarizes what we need for our
comparison.

\begin{defn}
  A category of classes \bC\ has a \textbf{good material model} if it
  satisfies the following.
  \killspacingtrue
  \begin{enumerate}
  \item \bC\ has a small \nno.\label{item:gcc0}
  \item If $f$ and $g$ are small, then so is $\Pi_f g$.\label{item:gcc1}
  \item \bC\ has powerclass objects.
  \item If $p$ is regular epi and $f p$ is small, then $f$ is small
    (equivalently, $1\ff_\bC$ \qq{if $Y\to X$ is regular epi and $Y$
      is small, then $X$ is small}).\label{item:gcc4}
  \item The endofunctor $P_s$ has an indexed initial algebra \sV,
    i.e.\ a $P_s$-algebra \sV\ such that $U^*\sV$ is an initial
    algebra for all $U\in\bC$.
  \end{enumerate}
  \killspacingfalse
\end{defn}

Assumptions~\ref{item:gcc0} and~\ref{item:gcc1} ensure that \bS\ is a
\Pi-pretopos with a \nno, so that the theory of \S\ref{sec:constr-mat}
can be applied in its stack semantics.  The other three axioms are
about describing the universe-object \sV.  Note that~\ref{item:gcc4}
is necessary in order to make $P_s$ into a covariant functor in the
usual way.  (In fact, $P_s$ can be made into a monad,
using~\ref{item:gcc4} again to obtain the multiplication---but we
emphasize that the relevant notion of ``algebra'' refers only to an
algebra for $P_s$ as an endofunctor, not as a monad.)  The
indexed initial $P_s$-algebra \sV\ is constructed in~\cite{jm:ast}
under the assumption of exactness and a subobject classifier, and
in~\cite{vdbm:pred-i-exact} under the assumption of W-types and
``bounded exactness;'' we will simply assume it to exist.

As for any initial algebra, the structure map $P_s \sV \xto{v} \sV$ is
an isomorphism.  Thus, by the universal property of $P_s$, the inverse
of $v$ classifies a binary relation $\prec$ on \sV\ which is
``set-based,'' i.e.\ the second projection $\sV_{\prec} \to \sV$ is
small.  It is not hard to show that $\prec$ is well-founded and
extensional.  An interpretation of material set theory in \bC\ is then
obtained by considering \sV\ with $\prec$ as a model for material sets
with $\in$ in the internal logic of \bC.  It satisfies many of the
axioms automatically, while others are inherited from the
corresponding categorical axioms for \bC;
see~\cite{jm:ast,mp:ttcst,vdbm:pred-i-exact} for more details.

\begin{rmk}
  There is a slightly different thread in algebraic set theory,
  starting with~\cite{simpson:elem-cofc} and continuing
  with~\cite{af:ast-classes,afw:ast-ideals}, in which material set
  theory is not modeled in the initial $P_s$-algebra but in an
  arbitrary object \sU\ equipped with a monomorphism $P_s\sU\into
  \sU$.  The elements of \sU\ not in the image of $P_s \sU$ are then
  regarded as ``atoms,'' and the axiom of foundation is not
  necessarily satisfied.  We will focus on the situation described
  above, which seems better suited to a comparison with the stack
  semantics.
\end{rmk}

Now, from any morphism $x\colon U\to \sV$ we can construct a
well-founded extensional \apg\ \xbar\ over $U$ as follows.  First note
that all our axioms are stable under pullback, so we may as well
assume that $U=1$.  Now let $\check{x}\colon N\to P_s\sV$ be the
unique map such that
\[\vcenter{\xymatrix{1 \ar[r]^o \ar[dr]_{\eta x} & N \ar[r]^s
    \ar[d]^{\check{x}} & N \ar[d]^{\check{x}}\\
  & P_s \sV \ar[r]_{c} & P_s \sV}}
\]
commutes.  Here $\eta\colon \sV\to P_s\sV$ is the unit of the monad
$P_s$, while $c$ is the composite:
\[ P_s \sV \xto{P_s (v\inv)} P_s P_s \sV \xto{\mu} P_s \sV
\]
(where \mu\ is the multiplication of the monad $P_s$).  Now
$\check{x}$ classifies an $N$-indexed family of small subobjects of
$\sV$, i.e.\ a map $\xbar \to N\times \sV$ such that $\xbar\to N$ is
small.  Since $N$ is also small, $\xbar$ is a small object.

Working in the internal logic of \bC, well-foundedness of \sV\ now
implies that the projection $\xbar\to\sV$ is monic.  Since it is also
a simulation, it follows that \xbar\ is well-founded and extensional.
It is rooted by $x$ and accessible by definition, so it is a
well-founded extensional \apg, as desired.  And just as in
\S\ref{sec:fourman}, we can verify that
\begin{enumerate}
\item $U \ff (\xbar\cong\ybar)$ if and only if $x=y$, and
\item $U \ff (\xbar \iin\ybar)$ if and only if $x\prec y$.
\end{enumerate}
Thus it remains only to verify the analogue of
\autoref{thm:fourman-apg2}.

\begin{lem}
  Let $U\in\bS$, and let $X$ be a well-founded extensional \apg\ in
  $\bS/U$ which remains well-founded in $\bC/U$.  Then there exists a
  unique $x\colon U\to \sV$ such that $X\cong \xbar$.
\end{lem}
\begin{proof}
  This is essentially the standard recursion theorem for well-founded
  relations.  Since \sV\ is an indexed initial algebra, it suffices to
  assume that $U=1$.  Define an \emph{attempt} to be a partial
  function $X\rightharpoonup \sV$ which is a simulation and whose
  domain is a \emph{small} initial segment of $X$.  We can then use
  the internal logic of \bC\ to define ``the set of all attempts'' as
  a subobject of $P_s(X\times \sV)$.  Since $X$ is well-founded, we
  can then prove by induction, as usual, that any two attempts agree
  on their common domain, that the union of attempts is an attempt,
  and that the union of all attempts has domain all of $X$.  The
  resulting simulation $X\too[f] \sV$ then exhibits $X$ as isomorphic
  to $\xbar$, where $x$ is the composite $1\too[\star] X \too[f] \sV$,
  and it is unique since $\sV$ is extensional.
\end{proof}

However, without further axioms, there is no reason why every
well-founded extensional \apg\ in $\bS/U$ should also be well-founded
in $\bC/U$.  This will be true if \bC\ satisfies the axiom of
separation, but in general there might be new subobjects of $X$ in
$\bC/U$ that are not in $\bS/U$, so well-foundedness in \bC\ is a
stronger statement than well-foundedness in \bS.  If we do have this
stronger property, though, then we can derive the desired equivalence.

\begin{thm}\label{thm:cofc-interp}
  Let \bC\ be a category of classes which is well-generated by small
  objects and has a good material model, with \bS\ its full
  subcategory of small objects, and assume that for all $U\in\bS$,
  every well-founded extensional graph in $\bS/U$ is also well-founded
  in \bC.  Then the interpretation of material set theory in the stack
  semantics of \bS\ is identical to its interpretation relative to
  \sV\ in the internal logic of \bC.
\end{thm}
\begin{proof}
  We have seen that under the given hypotheses, isomorphism classes of
  well-founded extensional \apgs\ in $\bS/U$ can be identified with
  maps $U\to \sV$, in a way which preserves the atomic formulas.  The
  same inductive argument from \autoref{thm:cofc-logic} then implies
  that the interpretations agree on all formulas.
\end{proof}

\begin{cor}
  If \bC\ is a category of classes which is well-generated by small
  objects and has a good material model, and moreover satisfies the
  separation axiom, then the interpretation of material set theory in
  the stack semantics of its category of small objects is identical to
  its interpretation relative to \sV\ in the internal logic of
  \bC.\qed
\end{cor}

Of course, asking that well-foundedness in \bS\ extends to \bC\ is a
version of the axiom of ``full well-founded induction'' relative to
the containing category of classes \bC.  But even if \bS\ satisfies
well-founded induction (in its stack semantics), this property will
generally fail to be inherited by categories of sheaves or ideals.
This is the same problem we observed after
\autoref{thm:sep-aut},
and it has the same solution.  Namely, we will show
in~\cite{shulman:2cofc} that the 2-category of definable stacks on
\bS\ does inherit full well-founded induction from \bS, and that the
stack semantics of \bS\ can be identified with the internal language
of this 2-category.  Thus, if we are willing to generalize from
categories of classes to ``2-categories of classes,'' stack semantics
and algebraic set theory can be seen as two faces of the same coin.

\appendix
\section{Material and structural foundations}
\label{sec:mat-vs-struc}

As mentioned in the introduction, we prefer to view the correspondence
between set theories and (pre)toposes as a relationship between two
different kinds of set theory, which we call \emph{material} and
\emph{structural}.  In material set theories such as ZFC, the elements
of a set $X$ have an independent reality and identity, apart from
their being collected together as the elements of $X$; frequently they
are also sets themselves.  The name ``material'' is a suggestion of
Steve Awodey~\cite{awodey:struct-mathlog}; they are also called
``membership-based'' set theories.  By contrast, in structural set
theories such as ETCS, the elements of a set $X$ have no existence or
identity independent of $X$, and in particular are not sets
themselves; they are merely abstract ``elements'' with which we build
mathematical structures.  We call these theories ``structural''
because they are closely aligned with the mathematical philosophy of
``structuralism'' (see e.g.~\cite{mclarty:numbers,awodey:struct});
they are also called ``categorical'' set theories.

\begin{rmk}
  In structural set theory, it is generally not meaningful to ask
  whether two elements of two different sets are equal,
  only whether two elements of a given ambient set are equal.  In this
  way, structural set theory is very similar to type theory, and a
  case can be made that they are really different names for
  the same thing.  There are undeniable differences in
  presentation and emphasis, but there is certainly a large space of
  theories which it is difficult to classify as one and not the other.
  In particular, as we remarked in footnote~\ref{fn:2} on
  page~\pageref{fn:2}, the stack semantics could equally well be
  phrased using a type theory with quantification over types.
\end{rmk}

Material set theory, of course, has a long history as a foundation
for mathematics, while structural set theory has not gained as broad a
following, although Lawvere first wrote down the axioms of ETCS nearly
50 years ago.  According to its proponents, structural set theory
provides a foundation for mathematics which is just as adequate as the
foundation provided by material set theory, and which is moreover
closer to mathematical practice and free of superfluous data.

In support of this thesis, we observe that in most of mathematics,
sets serve merely as carriers of structure: group structure, ring
structure, topological structure, etc.  This view of mathematics was
formally enshrined in Bourbaki's ``general theory of
structure''~\cite{bourbaki:sets}.  Moreover, these structures, and
hence the sets from which they are built, are generally only ever
considered up to isomorphism.  It never matters to the working
mathematician whether the natural numbers are defined in von Neumann's
way:
\[ 0 = \emptyset,\; 1 = \{0\},\; 2 = \{0,1\},\; 3 = \{0,1,2\},\; \dots
\]
or in some other way, such as:
\[ 0 = \emptyset,\; 1 = \{0\},\; 2 = \{1\},\; 3 = \{2\},\; \dots
\]
since the two resulting versions of $\mathbb{N}$ are canonically
isomorphic.  In particular, the intricate global membership structure
of material set theory is irrelevant outside of set theory proper.
Indeed, Bourbaki had to carefully specify exactly how this superfluous
data is to be forgotten, in order that all of their structures would
be invariant under isomorphism.  In structural set theory this process
of careful forgetting is unnecessary, since there is no global $\in$
in the first place.  (The complexity of the Cole-Mitchell-Osius
construction summarized in \S\ref{sec:constr-mat} testifies to the
amount of data that must be forgotten.)

From this point of view, the value of an equivalence between material
and structural set theory is that it implies that in principle, either
could serve equally well as a foundation for mathematics.  Of course,
what really matters is the naturalness of the encoding---structural
set theory would not be a very attractive foundation if mathematics
could only be encoded in it by way of material set theory!  However,
the observations above about mathematical structure provide good
empirical evidence that in fact, the situation is reversed: material
set theory provides a foundation for most mathematics only by way of
structural set theory.

On the other hand, the material point of view is not without
advantages, especially in set theory proper.  Although the entire
study of ZF could be reinterpreted structurally in terms of
well-founded extensional relations, the result would be unnecessarily
complicated.  In particular, some of the powerful methods of material set theory seem
difficult to duplicate structurally without significant
circumlocution.  This seems to be essentially the point expressed by
Mathias~\cite{mathias:str-maclane}:
\begin{quotation}
  If one wants to do geometry, why bother with von Neumann ordinals?
  What do they do that is not achieved by arbitrary well-orderings?
  Perhaps nothing\dots

  But if, sated with geometry, one wants to do transfinite recursion
  theory, why make life hard by avoiding von Neumann ordinals?
\end{quotation}

Thus, we believe that there should be no real conflict between the two
kinds of set theory.  Rather, as frequently happens in mathematics, we
simply have two different viewpoints on the same underlying reality.
And as always in such a situation, recognition of this means that we
can switch between the two theories freely, using whichever is most
convenient for any purpose.  From this perspective, our goal in this
paper has been to extend the language and tools of structural set
theory so as to deal with unbounded quantifiers directly, facilitating
more direct comparisons with the corresponding language and tools of
material set theory.  The extension we have arrived at (the stack
semantics) also of course has potential applications in topos theory
by itself, which we hope to explore in later work.

\bibliographystyle{halpha}
\bibliography{all,shulman,extra}

\end{document}